\documentclass{amsart}

\usepackage{amsmath}
\usepackage{amsfonts}
\usepackage{amssymb}

\makeatletter
\def\@cite#1#2{{\m@th\upshape\bfseries%
[{#1\if@tempswa{\m@th\upshape\mdseries, #2}\fi}]}}
\makeatother

\numberwithin{equation}{section}

\newtheorem{theorem}{Theorem}[section]
\newtheorem{lemma}[theorem]{Lemma}
\newtheorem{proposition}[theorem]{Proposition}

\newtheorem{definition}[theorem]{Definition}
\newtheorem{corollary}[theorem]{Corollary}

\theoremstyle{definition}
\newtheorem{example}[theorem]{Example}
\newtheorem{remark}[theorem]{Remark}

\newcommand{\be}{\begin{equation}}
\newcommand{\ee}{\end{equation}}
\newcommand{\bes}{\begin{equation*}}
\newcommand{\ees}{\end{equation*}}

\newcommand{\cH}{\mathcal{H}}
\newcommand{\cK}{\mathcal{K}}
\newcommand{\cL}{\mathcal{L}}
\newcommand{\cB}{\mathcal{B}}
\newcommand{\cM}{\mathcal{M}}

\newcommand{\cF}{\mathcal{F}}

\newcommand{\cA}{\mathcal{A}}

\newcommand{\cO}{\mathcal{O}}

\newcommand{\cT}{\mathcal{T}}
\newcommand{\cU}{\mathcal{U}}

\newcommand{\lel}{\left\langle}
\newcommand{\rir}{\right\rangle}

\newcommand{\mb}[1]{\mathbb{#1}}

\newcommand{\ad}{\operatorname{ad}}
\newcommand{\alg}{\operatorname{Alg}}
\newcommand{\Aut}{\operatorname{Aut}}
\newcommand{\diag}{\operatorname{diag}}
\newcommand{\dist}{\operatorname{dist}}
\newcommand{\Int}{\operatorname{int}}
\newcommand{\Mult}{\operatorname{Mult}}
\newcommand{\Sing}{\operatorname{Sing}}
\newcommand{\spn}{\operatorname{span}}
\newcommand{\id}{\operatorname{id}}
\newcommand{\lip}{\langle}
\newcommand{\rip}{\rangle}
\newcommand{\ip}[1]{\lip #1 \rip}

\renewcommand{\phi}{\varphi}

\newcommand{\wot}{\textsc{wot}}
\newcommand{\ol}{\overline}
\newenvironment{sbmatrix}{\left[\begin{smallmatrix}}{\end{smallmatrix}\right]}
\newcommand{\ca}{\mathrm{C}^*}

\begin{document}

\title[Isomorphism of operator algebras]%
{The isomorphism problem for some\\ universal operator algebras}

\author[K.R. Davidson]{Kenneth R. Davidson}
\thanks{First author partially supported by an NSERC grant.}
\address{Pure Mathematics Dept.\\
University of Waterloo\\
Waterloo, ON\; N2L--3G1\\
CANADA}
\email{krdavids@uwaterloo.ca}

\author[C. Ramsey]{Christopher Ramsey}
\email{ciramsey@uwaterloo.ca}

\author[O.M. Shalit]{Orr Moshe Shalit}
\email{oshalit@uwaterloo.ca}

\begin{abstract}
This paper addresses the isomorphism problem for the universal 
(nonself-adjoint) operator algebras 
generated by a row contraction subject to homogeneous polynomial relations. 
We find that two such algebras are isometrically isomorphic if and only if 
the defining polynomial relations are the same up to a unitary change of variables, 
and that this happens if and only if the associated subproduct systems are isomorphic. 
The proof makes use of the complex analytic structure of the character space,
together with some recent results on subproduct systems. 
Restricting attention to commutative operator algebras defined by a radical ideal of relations 
yields strong resemblances with classical algebraic geometry. 
These commutative operator algebras turn out to be algebras of analytic functions 
on algebraic varieties. 
We prove a projective Nullstellensatz connecting closed ideals 
and their zero sets. 
Under some technical assumptions, we find that
two such algebras are isomorphic as algebras if and only if they are similar, 
and we obtain a clear geometrical picture of when this happens. 
This result is obtained with tools from algebraic geometry, 
reproducing kernel Hilbert spaces, and some new complex-geometric 
rigidity results of independent interest. 
The C*-envelopes of these algebras are also determined. 
The Banach-algebraic and the algebraic classification results are shown
 to hold for the \wot-closures of these algebras as well.
\end{abstract}

\maketitle

\section{Introduction}

A fundamental problem is: \emph{given polynomials 
$p_1, \ldots, p_k$ in $\mb{C}[z_1, \ldots, z_d]$, 
find all solutions to the system of equations}
\be\label{eq:equations}
p_i(x_1, \ldots, x_d) = 0 \quad, \quad i=1, \ldots, k.
\ee
When the indeterminates $x_i$ are understood to be complex numbers, 
the solution set is a complex variety, and this is the starting point 
of complex algebraic geometry. 
This problem makes sense in operator theory, where the indeterminants
are bounded linear operators on Hilbert space.
We consider both the case of arbitrary operators and polynomials 
in $d$ non-commuting variables, and the case of $d$ commuting operators
and polynomials in commuting variables.
The issue we study is the isomorphism problem for the universal 
(nonself-adjoint) operator algebra determined
by the solutions.
In some sense, we are attempting to develop non-commutative complex 
algebraic geometry in this context.
Ideas from classical algebraic geometry are an important influence on our development.

Let us first consider the case in which there are no relations. 
In the setting of multi-variable operator theory, 
(\ref{eq:equations}) has a \emph{universal} solution 
if one adds a reasonable norm constraint. 
The algebra that arises is Popescu's non-commutative disk algebra \cite{Popescu91}.
In the abelian case (the relations $x_ix_j - x_jx_i = 0$ for $1 \le i < j \le d$),
one obtains Arveson's algebra \cite{Arv98} of multipliers on
symmetric Fock space; and this is realized as a continuous multipliers
on a reproducing kernel Hilbert space of functions.

When one imposes a family of (non-commutative) relations,
the universal algebra is realized as a quotient of the
non-commutative disk algebra.
This can be considered as an abstract operator algebra in the sense
of Blecher, Ruan and Sinclair \cite{BRS}. 
However, it has been shown to have an explicit faithful 
representation on a subspace of
Fock space associated to the ideal of relations.
In the abelian case, the algebra is a quotient of the algebra of 
continuous multipliers on symmetric Fock space; and it has a
rather explicit faithful representation as an algebra of multipliers
on a reproducing kernel Hilbert space determined by the zero set
of the relations. 

Let $E$ be a finite dimensional Hilbert space and fix an orthonormal 
basis $e_1, \ldots, e_d$ for $E$. 
Let $L = (L_1, \ldots, L_d)$ be the $d$-shift on the free Fock space 
$\cF(E) := \oplus_{n\geq 0} E^{\otimes n}$, defined by
\bes
L_i e_{\alpha_1} \otimes \cdots \otimes e_{\alpha_n} = 
e_i \otimes e_{\alpha_1} \otimes \cdots \otimes e_{\alpha_n} 
\,\, , \,\, i = 1 \ldots, d .
\ees
By the Bunce-Frazho-Popescu Dilation Theorem \cite{Bunce,Frazho,Popescu89}, 
every pure row contraction $T = (T_1, \ldots, T_d)$ is the compression of 
$L^{(\infty)}$ (a direct sum of infinitely many copies of $L$) to a coinvariant subspace. 
In fact, the normed closed algebra 
$\mathfrak{A}_d = \overline{\alg }\{I,L_1, \ldots, L_d\}$ 
is the universal operator algebra generated by a row contraction \cite{Popescu91}.  
That is, for every row contraction $T = (T_1, \ldots, T_d)$, there is a 
unital, completely contractive, surjective homomorphism 
$\varphi : \mathfrak{A}_d \to \overline{\text{alg}}\{I,T_1, \ldots,T_d\}$ 
sending $L_i$ to $T_i$. 
So $L$ can be considered as the universal (row contractive) solution 
to (\ref{eq:equations}) when there are no relations.

The existence of a universal solution for no relations allows us to exhibit a natural construction of 
a universal solution to (\ref{eq:equations}) when $p_1, \ldots, p_k$ 
generate a nontrivial ideal $I$ (in the algebra 
$\mb{C}\lel z_1, \ldots, z_d\rir$ of polynomials in 
$d$ non-commuting variables with complex coefficients). 
Let $\tilde{I}$ be the norm closed ideal in $\mathfrak{A}_d$ 
generated by the set $\{p(L) : p \in I\}$. 
Then the quotient $\cA_I := \mathfrak{A}_d / \tilde{I}$ is the universal 
operator algebra generated by a row contraction subject to relations 
(\ref{eq:equations}), and the images of $L_1, \ldots, L_d$ 
constitute a universal solution. 
Several researchers noticed over the years that $\cA_I$ can be 
naturally identified with the compression of $\mathfrak{A}_d$ to 
the coinvariant subspace $\cF_I := \cF(E) \ominus [\tilde{I} \cF(E)]$ 
(see, in increasing order of generality, \cite{Arv98,BhatBhat,ShalitSolel}, 
and \cite{DavPittsPick, Popescu06}). 
The $d$-tuple $L^I = (L^I_1, \ldots, L^I_d)$ obtained by compressing 
$L$ to $\cF_I$ is a universal solution of (\ref{eq:equations}), 
and every pure row contraction that satisfies (\ref{eq:equations}) 
is a compression of $L^I$ to a coinvariant subspace. 
The variety of (row contractive) solutions of (\ref{eq:equations}) is in one-to-one 
correspondence with the unital completely contractive representations of $\cA_I$. 

A different, yet closely related, route which leads to these operator algebras 
is via subproduct systems. 
A benefit of this route is that it is ``coordinate free". 
A \emph{subproduct system} is a family $X = \{X(n)\}_{n \in \mb{N}}$ 
of Hilbert spaces satisfying 
\bes
X(m+n) \subseteq X(m) \otimes X(n) \, \, , \,\, m,n \in \mb{N},
\ees
and $X(0) = \mb{C}$. 
These objects were introduced in \cite{ShalitSolel} as a framework 
for the dilation theory of $cp$-semigroups; independently, 
they appeared in \cite{BhatMukherjee} under the name \emph{inclusion systems}, 
to facilitate computations in amalgamated product systems. 
Every subproduct system naturally gives rise to an operator algebra 
$\cA_X$ acting on the space $\cF_X := \oplus_{n \geq 0}X(n)$. 
The isometric isomorphism class of $\cA_X$ is an invariant of $X$.
Whether or not it is a complete invariant was a question left open 
in \cite{ShalitSolel} which we resolve in the affirmative here. 
When these algebras were introduced there was some hope%
\footnote{In the heart of the less experienced author.} 
that they will shed light on the subproduct systems that gave rise to them. 
But it turned out that the structure of the subproduct systems is easier to understand. 
Luckily, it was also noticed that there is a bijection between 
subproduct systems and ideals (in $\mb{C}\lel z_1, \ldots, z_d\rir$),
\bes
X \longleftrightarrow I^X
\ees
and that $\cA_X = \cA_{I^X}$. 
This gave rise to a different conceptual point of view by which to 
consider the universal operator algebras discussed above. 

The main result of this paper is the classification of the algebras $\cA_X$. 
In the general case the classification is up to (completely) isometric isomorphism; 
in the commutative case, when the ideal of relations $I$ is radical, we classify both up to (completely) isometric 
isomorphism and up to algebraic isomorphism---this under some reasonable 
technical assumptions on the geometry of the affine algebraic variety associated with the ideal of relations $I$.
In the latter case, it is shown that the geometry of the affine algebraic variety determines the algebraic and isometric structures of the algebra.

\medbreak
In more detail, the contents of this paper are as follows.

The notation is set up in Section \ref{sec:def}. 
Among other things the correspondence between subproduct systems 
and ideals is explained. 
Some examples and motivation are given in Section \ref{sec:mot}, 
and it is shown that two subproduct systems $X$ and $Y$ are isomorphic 
if and only if the corresponding ideals $I^X$ and $I^Y$ can be obtained,
one from the other, by unitary change of variables (Proposition \ref{prop:idealsps}).
Section \ref{sec:class} contains an analysis of the character spaces 
of the algebras $\cA_X$, and it is shown that these can be identified 
with a homogeneous algebraic variety intersected with the unit ball. 
Further, it is shown that the character spaces have a complex analytic 
structure that is preserved under isometric isomorphisms. 
 From this we infer that the existence of an isometric isomorphism from 
$\cA_X$ onto $\cA_Y$ implies the existence of a
\emph{vacuum preserving} isometric isomorphism 
(Proposition \ref{prop:exist_vacuum}). 
A result from \cite{ShalitSolel} then applies to give our first 
classification result, Theorem \ref{thm:alg_sps_iso}, 
that says that $\cA_X$ is isometrically isomorphic to $\cA_Y$ 
if and only if $X$ is isomorphic to $Y$ 
(and then $\cA_X$ and $\cA_Y$ are, in fact, unitarily equivalent).

{}From this point onward we concentrate on the commutative case
(so the relations in (\ref{eq:equations}) include all relations 
$x_i x_j = x_j x_i$, $i=1, \ldots, d$).
Moreover, we assume that the ideal $I^X$ is radical.
In Section \ref{sec:mult} a connection is made to the theory of 
reproducing kernel Hilbert spaces.  
It is shown that $\cA_X$ is an algebra of multipliers, and, in particular, 
an algebra of functions. 
In Section \ref{sec:Null} we consider some natural questions in a 
wide class of algebras of functions and prove a Nullstellensatz for 
closed homogeneous ideals (Theorem \ref{thm:homNull}). 
A direct corollary (Corollary \ref{cor:approxpoly})  is that in these algebras, 
any function that vanishes on a homogeneous algebraic variety can be 
approximated in the norm by polynomials vanishing on that variety.

Sections \ref{sec:isomorphism} and \ref{sec:class2} are the main course, 
with most of the hard work in the former, and the main results in the latter. 
The first result in Section \ref{sec:isomorphism} is that a unital isomorphism 
from $\cA_I$ to $\cA_J$ induces a holomorphic mapping between the character spaces. 
The rest of the section is therefore devoted to studying mappings between  
homogeneous algebraic varieties. 
Some complex-geometric rigidity results of independent interest are 
obtained (Theorem \ref{thm:linear} and Propositions \ref{prop:rigid} 
and \ref{prop:eigenvalueanalysis}). 
We then turn to prove that, given two homogeneous ideals $I$ and $J$, 
every invertible linear map between the varieties $V(I)$ and $V(J)$ 
that is length preserving on the varieties, gives rise to an isomorphism 
of the corresponding algebras $\cA_I$ and $\cA_J$ 
(Theorem \ref{thm:linear_induce_similarity}). 
We are able to prove this only when the varieties are what we call 
{\em tractable}, which just means that their geometry is not too complicated. 
The precise definition of a tractable variety is given before 
Theorem \ref{thm:Atilde}, but let us mention now that many interesting 
varieties are tractable, for example: irreducible varieties, 
varieties with two irreducible components, varieties of codimension $1$ 
and varieties in $\mb{C}^3$. 
Algebraically, this means that our methods work for, e.g., principal ideals, 
prime ideals and in three variables.

In Section \ref{sec:class2} we sum up all that we obtained to give the 
classification (in the commutative case) of the algebras $\cA_I$ when $I$ is radical.  
Theorem \ref{thm:iso_iso_alg} says that $\cA_I$ is isometrically isomorphic 
to $\cA_J$ if and only if there is a unitary transformation mapping the 
algebraic variety $V(I)$ onto $V(J)$. 
Theorem \ref{thm:algiso_lin} says that, when $V(I)$ and $V(J)$ are 
tractable, then $\cA_I$ is isomorphic to $\cA_J$ 
if and only if there is a linear map, that is length preserving on $V(I)$, 
that maps $V(I)$ onto $V(J)$ (and then the two algebras are, in fact, similar).
Using the geometric rigidity results Propositions \ref{prop:rigid} and 
\ref{prop:eigenvalueanalysis}, this implies 
an operator-algebraic rigidity result: if $I$ is prime or principal and $\cA_I$ is isomorphic 
(as an \emph{algebra}) to $\cA_J$, then $\cA_I$ is unitarily equivalent to $\cA_J$.
 
Section \ref{sec:aut} closes our treatment of the algebras $\cA_I$ with a study 
of the automorphism groups of these algebras. 
Theorem \ref{thm:autoAd} establishes a one-to-one correspondence between 
the isomorphisms of $\cA_d$ (which is the universal operator algebra generated 
by a commuting row contraction) and the automorphism group of the 
unit ball in $\mb{C}^d$. 
We then turn to study when an automorphism of $\cA_I$ is induced by 
an automorphism of $\cA_d$, and we find the automorphism group of the 
algebras corresponding to a union of subspaces.

In Section \ref{sec:Toeplitz} we look at the ``Toeplitz" C*-algebras 
$\cT_X = \ca(\cA_X)$. 
We find that, in the commutative case, $\cT_X$ is the C*-envelope 
of $\cA_X$, and this allows us to deduce that all completely isometric 
isomorphisms between such algebras are unitarily implemented.
We also bring some evidence for a connection between the $*$-algebraic 
structure of $\cT_X$ and the topology of the variety $V(I^X)$.

In the final section we treat the algebras obtained by taking the closure 
of the algebras $\cA_X$ in the weak-operator topology. 
We find that the algebraic and the Banach-algebraic classification 
remains unchanged, as well as the algebraic rigidity. 
We also show that in the radical commutative case every isomorphism 
is automatically bounded and continuous in the weak-operator and weak-$*$ topologies.

\section{Definitions and notation}\label{sec:def}

\subsection{A word of explanation about notation}

In this paper we are concerned with two classes of operator algebras. 
The first class consists of universal operator algebras generated by a contractive row of operators subject to noncommutative homogeneous polynomial 
relations, and our objective is to classify these algebras up to isometric isomorphism
(we will find that when two such algebras are isometrically isomorphic,
then they are also completely isometrically isomorphic). 
The second class consists of universal operator algebras generated by a
contractive row of \emph{commuting} operators subject to 
(commutative) homogeneous polynomial relations, and our objective is to classify these algebras up to isometric isomorphism as well as up to (algebraic) isomorphism. 
Let us call the first class \emph{the noncommutative case} and 
the second class \emph{the commutative case}.

In this section we set up the notational framework for the paper. 
The commutative case is contained in the 
noncommutative case (we are simply adding the relations $z_i z_j = z_j z_i$), 
so in principle we can set up notation for the noncommutative case and use it 
consistently for the commutative case as well. 
However, since most of our attention will be directed towards the 
commutative case, and since it is natural to do so, 
we will set up a notational framework for the commutative case also. 
This will cause notational inconsistencies, but no confusion.

\subsection{The noncommutative case}

In this paper,  a \emph{subproduct system} is a collection 
$X = \{X(n)\}_{n \in \mb{N}}$ of finite dimensional Hilbert spaces 
that satisfy $X(0) = \mb{C}$ and $X(m+n) \subseteq X(m) \otimes X(n)$.
Subproduct systems were introduced and studied in greater generality in \cite{ShalitSolel}.

Given a subproduct system $X$, let $E = X(1)$. 
Then $X(n) \subseteq E^{\otimes n}$. Write $p^X_n$ for the projections
$p^X_n : E^{n} \rightarrow X(n)$. 
Then $X$ has an associative multiplication that extends to tensor products 
given by product maps $U_{m,n}^X : X(m) \otimes X(n) \rightarrow X(m+n)$, 
\bes
U^X_{m,n} (x \otimes y)  = p^X_{m+n}(x \otimes y).
\ees

We define the $X$-Fock space, denoted $\cF_X$, to be  
$\cF_X := \oplus_{n\geq 0} X(n)$. 
If $E = X(1)$, then $\cF_X$ is a subspace of the full Fock space 
$\cF(E) := \oplus_{n\geq 0} E^{\otimes n}$. 
The symbol $\Omega_X$ will denote the \emph{vacuum vector} 
$\Omega_X = 1 \in X(0) \subseteq \cF_X$ of $\cF_X$.

Now fix an orthonormal basis $\{e_1, \ldots, e_d\}$ for $E$. 
Let $\mb{C}\lel z_1, \ldots, z_d\rir$ be the algebra of polynomials in 
$d$ noncommuting variables with complex coefficients. 
When $d$ is understood, we simply write $\mb{C}\lel z \rir$. 
If $p$ is a polynomial in $\mb{C}\lel z\rir$, we write $p(e)$ or $p$ for 
the element of $\cF(E)$ given by ``evaluating" $p$ at $e_1, \ldots, e_d$. 
For example, if $p(z) = z_1 z_2 - z_3 z_1 z_3$, 
then $p(e) = e_1 \otimes e_2 - e_3 \otimes e_1 \otimes e_3$. 

There is a natural bijection between homogeneous ideals in
$\mb{C}\lel z \rir$ and subproduct systems $X$ with $X(1) \subseteq E$ 
(after fixing an orthonormal basis $\{e_1, \ldots, e_d\}$ for $E$). 
If $X$ is a subproduct system, we denote the associated ideal by 
$I^X$, and if $I$ is a homogeneous ideal, we denote the associated 
subproduct system by $X_I$. 
The relation between $X$ and $I^X$ is the following:
\be\label{eq:idealspsbijection}
I^X=\spn\{p : p(e) \in E^{\otimes n} \ominus X(n) \textrm{ for some } n  \}.
\ee
See \cite[Section 7]{ShalitSolel} for details.

On $\cF(E)$ there are the natural left creation operators $L_1, \ldots, L_d$, given by 
\bes
L_i (e_{\alpha_1} \otimes \cdots \otimes e_{\alpha_n}) = 
e_i \otimes e_{\alpha_1} \otimes \cdots \otimes e_{\alpha_n} 
\,\, , \,\, i = 1 \ldots, d .
\ees
Let $S_1^X, \ldots, S_d^X$ denote their compression to $\cF_X$. 

We define $\cA_X$ to be the norm closed operator algebra generated 
by $I, S_1^X, \ldots, S_d^X$. 
This is the main object of study in this paper. 
Recall that $\cA_X$ is equal to the universal norm closed unital 
operator algebra generated by a row contraction subject to the 
relations in $I^X$ (see Section 8 in \cite{ShalitSolel} for details). 
We also define $\cT_X := \ca(\cA_X)$ and 
$\cO_X = \cT_X/\cK(\cF_X)$, where $\cK(\cF_X)$ is the algebra of 
compact operators on $\cF_X$.

In \cite{Viselter}, following terminology from \cite{MS98}, the algebra 
$\cA_X$ was denoted $\cT_+(X)$ and called \emph{the tensor algebra 
of $X$}, and the algebra $\cT_X$ was denoted $\cT(X)$ and called 
\emph{the Toeplitz algebra of $X$}. 
We shall also refer to $\cT_X$, sometimes, as the Toeplitz algebra of $X$.

There is another way to obtain the algebra $\cA_X$. 
Let $\mathfrak{A}_d$ be the noncommutative disc algebra, that is, 
the norm closed algebra generated by $I, L_1, \ldots, L_d$. 
By \cite[Theorem 3.9]{Popescu91}, $\mathfrak{A}_d$ is the universal 
unital operator algebra generated by a row contraction. 
If $\tilde{I}$ is the ideal in $\mathfrak{A}_d$ generated  by 
$\{p(L_1, \ldots, L_d) : p \in I^X\}$, then the quotient 
$\mathfrak{A}_d / \tilde{I}$ is also the universal unital operator algebra 
generated by a row contraction subject to the relations in $I^X$, 
thus it is completely isometrically isomorphic to $\cA_X$ \cite{Popescu06}.

Let $\cL_d$ be the noncommutative analytic Toeplitz algebra, 
that is, closure of of $\mathfrak{A}_d$ in the weak-operator topology (\wot).  
We also denote by $\cL_X$ the \wot-closure of $\cA_X$.

\subsection{The commutative case}
When focusing on the commutative case it will be more natural to use 
the following framework. 

Let $E$ be a Hilbert space of dimension $d$. 
Denote by $E^n$ the symmetric tensor product of $E$ with itself $n$ times. 
For $x_1, x_2, \ldots, x_n \in E$, we write $x_1 x_2 \cdots x_n$ 
for their symmetric product in $E^n$. 
The family $\{E^n\}_{n \geq 0}$ forms a subproduct system in which 
the product is just the symmetric product.
Briefly, the commutative case is the case in which we take $X$ 
to be a subproduct subsystem of the symmetric subproduct system 
$\{E^{n}\}_{n \in \mb{N}}$. 
Such a subproduct system will be referred to below as a 
\emph{commutative subproduct system}, and note that multiplication 
in these subproduct systems is commutative.

In more detail, the notation for the commutative case will be almost 
the same as for the noncommutative case described above, but with 
the following adjustments made.

We replace the algebra $\mb{C}\lel z \rir$ with the algebra 
$\mb{C}[z_1, \ldots, z_d]$ of complex polynomials in $d$ (complex) variables. 
Again, when $d$ is understood, we write $\mb{C}[z]$. 
Also, we replace the full Fock space by the symmetric Fock space, 
also known as Drury-Arveson space, which we denote by $H^2_d$ (see \cite{Arv98}).

As in the noncommutative case, once we fix an orthonormal basis 
$\{e_1, \ldots, e_d\}$ for $E$, there is a natural bijection between 
homogeneous ideals in $\mb{C}[z]$  and commutative subproduct 
systems $X$ with $X(1) \subseteq E$. 
If $X$ is a subproduct system, we denote the associated ideal by $I^X$, 
and if $I$ is a homogeneous ideal, we denote the associated subproduct 
system by $X_I$. 
The relation between $X$ and $I^X$ is the following:
\bes
I^X=\spn\{p : p(e) \in E^n \ominus X(n) \textrm{ for some } n  \}.
\ees
Note that we are using the same notation, but now $I^X$ is understood 
to be an ideal in $\mb{C}[z]$.
Here and below, when given a polynomial 
$p(z) = p(z_1, \ldots, z_d) = \sum c_{i_1\cdots i_d} z_1^{i_1} \cdots z_d^{i_d}$,
we will write $p(e) = p(e_1, \ldots, e_d)$ for the element in the symmetric 
Fock space given by $\sum c_{i_1 \cdots i_d} e_1^{i_1} \cdots e_d^{i_d}$. 
For a multi-index $\alpha = (\alpha_1, \ldots, \alpha_d)$,
we will write $e^\alpha$ for the polynomial 
$z^\alpha = z_1^{\alpha_1} \cdots z_d^{\alpha_d}$ evaluated at $e$.
Let $Z_1, \ldots, Z_d$ denote the coordinate functions on $H^2_d$. 
Then $Z_i$ is the compression of $L_i$ to $H^2_d$, and 
$S_1^X, \ldots, S_d^X$ are also the compressions of the $Z_i$ to $\cF_X$. 

We denote by $\cA_d$  the norm closed algebra generated by $I, Z_1, \ldots, Z_d$. 
By \cite[Theorem 6.2]{Arv98}, (and also by the discussion in the 
previous subsection), $\cA_d$ is the universal unital operator algebra 
generated by a commuting row contraction.
If $\tilde{I}$ is the ideal in $\cA_d$ generated  by 
$\{p(Z_1, \ldots, Z_d) : p \in I^X\}$, then the quotient 
$\cA_d / \tilde{I}$ is completely isometrically isomorphic to $\cA_X$.

In the commutative case (and in that case only), when $I = I^X$, 
then we will also write $\cA_I$ instead of $\cA_X$. We will also write $\cL_I$ for $\cL_X$.

\subsection{Ideals and zero sets}
If $I$ is an ideal in $\mb{C}[z]$ or in $\mb{C}\lel z \rir$, we let
\bes
V(I) = \{z \in \mb{C}^d :  p(z) = 0 \text{ for all } p \in I\}.
\ees
When $I$ is an ideal of polynomials in noncommutative variables, 
there is still a well defined notion of $p(z)$ for 
$z = (z_1, \ldots, z_d) \in \mb{C}^d$. 
In both the commutative and noncommutative cases the set $V(I)$ 
is an (affine) algebraic variety in $\mb{C}^d$. 
Throughout the paper we will use some well known results and 
terminology from algebraic geometry. 

An ideal $I \subseteq \mb{C}[z]$ is said to be {\em radical} if
\[
I = \sqrt{I} := \{p \in \mb{C}[z] : \exists n . p^n \in I\}.
\]

In algebraic geometry it is natural to associate to a homogeneous ideal 
a \emph{projective} variety (rather than an affine variety), 
but we do not do so for reasons that will become clear. 
The decisive role will be played by the sets
\bes
Z(I) = V(I) \cap \overline{\mb{B}}_d
\ees
and
\bes
Z^o(I) = V(I) \cap \mb{B}_d,
\ees
where $\mb{B}_d$ is the unit ball of $\mb{C}^d$.
The set of singular points of a variety $V$ will be denoted $\Sing (V)$.

\section{Motivation and examples}\label{sec:mot}

Two subproduct systems $X$ and $Y$ are said to be isomorphic, 
written $X \cong Y$, if there is a family $W = \{W_n\}_n$ of unitaries 
$W_n : X(n) \rightarrow Y(n)$ such that for all $m,n$,
\be\label{eq:iso}
W_{m+n} \circ U^X_{m,n} = U^Y_{m,n}\circ(W_m \otimes W_n).
\ee
It is clear that if $X \cong Y$ then $\cA_X$ is completely isometrically
 isomorphic to $\cA_Y$, because then the map
\bes
V:= \oplus_{n=0}^\infty W_n : \cF_X \rightarrow \cF_Y
\ees
is a unitary that gives rise to a completely isometric isomorphism 
$\varphi :\cA_X \rightarrow \cA_Y$ by 
\bes
\varphi(a) = VaV^* \,\, , \,\, a \in \cA_X.
\ees
Answering the converse question, \emph{``if $\cA_X$ is isometrically 
isomorphic to $\cA_Y$, does it follow that $X \cong Y$?"}, is our main 
objective in this section and the next. 
In \cite{ShalitSolel} it was verified within several special classes of 
subproduct systems that the answer is yes. 
In the next section we will show that the answer is yes in general.

Let us indicate why the above problem---classifying the algebras 
$\cA_X$ in terms of the subproduct systems $X$---is interesting.
First, the subproduct systems give a concrete and easily computable 
handle to the more complicated category of operator algebras. 
In the last few sections of \cite{ShalitSolel} several examples are given 
where it was possible to effectively distinguish between naturally 
defined operator algebras in terms of the associated subproduct systems. 
The second reason is that an isomorphism of 
subproduct systems is ``the same" as a unitary equivalence of the 
associated ideals defining the relations.

\begin{proposition}\label{prop:idealsps}{\bf [Proposition 7.4, \cite{ShalitSolel}]}
Let $X$ and $Y$ be [commutative] subproduct systems with 
$\dim X (1) = \dim Y (1) = d < \infty$. 
Then $X$ is isomorphic to $Y$ if and only if there is a unitary linear 
change of variables in $\mb{C}\lel z \rir$ $\Big[\mb{C}[z]\Big]$ that sends 
$I^X$ onto $I^Y$. 
Moreover, every isomorphism of subproduct systems is induced by a 
unitary linear change of variables, and vice-versa.
\end{proposition}

This theorem was stated in \cite{ShalitSolel} in the noncommutative case.
Since in \cite{ShalitSolel} a proof was not provided, 
we include one for the commutative case. 
A similar proof works in the noncommutative case.

\begin{proof}
Assume that $I^X$ is sent to $I^Y$ when applying a unitary change 
of variables in $\mb{C}[z]$. 
By this we mean that there is a unitary $U$ acting on $\mb{C}^d$ such that
\bes
I^Y = \{ f \circ U : f\in I^X\}.
\ees
We now define an isomorphism $W$ of subproduct systems from 
$X = X_{I^X}$ to $Y = X_{I^Y}$. 
We define a unitary $W_n$ on $E^{n}$ by sending $p(e_1, \ldots, e_d)$ 
(where $p(z_1, \ldots, z_d)$ is a homogeneous polynomial of degree $n$) 
to $p\circ U (e_1, \ldots, e_d) = p(U^t e_1, \ldots, U^t e_d)$. 
The unitary $W_n$ sends $X(n)^\perp$ to $Y(n)^\perp$, 
thus it sends $X(n)$ unitarily onto $Y(n)$. 
The family $W = \{W_n\}$ is an isomorphism of subproduct systems. 
To see this, notice that an arbitrary element of $Y(m+n)$ can be 
written as $\sum_i (p_i\circ U) (e) \otimes (q_i\circ U)(e)$, 
where $(p_i\circ U)(e) \in Y(m)$, and $(q_i\circ U)(e) \in Y(n)$. 
On the one hand, applying to such an element the inclusion map 
$Y(m+n) \rightarrow Y(m) \otimes Y(n)$ followed by $(W_m \otimes W_n)^{-1}$, 
we get the element $\sum_i p_i (e) \otimes q_i (e) \in X(m) \otimes X(n)$. 
On the other hand, applying to $\sum_i (p_i\circ U) (e) \otimes (q_i\circ U)(e)$ 
first $W_{m+n}^{-1}$ and then applying the inclusion 
$X(m+n) \rightarrow X(m) \otimes X(n)$ 
we again get the element $\sum_i p_i (e) \otimes q_i (e) \in X(m) \otimes X(n)$. 
Taking the adjoint of the above argument, we obtain (\ref{eq:iso}).

Conversely, assume that $W: X \rightarrow Y$ is an isomorphism 
of subproduct systems. 
We define a unitary $U = (u_{ij})_{i,j=1}^d$ by the following relations:
\bes
W_1 e_i = \sum_{j=1}^d u_{ij}e_j    \,\, , ,\,\, i=1, \ldots, d.
\ees

Reversing the reasoning above, we find that $U$ sends $I^X$ to $I^Y$.
Here are the details. $W_1$ extends to a unitary 
$\tilde{W}_n: E^{\otimes n} \rightarrow E^{\otimes n}$ by 
\bes
\tilde{W}_n (e_{i1} \otimes \cdots \otimes e_{i_n}) = 
(W_1 e_{i_1}) \otimes \cdots \otimes (W_1 e_{i_n}).
\ees
Because $W$ respects the product, 
\bes
W_n p^X_n (x_1 \otimes \cdots \otimes x_n) = 
p^Y(W_1 x_1 \otimes \cdots \otimes W_1 x_n).
\ees
Thus $W_n p^X_n = p^Y_n \tilde{W}_n$. 
Because $W_n$ is a unitary from $X(n)$ onto $Y(n)$ 
we have $\tilde{W}_n |_{X(n)} = W_n$. 
Thus $p(e) \mapsto p\circ U(e) = p(W_1 e_1, \ldots, W_1 e_d)$ 
sends $X(n)$ to $Y(n)$, and thus it sends $X(n)^\perp$ to $Y(n)^\perp$. 
It follows that $p(z) \mapsto p\circ U(z)$ sends $I^X$ to $I^Y$.
\end{proof}

\begin{remark}
To a reader who is wondering why not forget about subproduct systems 
and classify these algebras using ``equivalence classes" of ideals,
we note, for example, the role of the integer $d$ in the above proposition.
\end{remark}

When the ideal $I^X$ is radical (in the commutative setting) we will show 
below that the geometry of a certain variety determines $\cA_X$. 
However, when $\cA_X$ comes from a non-radical ideal of relations, 
this geometrical classifying object disappears, 
and the subproduct systems is the next best thing.

\begin{example}\label{expl:iso}
Let $I = \langle xy, y^2, x^3 \rangle$ and 
$J = \langle x(x+y), (x+y)^2,x^3 \rangle$ in $\mb{C}[x,y]$. 
There is a unique unital (algebraic) automorphism $\phi$ of 
$\mb{C}[x,y]$ determined by $\phi(x) = x$, $\phi(y) = x+y$.
Clearly, $\phi$ sends $I$ onto $J$, thus it induces an isomorphism of algebras
\bes
\overline{\phi}:\mb{C}[x,y] / I \rightarrow \mb{C}[x,y] / J.
\ees
Now write $X = X_I$ and $Y = X_J$.
Since $\cA_{X}$ and $\cA_{Y}$ are finite dimensional, 
they are the universal commutative unital algebras generated by a pair 
satisfying the relation in $I$ and in $J$, respectively.
Thus $\cA_{X} \cong \mb{C}[x,y]/I \cong \mb{C}[x,y] / J \cong \cA_{Y}$ as algebras.
More is true: $\cA_{X}$ and $\cA_{Y}$ are actually isometrically isomorphic. 

The Fock space $\cF_{X}$ is seen to have an orthonormal basis 
$\{\Omega_X, e_1, e_2, e_1^2\}$. 
In this basis we have
\bes
S_1^X = \begin{pmatrix}
  0 & 0 & 0 & 0  \\
  1 & 0 & 0  & 0  \\
  0 & 0 & 0 & 0 \\ 
  0 & 1 & 0 & 0  \\
\end{pmatrix} \,\, , \,\, 
S_2^X = \begin{pmatrix}
  0 & 0 & 0 & 0  \\
  0 & 0 & 0  & 0  \\
  1 & 0 & 0 & 0 \\ 
  0 & 0 & 0 & 0  \\
\end{pmatrix}.
\ees
It follows that 
\bes
\cA_X = \left\{\begin{pmatrix}
  a & 0 & 0 & 0  \\
  b & a & 0  & 0  \\
  c & 0 & a & 0 \\ 
  d & b & 0 & a  \\
\end{pmatrix} : a,b,c,d \in \mb{C} \right\}.
\ees
Similarly, $\cF_Y$ is seen to have 
$\{\Omega_Y, e_1, e_2, (e_1^2-2e_1e_2+e_2^2)/2 \}$ 
as an orthonormal basis.  (Recall that 
$\|e_1e_2\| = \|(e_1\otimes e_2 + e_2\otimes e_1)/2\| = 1/\sqrt2$.)
So we obtain the shifts
\bes
S_1^Y = \begin{pmatrix}
  0 & 0 & 0 & 0  \\
  1 & 0 & 0  & 0  \\
  0 & 0 & 0 & 0 \\ 
  0 & 1/2 & -1/2 & 0  \\
\end{pmatrix} \,\, , \,\, 
S_2^Y = \begin{pmatrix}
  0 & 0 & 0 & 0  \\
  0 & 0 & 0  & 0  \\
  1 & 0 & 0 & 0 \\ 
  0 & -1/2 & 1/2 & 0  \\
\end{pmatrix},
\ees
and the algebra
\bes
\cA_Y = \left\{\begin{pmatrix}
  a & 0 & 0 & 0  \\
  b & a & 0  & 0  \\
  c & 0 & a & 0 \\ 
  d & (b-c)/2 & (c-b)/2 & a  \\
\end{pmatrix} : a,b,c,d \in \mb{C} \right\}.
\ees
 From this description of the algebras it is not clear that they are isometric. 
But it can be checked that the unitary change of variables 
\[ x \mapsto (x - y) /\sqrt{2} \quad,\quad y \mapsto (x+y) /\sqrt{2} \]
sends $I$ onto $J$.
Thus by Proposition \ref{prop:idealsps} and the discussion before it, 
we conclude that $\cA_X$ and $\cA_Y$ are isometrically isomorphic 
(and, in fact, they are spatially isomorphic). 
It is hard to recognize this because the isometric  isomorphism will not send 
$\{S^X_1,S^X_2\}$ to $\{S^Y_1, S^Y_2\}$. 
\end{example}

The following example shows that if $I$ and $J$ are ideals in 
$\mb{C}[z_1,\ldots, z_d]$ that are related by a linear change of variables, 
then their universal operator algebras may not be isometrically isomorphic.

\begin{example}\label{expl:notiso}
Let $I = \langle xy, y^3, x^3 \rangle$ and $J = \langle x(x+y), y^3,x^3 \rangle$. 
Again, there is a unique unital (algebraic) automorphism $\phi$ of $\mb{C}[x,y]$ 
determined by $\phi(x) = x$, $\phi(y) = x+y$.
Note that $\phi$ sends $I$ onto $J$.
Thus it induces an isomorphism of algebras
\bes
\overline{\phi}:\mb{C}[x,y] / I \rightarrow \mb{C}[x,y] / J.
\ees
Now write $X = X_I$ and $Y = X_J$.
Exactly as above, 
$\cA_{X} \cong \mb{C}[x,y]/I \cong \mb{C}[x,y] / J \cong \cA_{Y}$ 
as algebras.
However, $\cA_{X}$ and $\cA_{Y}$ are not isometrically isomorphic. 

The Fock space $\cF_{X}$ is seen to have an orthonormal basis 
$\{\Omega_X, e_1, e_2, e_1^2, e_2^2 \}$. In this basis we have
\bes
S_1^X = \begin{pmatrix}
  0 & 0 & 0 & 0 & 0  \\
  1 & 0 & 0  & 0  & 0\\
  0 & 0 & 0 & 0 & 0\\ 
  0 & 1 & 0 & 0  & 0\\
 0 & 0 & 0 & 0  & 0 \\
\end{pmatrix} \,\, , \,\, 
S_2^X = \begin{pmatrix}
  0 & 0 & 0 & 0 & 0  \\
  0 & 0 & 0  & 0  & 0\\
  1 & 0 & 0 & 0 & 0\\ 
  0 & 0 & 0 & 0  & 0\\
 0 & 0 & 1 & 0  & 0 \\
\end{pmatrix} .
\ees
It follows that 
\bes
\cA_X = \left\{\begin{pmatrix}
  a & 0 & 0 & 0 & 0  \\
  b & a & 0  & 0  & 0\\
  c & 0 & a & 0 & 0\\ 
  d & b & 0 & a  & 0\\
 e & 0 & c & 0  & a \\
\end{pmatrix}  : a,b,c,d,e \in \mb{C} \right\}.
\ees
Similarly, $\cF_Y$ is seen to have 
$\{\Omega_Y, e_1, e_2, (e_1^2  - 2e_1 e_2)/\sqrt{3} , e_2^2\}$ 
as an orthonormal basis, so we obtain the shifts
\bes
S_1^Y = \begin{pmatrix}
  0 & 0 & 0 & 0 & 0  \\
  1 & 0 & 0  & 0  & 0\\
  0 & 0 & 0 & 0 & 0\\ 
  0 & 1/\sqrt{3} & -1/\sqrt{3} & 0  & 0\\
 0 & 0 & 0 & 0  & 0 \\
\end{pmatrix} \,\, , \,\, 
S_2^Y = \begin{pmatrix}
  0 & 0 & 0 & 0 & 0  \\
  0 & 0 & 0  & 0  & 0\\
  1 & 0 & 0 & 0 & 0\\ 
  0 & -1/\sqrt{3} & 0 & 0  & 0\\
 0 & 0 & 1 & 0  & 0 \\
\end{pmatrix},
\ees
and the algebra
\bes
\cA_Y = \left\{\begin{pmatrix}
  a & 0 & 0 & 0 & 0  \\
  b & a & 0  & 0  & 0\\
  c & 0 & a & 0 & 0\\ 
  d & \frac{b-c}{\sqrt{3}} & \frac{-b}{\sqrt{3}} & a  & 0\\
 e & 0 & c & 0  & a \\
\end{pmatrix} : a,b,c,d,e \in \mb{C} \right\}.
\ees
Here (as in any finite dimensional example), we have 
$\cT_X = \cT_Y = M_5(\mb{C})$. 
How does one go about showing that the algebras $\cA_X$ 
and $\cA_Y$ are not isometrically isomorphic? 
We will provide an answer at the end of the next section.
\end{example}

\section{Classification of the algebras by their subproduct systems}\label{sec:class}

\subsection{The character spaces as analytic varieties}
In this section, our subproduct systems are not necessarily commutative. 
Let $X$ be a subproduct system.
Let $\cM_X$ denote the space of all unital, multiplicative linear 
functionals on $\cA_X$. 
The maps in $\cM_X$ will be called \emph{characters}. 
Recall that every character is automatically contractive, 
hence completely contractive too.

The character space may be (homeomorphically) identified with the set 
\bes
Z(I^X) = \{z \in \overline{\mb{B}_d} :  p(z) = 0 \text{ for all } p \in I^X  \}
\ees
via the identification:
\be\label{eq:identification}
\cM_X \ni \rho \longleftrightarrow (\rho(S_1^X), \ldots, \rho(S_d^X)) \in Z(I^X).
\ee
See \cite[Section 10.2]{ShalitSolel} for details. 

We will also use the notation and identification 
\bes
\cM_X^o \cong Z^o(I^X) = \{z \in \mb{B}_d :   p(z) = 0  \text{ for all } p \in I^X \}.
\ees

The character corresponding to the point $0 \in Z(I^X)$ is called 
\emph{the vacuum state}, and is denoted by $\rho_0$. 
It is the unique multiplicative linear functional sending $I$ to $1$ 
and $S_i^X$ to $0$ for $i=1,\ldots,d$. 
The vacuum state is a vector state, and is given by
\bes
\rho_0(T) = \lel T \Omega_X, \Omega_X\rir.
\ees
We intentionally use the same notation for vacuum states acting on different algebras. 
If $\varphi : \cA_X \rightarrow \cA_Y$ and $\varphi^*(\rho_0) = \rho_0$ 
then we say that $\varphi$ \emph{preserves the vacuum state}. 
The following theorem explains the significance of the vacuum state to our discussion.

\begin{theorem}{\bf [Theorem 9.7, \cite{ShalitSolel}]}\label{thm:iso_vacuum}
$X \cong Y$ if and only if $\cA_X$ and $\cA_Y$ are isometrically isomorphic 
via an isomorphism that preserves the vacuum state. 
In fact, if $\varphi : \cA_X \rightarrow \cA_Y$ is a vacuum preserving 
isometric isomorphism, then there is an isomorphism 
$V : X \rightarrow Y$ such that for all $T \in \cA_X$,
\bes
\varphi(T ) = V T V^* .
\ees 
\end{theorem}

For $\lambda = (\lambda_1, \ldots, \lambda_d) \in Z(I^X)$, let us denote by 
$\rho_\lambda$ the character sending $S^X_i$ to $\lambda_i$. 
For every $T \in \cA_X$, the Gelfand transform gives rise to a 
continuous function on $\cM_X$ by 
\[
 \hat{T}(\lambda) = \rho_\lambda (T) .
\] 
If $p \in \mb{C}[z]$, then
$\widehat{p(S^X)}(\lambda) = \rho_\lambda(p(S^X)) = p(\lambda)$. 
If $T \in \cA_X$ and $p_n(S^X)$ converges to $T$ in norm, then by the 
contractivity of the Gelfand transform, $p_n$ converges uniformly 
to $\hat{T}$ on $\cM_X$. 
Therefore, for every fixed $\lambda \in \cM_X$, the function 
$\hat{T}_\lambda(t) = \hat{T}(t\lambda_1, \ldots, t\lambda_d)$
is analytic in $\mb{D}$.

Every continuous isomorphism $\varphi : \cA_X \rightarrow \cA_Y$ 
gives rise naturally to a homeomorphism 
$\varphi^* : \cM_Y \rightarrow \cM_X$ given by 
$\varphi^*(\rho) = \rho\circ \varphi$.

\begin{lemma}\label{lem:interior}
If $\varphi$ is an isometric isomorphism, then $\varphi^*$ 
maps $\cM_Y^o$ onto $\cM_X^o$.
\end{lemma}

\begin{proof}
Let $\rho \in \cM_X \setminus \cM_X^o$. By applying a unitary transformation to 
the variables we may assume that $\rho = (1,0,\ldots,0)$. 
Assume that $(\varphi^*)^{-1} \rho = \rho_{t_0 \lambda}$, 
where $t_0 \in [0,1)$ and $\lambda \in \cM_Y$. 
Put $T = \varphi(S^X_1)$. 
Then $\|T\| = 1$, thus $|\hat{T}_\lambda(t)| \leq 1$ for $t \in \mb{D}$. 
On the other hand, $\hat{T}_\lambda(t_0) = \rho(S^X_1) = 1$. 
By the maximum modulus principle, $\hat{T}_\lambda$ is 
constant $1$ on $\mb{D}$. 
We claim that this is possible only if $T = I$. 
That would show that $\varphi(S^X_1) = I$, but that is impossible 
because $\varphi$ is injective and unital. 
This contradiction completes the proof.

To derive $T = I$ from $\hat{T}_\lambda(t) \equiv 1$, assume that 
$T = \sum_n  p_n(S^X)$ is the Ces\`{a}ro norm-convergent series of 
$T$ (see \cite[Proposition 9.3]{ShalitSolel}), where $p_n$ are homogeneous 
polynomials of degree $n$. 
The terms $p_n(S^X)$ must be bounded, therefore $p_n(\lambda)$ 
are also bounded. 
Then for $t \in \mb{D}$ we have that
\bes
\hat{T}_\lambda(t) = \sum_n  p_n(t \lambda) = \sum_n p_n(\lambda) t^n .
\ees
This holomorphic function can be constantly equal to $1$ only if 
$p_n(\lambda) = 0$ for $n \neq 0$ and $p_0 = 1$. 
So $T = I + \sum_{n>0}p_n(S^X)$. 
Now $\|T\| = 1$ implies $\sum_{n>0}p_n(S^X) = 0 $.
\end{proof}

\begin{remark}\label{rem:bounded_interior}
It is also true that if $\varphi : \cA_X \rightarrow \cA_Y$ is a bounded 
isomorphism, then $\varphi^*$  maps $\cM_Y^o$ onto $\cM_X^o$. 
Since we will not require this result, the proof is omitted.
See Proposition~\ref{prop:algiso_biholo} for the commutative case.
\end{remark}

\begin{lemma}\label{lem:analytic}
Let $X$ and $Y$ be two subproduct systems with $\dim X(1) = d'$ 
and $\dim Y(1) = d$. 
Let $\varphi : \cA_X \rightarrow \cA_Y$ be an isometric isomorphism. 
Then there exists a holomorphic map $f : \mb{B}_{d} \rightarrow \mb{C}^{d'}$ 
such that 
\bes
\varphi^* |_{\cM_Y^o} = f |_{\cM_Y^o}.
\ees
That is, the restriction of $\varphi^*$ to $\cM_Y^o$ is an analytic map 
of analytic varieties.
\end{lemma}

\begin{proof}
Let $T = \varphi(S^X_1)$, and let $T = \sum_n p_n(S^Y)$ be the 
Ces\`{a}ro norm-convergent series of $T$. 
Denote $E = Y(1)$, and let $\{e_1, \ldots, e_d\}$ be an orthonormal basis for $E$. 
We can rewrite the series for $T$ as
\[
 T = \sum_{n=0\ }^\infty \sum_{i_1, \ldots, i_n = 1}^d 
 b_{i_1, \ldots, i_n} S^Y_{i_1} \cdots S^Y_{i_n} ,
\]
where
\[
 T (\Omega_Y) =  \sum_{n=0\ }^\infty \sum_{i_1, \ldots, i_n = 1}^d 
 b_{i_1, \ldots, i_n} e_{i_1} \otimes \cdots \otimes e_{i_n}
\] 
is the image of the vacuum vector in the full Fock space $\cF(E)$.

It follows that the coefficients $\{b_{i_1,\ldots,i_n}\}$ are $\ell^2$ summable. 
The estimate
\bes
\sum |b_{i_1, \ldots, i_n} z_{i_1} \cdots z_{i_n}| \leq 
\big(\sum |b_{i_1, \ldots, i_n}|^2)^{1/2}
\big(\sum |z_{i_1} \cdots z_{i_n}|^2\big)^{1/2}
\ees
together with the identity
\bes
\sum_{n=0}^\infty \sum_{i_1, \ldots, i_n = 1}^d |z_{i_1} \cdots z_{i_n}|^2 = 
\sum_{n=0}^\infty (|z_1|^2 + \ldots + |z_{d}|^2)^n
\ees
shows that the function 
\bes
f_1(z) = \sum b_{i_1, \ldots, i_n} z_{i_1} \cdots z_{i_n}
\ees
is holomorphic in $\mb{B}_{d}$. 
But 
\[
 \varphi^*(\rho_\lambda)(S^X_1) = \rho_\lambda (T) 
 = \sum b_{i_1, \ldots, i_n} \lambda_{i_1} \cdots \lambda_{i_n} = f_1(\lambda) .
\]
Thus $\varphi^* \rho_\lambda = \rho_{\mu}$, where $\mu_1 = f_1(\lambda)$. 
In the same way, we see that $\mu_i = f_i(\lambda)$, for all $i=1, \ldots, d'$, 
where $f_i: \mb{B}_{d} \rightarrow \mb{C}^{d'}$ is holomorphic.
\end{proof}

\subsection{The singular nucleus of a homogeneous variety}

\begin{lemma}\label{lem:sing}
Let $V = V(I)$ be the variety in $\mb{C}^d$ determined 
by a radical homogeneous ideal $I$. 
Then either $V$ has singular points, or $V$ is a linear subspace.
\end{lemma}

\begin{proof}
If $V$ is reducible, then by Theorem 8(iv) in \cite[Section 9.6]{CLO92} 
the origin is in the singular set. 
So we may assume that $V$ is irreducible. 

Let $f_1, \ldots, f_k$ be a generating set for $I$, and assume the 
dimension of $V(I)$ is $m$. 
By a theorem in \cite[page 88]{SKKT}, the singular locus of $V$ is 
the common zero set of polynomials obtained from the 
$(d-m) \times (d-m)$ minors of the Jacobian matrix
\bes
\begin{pmatrix}
\frac{\partial f_1}{\partial z_1} & \cdots & \frac{\partial f_1}{\partial z_d} \\
\vdots & & \vdots \\
\frac{\partial f_k}{\partial z_1} & \cdots & \frac{\partial f_k}{\partial z_d}
\end{pmatrix} \,\,.
\ees
But since $f_1, \ldots, f_k$ are homogeneous, all these minors will vanish 
at the point $0$ unless at least $d-m$ of the $f_i$'s are linearly independent 
linear forms. 
But then $V$ lies inside $m$ dimensional subspace. 
Being an $m$-dimensional variety, $V$ must be that subspace.
\end{proof}

Let $V$ be a homogeneous variety in $\mb{C}^d$. 
Then by the lemma, either $V$ is a subspace of $\mb{C}^d$, 
or the singular locus $\Sing (V)$ is nonempty. 
Now $\Sing (V)$ is also a homogeneous variety, 
so either $\Sing (V)$ is a subspace or 
$\Sing (\Sing (V))$ is not empty. 
Since the dimension of the singular locus is strictly less than 
the dimension of a variety, we eventually arrive at a subspace 
$N(V) = \Sing (\cdots(\Sing (V)\cdots)$ 
which we call \emph{the singular nucleus of $V$}. 
Note that $N(V) = \{0\}$ might happen, as well as $N(V) = V$. 

If $X$ is a subproduct system and $I = I^X$, then from 
Lemma~\ref{lem:analytic} it is clear that $\mb{B}_d \cap N(V(I))$ 
is an invariant of the isometric isomorphism class of $\cA_{X}$. 
We also refer to this set as the singular nucleus of $I$.

\subsection{Classification of the algebras by subproduct systems}

In what follows we will need to consider the group $\Aut(\mb{B}_n)$ 
of automorphisms of $\mb{B}_n$, that is, the biholomorphisms of the unit ball. 
We will use well known properties of these fractional linear maps 
(see \cite[Section 2.2]{RudinBall}). 
For $a \in \mb{B}_n$, we define
\be\label{eq:mobius}
\phi_a(z) = \frac{a - P_a z - s_a Q_a z}{1-\lel z,a\rir},
\ee
where $P_a$ is the orthogonal projection onto $\spn\{a\}$, 
$Q_a = I_n - P_a$ and $s_a = (1-|a|^2)^{1/2}$. 
Then $\phi_a$ is an automorphism of $\overline{\mb{B}}_n$ that maps 
$0$ to $a$ and satisfies $\phi_a^2 = \id$. 
For every $\psi \in \Aut(\mb{B}_n)$ there exists a unique 
unitary $U$ and $a \in \mb{B}_n$ such that $\psi = U \circ \phi_a$.

By a \emph{disc} in $\mb{B}_n$ we shall mean a set $D$ of the form 
$D = \mb{B}_n \cap L$, where $L \subseteq \mb{C}^n$ is a one dimensional subspace.

\begin{lemma}\label{lem:discs}
Let $\psi\in \Aut (\mb{B}_n)$. 
Then there are two discs $D_1, D_2$ in $\mb{B}_n$ such that $\psi(D_1) = D_2$.
\end{lemma}

\begin{proof}
If $\psi = U \circ \phi_a$ and $a \ne 0$, take $D_1 = \spn\{a\}\cap \mb{B}_n$.
Then $\phi_a|_{D_1}$ is a M\"obius map of $D_1$ onto itself.
Take $D_2 = UD_1$.
If $a=0$, take $D_1=D_2$ to be $\mb{B}_n\cap L$ where $L$ is any
one-dimensional eigenspace of $U$.
\end{proof}

\begin{proposition}\label{prop:exist_vacuum}
Let $X$ and $Y$ be subproduct systems and assume that there exists 
an isometric isomorphism $\varphi: \cA_X \rightarrow \cA_Y$. 
Then there exists a vacuum preserving isometric isomorphism from 
$\cA_X$ to $\cA_Y$.
\end{proposition}

\begin{proof}
By the discussion following Lemma \ref{lem:sing}, the singular nucleus 
of $I^Y$ must be mapped biholomorphically by $\varphi^*$ onto the 
singular nucleus of $I^X$. 
If these nuclei are both $\{0\}$ then $\varphi$ itself must be vacuum preserving, 
and we are done. 
Otherwise, by rotating the coordinate systems we may assume that 
$N(V(I^X)) = N(V(I^Y)) = B$, a complex ball.

Now, $\varphi^*|_B \in \Aut(B)$. Thus by Lemma \ref{lem:discs} ,
there are two discs $D_1, D_2 \subseteq B$ such that $\varphi^*(D_2) = D_1$.
Define
\bes
\cO(0;X,Y) = \{z \in D_1 :  z = \psi^*(0) \textrm{ for some isometric isomorphism }  
\psi : \cA_{X} \to \cA_{Y} \},
\ees
and
\bes
\cO(0;Y) = \{z \in D_2 :  z = \psi^*(0) \textrm{ for some isometric automorphism }  
\psi \textrm{ of } \cA_Y\}.
\ees

\smallbreak
\noindent \textit{Claim:}
\emph{The sets $\cO(0;X,Y)$ and $\cO(0;Y)$  are invariant under rotations about $0$.}\\
\textit{Proof of claim:} For $\lambda$ with $|\lambda| = 1$, 
write $\varphi_\lambda$ for the isometric 
automorphism mapping $S_i^{X}$ to $\lambda S_i^{X}$ ($i = 1,\ldots,d$). 
Let $b = \varphi^*(0) \in \cO(0;X,Y)$. 
Recall that $b = (b_1, \ldots, b_d)$ is identified with a character 
$\rho_b \in \cM_X^o$ such that $\rho_b(S_i^X) = b_i$ for $i=1, \ldots, d$. 
Consider $\varphi \circ \varphi_\lambda$.
We have 
\[
\rho_0((\varphi \circ \varphi_\lambda)(S_i^{X})) = 
\rho_0(\varphi(\lambda S_i^{X})) = 
\lambda \rho_0(\varphi(S_i^{X})) = \lambda b_i .
\]
Thus $\lambda b = (\varphi \circ \varphi_\lambda)^* (\rho_0) \in \cO(0;X,Y)$. 
The proof for $\cO(0;Y)$ is the same. 
\smallbreak

We can now show the existence of a vacuum preserving isometric isomorphism.
Let $b = \varphi^*(0)$. 
If $b = 0$ then we are done, so assume that $b\neq 0$.
By definition, $b \in \cO(0;X,Y)$. 
Denote $C := \{z \in D_1 : |z|=|b|\}$. 
By the above claim, $C \subseteq \cO(0;X,Y)$. 
Consider $C' := (\varphi^*)^{-1}(C)$. 
We have that $C' \subseteq \cO(0;Y)$.
Now $C'$ is a circle in $D_2$ that goes through the origin. 
By the claim, the interior of $C'$, $\Int(C')$, is in $\cO(0;Y)$. 
But then $\varphi^*(\Int(C'))$ is the interior of $C$, 
and it is in $\cO(0;X,Y)$. 
Thus $0$ lies in $\cO(0;X,Y)$, as required.
\end{proof}

Combining Theorem \ref{thm:iso_vacuum} and Proposition \ref{prop:exist_vacuum},
we obtain our main non-commutative result:

\begin{theorem}\label{thm:alg_sps_iso}
Let $X$ and $Y$ be subproduct systems. 
Then $\cA_X$ is isometrically isomorphic to $\cA_Y$ 
if and only if $X$ is isomorphic to $Y$.
\end{theorem}

\begin{remark}
It follows from the above theorem that if $\cA_X$ and $\cA_Y$ 
are isometrically isomorphic, then they are also completely isometrically isomorphic.
\end{remark}

\begin{example}
Let us return to Example \ref{expl:notiso}. 
We now show that $\cA_X$ is not isometrically isomorphic to $\cA_Y$. 
Using the above theorem, it is enough to show that $X$ is not isomorphic to $Y$. 
By Proposition \ref{prop:idealsps}, one must show that there is 
no unitary change of variables that takes $I$ onto $J$. 
But if there was, then the set
\bes
Z(I^{(2)}) = \{z \in \mb{B}_2 :  f(z) =0 \text{ for all }  f \in I^{(2)} \}
\ees 
would be mapped unitarily onto the set
\bes
Z(J^{(2)}) = \{z \in \mb{B}_2 : f(z)=0 \text{ for all }  f \in J^{(2)} \} ,
\ees 
where $I^{(2)}$ denotes the set of homogeneous polynomials 
in $I$ with degree $2$, etc. 
However, $Z(I^{(2)})$ consists of two complex lines that intersect 
at an angle $\pi/2$, and $Z(J^{(2)})$ consists of two complex lines 
that intersect at an angle $\pi/4$. 
It follows from the theorem (together with Proposition \ref{prop:idealsps}) 
that $\cA_X$ and $\cA_Y$ are not isometrically isomorphic.
\end{example}

\section{The algebras $\cA_X$ as algebras of continuous multipliers}\label{sec:mult}

{}From this point onward, we will concentrate on the commutative case.
The purpose of this section is to show that when $X$ is commutative 
and $I^X$ is a radical ideal in $\mb{C}[z]$, the algebra $\cA_X$ can be 
realized as a norm closed subalgebra of the 
multiplier algebra of a reproducing kernel Hilbert space.

Let $I \subseteq \mb{C}[z]$ be an ideal, not necessarily homogeneous. 
We will denote the closure of $I$ in $H^2_d$ by $[I]$. 
Define 
\bes
\cF_I = H^2_d \ominus I.
\ees
When $I = I^X$ is a homogeneous ideal, then $\cF_I = \cF_X$, the $X$-Fock space. 

Recall that for an ideal $I \subseteq \mb{C}[z]$ we denote 
\bes
V(I) = \{z \in \mb{C}^d : p(z) = 0 \text{ for all }  p \in I \},
\ees
\bes
Z(I) = V \cap \overline{\mb{B}_d} ,
\ees
and
\bes
Z^o(I) = V \cap \mb{B}_d.
\ees
If $W \subseteq \mb{C}^d$, we define
\bes
I(W) = \{f \in \mb{C}[z] : f(\lambda) = 0 \text{ for all }  \lambda \in W   \}.
\ees

\begin{lemma}\label{lem:intersect}
Let $I$ be a radical ideal in $\mb{C}[z]$ such that all the irreducible 
components of $V(I)$ intersect $\mb{B}_d$. 
Then $I(Z^o(I)) = I$.
\end{lemma}

\begin{proof}
This is an exercise in algebraic geometry. 
Assume first that $V(I)$ is irreducible.
Let $f \in \mb{C}[z]$ such that $f(\lambda) = 0$ for all 
$\lambda \in Z^o(I) = V(I) \cap \mb{B}_d$. 
Denote $W = V(f)$. 
By assumption, $W \cap \mb{B}_d \supseteq V(I) \cap \mb{B}_d$, 
therefore $\dim W \cap V(I) = \dim V(I)$. 
It follows from \cite[Proposition 1.4]{Mum} that 
$W \cap V(I) = V(I)$, therefore $f \in I(V(I)) = I$.

Finally, if $V(I)$ is reducible then we apply this argument to each 
irreducible component.
\end{proof}

For any ideal $I$ in $\mb{C}[z]$, the radical of $I$ is
\[ \sqrt I = \{ f \in \mb{C}[z] : f^n \in I \text{ for some } n \ge 1 \} = I(V(I)) .\]

\begin{corollary}\label{cor:ballvariety}
If $I$ is a homogeneous ideal, then 
\[ \sqrt I = I(V(I)) = I(Z(I)) = I(Z^o(I)) .\]
\end{corollary}

\begin{lemma}\label{lem:homotrivial}
If $I$ is a homogeneous ideal in $\mb{C}[z]$, then $[I] \cap \mb{C}[z] = I$.
\end{lemma}

We omit the easy proof of this lemma.
However we note that it is not true for non-homogeneous ideals. 
Indeed, if $d=1$, $I = \lel x - 1 \rir$, 
then $H^2_1 \ominus I = \{0\}$.
Thus $[I] = H^2_1$, and $[I] \cap H^2_1 = \mb{C}[z]$. 
\smallbreak

Now we turn to the reproducing kernel.
For any $\lambda \in \mb{B}_d$, let 
\be\label{eq:kernel}
\nu_\lambda = (1-\|\lambda\|^2)^{1/2} 
\sum_{n=0}^\infty \sum_{i_1,\ldots,i_n =1}^d  
\overline{\lambda_{i_1} \cdots \lambda_{i_n}} 
e_{i_1} \otimes \cdots \otimes e_{i_n} .
\ee
It is known \cite{DavPitts1} that $\nu_\lambda$ are eigenvectors 
for the operators $L_i^*$ (the adjoints of the left creation operators 
$L_i$ on the full Fock space) with eigenvalue $\overline{\lambda_i}$. 
Since the multiplication operators $Z_i$ are co-restrictions of the 
$L_i$'s to $H^2_d$, and since 
\be\label{eq:eigenvector}
\nu_\lambda = (1-\|\lambda\|^2)^{1/2} 
\sum_{\alpha} \frac{|\alpha|!}{\alpha_1! \cdots \alpha_d!}
\overline{\lambda}^\alpha e^\alpha \in H^2_d ,
\ee
we have that $\nu_\lambda$ are eigenvectors of $Z_i^*$ with 
eigenvalues $\overline{\lambda_i}$. 

Alternatively, $H^2_d$ is known \cite{Arv98} to be a reproducing kernel 
Hilbert space with kernel 
\bes
k(\xi,\lambda) = \frac{1}{1-\lel \xi, \lambda \rir} .
\ees
The kernel function at $\lambda$, the function $k(\cdot, \lambda)$, is seen to 
correspond to (\ref{eq:eigenvector}). 
Denote by $\Mult(H^2_d)$ the multiplier algebra of $H^2_d$. 
 From the basic theory of multiplier algebras, it follows that for any 
$\phi \in \Mult(H^2_d)$, $\nu_\lambda$ is an eigenvector for 
$M_\phi^*$ with eigenvalue $\overline{\phi(\lambda)}$ \cite[Chapter 2]{AM}.

We now compute which $\nu_\lambda$ belong to $\cF_I$ for a given ideal $I$. 

\begin{lemma}\label{lem:F_I}
The vector $\nu_\lambda$ is in $\cF_I$ if and only if $\lambda \in Z^o(I)$.
\end{lemma}

\begin{proof}
Fix $\lambda \in \mb{B}_d$. 
Then $\nu_\lambda$ lies in $\cF_I$ if and only if 
$\nu_\lambda$ is orthogonal to $I$, if and only if for all $f \in I$ we have
\bes
f(\lambda) = \lel f, \nu_\lambda \rir  = 0 .
\ees
This happens if and only if $\lambda \in Z^o(I) = V(I) \cap \mb{B}_d$. 
\end{proof}

\begin{lemma}\label{lem:rad}
Let $I \subseteq \mb{C}[z]$ be a homogeneous ideal. 
Then 
\[
 \cF_I = \overline{\spn}\{\nu_\lambda : \lambda \in Z^o(I)\} 
\] 
if and only if $I$ is radical.
\end{lemma}

\begin{proof}
Assume that $\cF_I = \overline{\spn}\{\nu_\lambda : \lambda \in Z^o(I)\}$. 
Let $[I]$ denote the closure of $I$ in $H^2_d$. 
Then 
\[
[I] = \cF_I^{\perp} = 
\{f \in H^2_d : f(\lambda) = 0 \text{ for all }  \lambda \in Z^o(I) \} .
\]
By Lemma \ref{cor:ballvariety}, $[I] \cap \mb{C}[z] = I(V(I)) = \sqrt{I}$, 
and by Lemma \ref{lem:homotrivial}, $[I] \cap \mb{C}[z] = I$. 
Thus, $I$ is radical.

Now assume that $I$ is radical. 
By Lemma \ref{lem:F_I}, $\nu_\lambda \in \cF_I$ for all 
$\lambda \in Z(I) \cap \mb{B}_d$. 
Thus we need only show that if $f \in H^2_d$ is orthogonal to 
$\{\nu_\lambda : \lambda \in Z^o(I) \}$ then $f \in [I]$. 
Let $f \in \{\nu_\lambda : \lambda \in Z^o(I)\}^\perp$.  
Write the Taylor series of $f$ as $f(z) = \sum_{\alpha} a_\alpha z^\alpha$. 
Then for all $\lambda \in Z^o(I)$, we define a function $g_\lambda$ on $\mb{D}$ by  
\bes
 g_\lambda(t) = f(t \lambda) = 
 \sum_n \left(\sum_{|\alpha|=n} a_\alpha \lambda^\alpha\right) t^n.
\ees
But $g_\lambda \equiv 0$, thus $\sum_{|\alpha|=n} a_\alpha z^\alpha \in I(Z^o(I))$ 
for all $n$. 
Since $I$ is radical, $I = I(Z^o(I))$ by 
Lemma~\ref{cor:ballvariety}. 
So $f$ belongs to $[I]$.
\end{proof}

\begin{proposition}\label{prop:mult}
Let $I \subseteq \mb{C}[z]$ be a radical homogeneous ideal. Then 
$\cF_I$ is naturally a reproducing kernel Hilbert space on the set $Z^o(I)$.
The algebra $\cA_I$ is the norm closure of the polynomials in $\Mult(\cF_I)$, 
and can be identified with 
\bes
\{f|_{Z^o(I)} : f \in \cA_d\}.
\ees
Moreover, $\cL_I = (\cL_d^* |_{\cF_I} )^*$ can be identified
with $\textrm{\em Mult}(\cF_I)$, and 
\be\label{eq:restrict_multipliers}
\textrm{\em Mult}(\cF_I) = \{f|_{Z^o(I)} : f \in \textrm{\em Mult}(H^2_d)\}.
\ee
\end{proposition}

\begin{proof}
Since $\cF_I = \overline{\spn}\{\nu_\lambda : \lambda \in Z^o(I)\}$, 
it is naturally a reproducing kernel Hilbert space on the set $Z^o(I)$ 
with kernel functions $\nu_\lambda$, $\lambda \in Z^o(I)$.

Now, $\cA_I$ is generated as the operator norm closure of the identity 
and the compressions of the coordinate functions 
$S_i = P_{\cF_I} Z_i |_{\cF_I}$, $i=1, \ldots, d$, to a coinvariant space. 
Since $S_i^* \nu_\lambda = \overline{\lambda_i} \nu_\lambda$, 
$S_i$ is the multiplier operator that sends $f(z) \in \cF_I$ 
(a function on $Z^o(I)$) to $z_i f(z)$. This shows that $\cA_I$ is
the norm closure of the polynomials in $\textrm{Mult}(\cF_I)$.

The same argument shows that $\cL_I$ is a \wot-closed algebra of multipliers 
in $\textrm{Mult}(\cL_I)$ generated by polynomials. 
Furthermore, if $f \in \textrm{Mult}(H^2_d)$ and $M_f$ is the corresponding 
multiplication operator on $\cF_I$, then $P_{\cF_I} M_f |_{\cF_I} = M_g$, 
where $g$ is the multiplier on $\cF_I$ given by $g = f|_{Z^o(I)}$.  
This provides a natural identification between $\cL_I$ and 
$\{f|_{Z^o(I)} : f \in \textrm{Mult}(H^2_d)\}$. 

To establish Equation (\ref{eq:restrict_multipliers}), it remains to show that every 
multiplier in $\textrm{Mult}(\cL_I)$ extends to a multiplier in $\textrm{Mult}(H^2_d)$ of the same norm. 
This follows from \cite[Theorem 3.3]{DavPittsPick} or  \cite[Theorem 2.8]{AriasPopescu}.
\end{proof}

Thus, the algebra $\cA_{I}$, which is the universal unital 
operator algebra generated by a row contraction satisfying the relations 
in $I$, can be given three interpretations. 
First $\cA_{I}$ is the quotient algebra $\cA_d / \overline{I}$; 
second, $\cA_{I}$ is the concrete operator algebra generated by compression 
of $\cA_d$ to $\cF_I$; and thirdly, it is also an algebra of functions
\bes
\{f|_{Z^o(I)} : f \in \cA_d\},
\ees
of restrictions given the multiplier norm (on the subspace $\cF_I$). 
All of these points of view are useful.

\section{Nullstellensatz for homogeneous ideals in multiplier algebras}\label{sec:Null}

Our goal in this section is to obtain a (projective) Nullstellensatz for a 
large class of operator algebras, including $\cA_d$ and the 
``ball algebra" $A(\mb{B}_d)$. 
 From this result we will derive an approximation result 
(Corollary \ref{cor:approxpoly}) that will allow us to describe isomorphisms 
between the algebras $\cA_X$ that are induced by automorphisms 
of $\mb{B}_d$ (Proposition \ref{prop:induced} below). 
At the end of the section we will also provide a different and 
quick proof of Corollary \ref{cor:approxpoly} for the algebra $\cA_d$. 

Let $\Omega \subseteq \mb{C}^d$ be an open bounded domain that is the 
union of polydiscs centered at $0$. 
Then $\Omega$ has the following property:
\bes
 \lambda \in \Omega \Rightarrow t \lambda \in \Omega \,\, , \,\, 
 \textrm{for all } t \in \overline{\mb{D}}
\ees
and $\Omega$ also the property that every function $f$ holomorphic in 
$\Omega$ has a Taylor series that converges in $\Omega$.

Let $\cH$ be a reproducing kernel Hilbert space of analytic functions in 
$\Omega$ containing the polynomials with the additional property that 
$f(z) \mapsto f(e^{it}z)$ is a unitary operator on $\cH$ for all $t \in \mb{R}$.
It follows that if $p,q \in\cH$ are homogeneous polynomials of 
different total degrees, then $\lel p, q\rir = 0$. 

In the discussion below $B$ will denote the closure of the polynomials in the 
multiplier algebra $\Mult(\cH)$. 
If $\cH = H^2_d$, then $B = \cA_d$, which is the case of principal interest in this paper.
If $\cH$ is taken to be the Bergman space on $\Omega$, then $B$ is $A(\Omega)$,
the space of continuous functions on $\ol{\Omega}$ which are analytic on $\Omega$,
with the sup norm. 
As is always the case with algebras of multipliers, the norm of $B$, 
which will be denoted simply by $\|\cdot\|$, 
satisfies $\|f\|_\infty \leq \|f\|$ (see \cite[Chapter 2]{AM}).

Every $f \in B$ has a Taylor series in $\Omega$, 
$f(z) = \sum_\alpha a_\alpha z^\alpha$. 
We write
\be\label{eq:hom}
f = \sum_{n=0}^\infty f_n
\ee
where $f_n(z) = \sum_{|\alpha| = n}a_\alpha z^\alpha$ is the 
$n$th homogeneous component of $f$. 
The series (\ref{eq:hom}) converges locally uniformly in $\Omega$.

\begin{lemma}\label{lem:closcom}
For all $n$, the map $P_n : B \rightarrow \mb{C}[z] \subseteq B$ 
given by $P_n(f) = f_n$ is contractive. 
Furthermore, the series (\ref{eq:hom}) is Ces\`{a}ro norm convergent 
to $f$ in the norm of $B$.
\end{lemma}

\begin{proof}
Consider the gauge automorphisms on $B$:
\bes
[\gamma_t(f)](z) = f(e^{it}z).
\ees
The unitary group given by $[U_t(h)](z) = h(e^{it}z)$ is continuous in the 
strong operator topology, and $\gamma_t = \ad U_t$.
Hence the path $t \mapsto \gamma_t(f)$ is continuous with respect 
to the strong operator topology. 
One sees therefore that the integral
\bes
P_n(f) = \frac{1}{2\pi}\int_0^{2\pi} \gamma_t(f) e^{-int} dt
\ees
converges in the strong operator topology to an element of $B(\cH)$. 
The operator $P_n$ is a complete contraction, as it is an average of 
complete contractions. 
Note that $P_n$ maps $\mb{C}[z]$ onto the space $H_n$ of 
homogeneous polynomial of degree $n$. 
This fact follows from the simple identity $U_s P_n(f) = e^{ins}P_n(f)$.  
Therefore, $P_n$ maps $B = \overline{\mb{C}[z]}^{\|\cdot\|}$ onto $H_n$. 
A standard argument using the Fej\'{e}r kernel shows that 
the Ces\`{a}ro means $\Sigma_n(f)$ are completely contractive
and  converge in norm to $f$, and that $P_n(f) = f_n$. 
\end{proof}

In particular, we see that $f$ is in the closed linear span of its 
homogeneous components. This will be used repeatedly below.

\begin{definition}
An ideal $J \subseteq B$ is said to be \emph{homogeneous} if 
$f_n \in J$ for all $n \in \mb{N}$ and all $f \in J$.
\end{definition}

\begin{proposition}\label{prop:homeq}
A closed ideal $J \subseteq B$ is homogeneous if and only if for all 
$t \in \mb{D}$ and all $f \in J$, one has $f(tz) \in J$.
\end{proposition}

\begin{proof}
Assume that $J$ is homogeneous, and let $f(z) = \sum_n f_n(z) \in J$. 
By the previous lemma $\|f_n\| \leq \|f\|$, so for all $t\in \mb{D}$, 
$f(tz) = \sum_n t^n f_n(z)$ is a norm convergent series of elements in $J$.
Hence $f(tz) \in J$.

Conversely, let $f \in J$, and assume that for all $t \in \mb{D}$, $f(tz) \in J$. 
Assuming that $J$ is proper, $f_0 = 0$ follows from taking $t = 0$. 
But then 
\bes 
\frac{f(tz)}{t} = \sum_{n=0}^\infty t^n f_{n+1} \in J.
\ees
Taking $t \rightarrow 0$ we find that $f_1(z) \in J$. 
Now we consider
\bes 
\frac{f(tz) - f_1(tz)}{t^2} = \sum_{n=0}^\infty t^n f_{n+2}(z) \in J,
\ees
taking the limit as $t \rightarrow 0$ we find that $f_2(z) \in J$. 
The result follows by recursion.
\end{proof}

\begin{lemma}\label{lem:pinI}
Let $I \subseteq \mb{C}[z]$ be a homogeneous ideal. 
Then the closure of $I$ in $B$ is homogeneous. 
If $p$ is a homogeneous polynomial in $\overline{I}$, then  $p \in I$.
\end{lemma}

\begin{proof}
This follows easily from the continuity of $P_n$.
\end{proof}

\begin{lemma}\label{lem:Iin}
Let $J$ be a homogeneous ideal in $B$. 
Then the ideal $I = \mb{C}[z] \cap J$ of $\mb C[z]$ satisfies 
$I \subseteq J \subseteq \overline{I}$, 
and it is the unique homogeneous ideal in $\mb{C}[z]$ with this property. 
\end{lemma}

\begin{proof}
Clearly $I \subseteq J$, and that $J \subseteq \overline{I}$ follows from 
Lemma \ref{lem:closcom}.
If $K$ is another homogeneous ideal in $\mb{C}[z]$ such that 
$K \subseteq J \subseteq \overline{K}$, then we have 
$I \subseteq \overline{K}$ and $K \subseteq \overline{I}$.
 From Lemma \ref{lem:pinI}, $I = K$.
\end{proof}

\begin{corollary}\label{cor:fingen}
Every closed homogeneous ideal in $B$ is finitely generated (as a closed ideal).
\end{corollary}

\begin{remark}
There do exist closed ideals in $A(\mb{B}_d)$ which are not finitely generated 
(one may adjust the example in \cite[Proposition 4.4.2]{RudinPoly}).
\end{remark}

For a closed ideal $J \subseteq B$, \emph{the radical of $J$} is defined 
to be the ideal $\sqrt{J}$ given by 
\bes
\sqrt{J} = \{f \in B : f^n \in J \text{ for some } n\ge1 \}.
\ees

\begin{lemma}\label{lem:radhomhom}
The radical of a closed homogeneous ideal $J$ of $B$ is homogeneous.
\end{lemma}

\begin{proof}
Let $f$ and $m$ be such that $f^m \in J$. 
Write the homogeneous decomposition of $f$ as $f(z) = \sum_{n\geq k} f_n(z)$, 
where $f_k(z)$ is the lowest non-vanishing homogeneous term. 
Then $f^m(z) = f_k(z)^m + \ldots$. 
Since $J$ is homogeneous, $f_k^m \in J$, so $f_k \in \sqrt{J}$. 
Proceeding recursively, we find that $f_j \in \sqrt{J}$ for all $j$.
\end{proof}

\begin{theorem}\label{thm:effNull}
Let $J \subseteq B$ be a closed homogeneous ideal. 
Then there exists $N \in \mb{N}$ such that $f^N \in J$ for all $f \in \sqrt{J}$.
\end{theorem}
\begin{proof}

By the effective Nullstellensatz \cite[Theorem 1.5]{Kol} 
there is an $N \in \mb{N}$ such that $p^N \in J \cap \mb{C}[z]$ 
for all $p \in \sqrt{J \cap \mb{C}[z]} = \sqrt{J} \cap \mb{C}[z]$. 
If $f \in \sqrt{J}$, then $f \in \overline{\sqrt{J} \cap \mb{C}[z]}$ by 
Lemma \ref{lem:Iin}. 
If $\{f_n\}$ is a sequence in $\sqrt{J} \cap \mb{C}[z]$ converging 
to $f$, then $f_n^N \in J$ for all $n$, thus $f^N = \lim_n f_n^N \in J$.
\end{proof}

\begin{corollary}\label{cor:closed}
The radical of a closed homogeneous ideal $J \subseteq B$ is closed.
\end{corollary}

\begin{proposition}\label{prop:radical}
If $I \subseteq \mb{C}[z]$ is radical, then $\overline{I}$ is radical in $B$.
\end{proposition}

\begin{proof}
Put $J = \overline{I}$. Then $\sqrt{J} \cap \mb{C}[z]$ is the unique 
homogeneous ideal in $\mb{C}[z]$ with closure equal to $\sqrt{J}$. 
But $\sqrt{J} \cap \mb{C}[z] = \sqrt{J \cap \mb{C}[z]} = I$, 
so $\sqrt{J} = \overline{I} = J$.
\end{proof}

The main result of this section is a projective Nullstellensatz for closed ideals in $B$. 
We shall need the following notation. For an ideal $J \subseteq B$, we define
\bes
V_\Omega(J) = \{z \in \Omega :  f(z) = 0 \text{ for all }  f \in J \}.
\ees
If $X \subseteq \Omega$, we define
\bes
I_B(X) = \{f \in B :  f(\lambda) = 0 \text{ for all }  \lambda \in X \}.
\ees

\begin{theorem}\label{thm:homNull}
Let $J \subseteq B$ be a closed homogeneous ideal. Then 
\be
\sqrt{J} = I_B(V_\Omega(J)).
\ee
\end{theorem}

\begin{proof}
Define $K = I_B(V_\Omega(J))$.
First, note that $K$ is closed.
Next we show that $K$ is homogeneous. 
Notice that $V_\Omega(J) = V_\Omega(J \cap \mb{C}[z])$, so 
$t  V_\Omega(J) \subseteq V_\Omega(J)$ for all $t \in \mb{D}$. 
Thus if $f \in K$, then for all $\lambda \in V_\Omega(J)$ it follows 
that $f(t \lambda) = 0$. 
By Proposition \ref{prop:homeq}, $K$ is homogeneous. 

Finally, $K \cap \mb{C}[z]$ is the set of all polynomials vanishing 
on 
\[
 V_\Omega(J) = V_\Omega(J \cap \mb{C}[z]) 
 = V(J \cap \mb{C}[z]) \cap \Omega .
\]
So by an easy extension of Corollary \ref{cor:ballvariety}, we find 
\[
 K \cap \mb{C}[z] = \sqrt{J \cap \mb{C}[z]} = \sqrt{J} \cap \mb{C}[z] .
\] 
By Lemma \ref{lem:Iin} and Corollary \ref{cor:closed}, 
\bes
K = \overline{K \cap \mb{C}[z]} = \overline{\sqrt{J} \cap \mb{C}[z]} = \sqrt{J}.
\qedhere
\ees
\end{proof}

\begin{corollary}\label{cor:approxpoly}
Let $I \subseteq \mb{C}[z]$ be a radical homogeneous ideal, and 
let $f \in B$ be a function that vanishes on $V(I) \cap \Omega$. 
Then $f \in \overline{I}$.
\end{corollary}

\begin{proof}
Define $J = \overline{I}$. 
Then, using Theorem \ref{thm:homNull} and then Proposition \ref{prop:radical},
\bes
 f \in I_B(V_\Omega(I)) = I_B(V_\Omega(J)) = \sqrt{J} = J = \overline{I}.
 \qedhere
\ees
\end{proof}

A natural question now is the following: suppose that a function 
$f \in B$ is known to be \emph{small} on $V(I) \cap \Omega$. 
Does it follow that $f$ is \emph{close} to $I$? 
The following proposition shows that this equivalent to an extension problem.

\begin{proposition}
Let $I \subseteq \mb{C}[z]$ be a homogeneous ideal, and 
let $D$ be an algebra of functions on $V_\Omega(I)$ 
that is the closure of the polynomials in some norm that 
satisfies $\|f|_{V_\Omega(I)}\|_D \leq \|f\|_B$. 
Then the following are equivalent. 
\begin{enumerate}
\item For every $g \in D$ there exists an $f \in B$ such that $f|_{V_\Omega(I)} = g$.
\item There exists a constant $C > 0$ such that for all  $f \in B$
\be\label{eq:ineqrestrict}
\dist (f,I) \leq C\left\|f|_{V(I) \cap \Omega} \right\|_{D}.
\ee
\end{enumerate}
\end{proposition}

\begin{proof}
$(1) \Rightarrow (2)$. Define the map $\phi : B \rightarrow D$ by 
$\phi(f) = f|_{V_\Omega(I)}$. 
By Corollary \ref{cor:approxpoly}, $\ker \phi = \overline{I}$. 
Therefore, $\phi$ induces an injective and surjective bounded 
map $\tilde{\phi} : B / \overline{I} \rightarrow D$. 
Therefore $\tilde{\phi}$ 
has a bounded inverse, and that proves (\ref{eq:ineqrestrict}).

$(2) \Rightarrow (1)$. Define $\phi$ and $\tilde{\phi}$ as above. 
Equation (\ref{eq:ineqrestrict}) implies that $\tilde{\phi}$ has closed range. 
But the range of $\tilde{\phi}$ is clearly dense because it contains the polynomials. 
Hence $\phi$ is surjective.
\end{proof}

\begin{remark}\label{rem:slickproof}
Let $I \subseteq \mb{C}[z]$ be a radical homogeneous ideal, and 
let $J$ be the closure of $I$ in $\cA_d$. 
Then both $\cA_I$ and $\cA_d/J$ are the universal unital operator algebras 
generated by a row contraction satisfying the relations in $I$, 
so they are naturally isomorphic. 
In particular, using Proposition \ref{prop:mult}, it follows that for all $f \in \cA_d$,
\bes
\dist (f,I) = \left\Vert f|_{Z^o(I)}\right\Vert_{\Mult(\cF_I)}.
\ees
This gives another proof for Corollary \ref{cor:approxpoly} for the 
special case $B = \cA_d$. 
By the above proposition, it also follows that every function 
that is in the closure of the polynomials on $Z^o(I)$ with respect to
 the multiplier norm on $\cF_I$ is extendable to a function in $\cA_d$.
\end{remark}

\section{Isomorphisms of algebras, biholomorphisms of character spaces, and their rigidity}
\label{sec:isomorphism}

We now turn our attention to algebras that are universal for row contractions
of commuting operators satisfying the relations of a \emph{radical} 
homogeneous ideal $I \subseteq \mb{C}[z]$. 
In this special and important case we will be able to sharpen 
our results in three ways. 
First, we will classify the algebras up to \emph{(completely) isometric isomorphism} 
and also, in many cases, up to \emph{isomorphism}. 
Second, the classifying objects will no longer be subproduct systems 
(or ideals), but rather geometric objects. 
Finally, we will describe the isomorphisms and (completely) isometric isomorphisms of 
the algebras in terms of holomorphic maps of the unit ball in $\mb{C}^d$.

\subsection{Unital homomorphisms are composition operators}

Let $I$ be a radical homogeneous ideal, and let $X = X_I$. 
The algebra $\cA_X$ will be denoted by $\cA_I$. 
Also, the character space $\cM_X$ will be identified with $Z(I)$. 

Recall that by Proposition \ref{prop:mult}, $\cA_I$ can be considered 
as an algebra of functions:
\bes
\cA_I = \{f|_{Z^o(I)} : f \in \cA_d\},
\ees
where the norm is the multiplier norm on the reproducing kernel 
Hilbert space $\cF_I = \overline{\spn}\{\nu_\lambda : \lambda \in Z^o(I)\}$. 

If $I$ and $J$ are radical homogeneous ideals in $\mb{C}[z_1, \ldots, z_d]$ 
and $\mb{C}[z_1, \ldots, z_{d'}]$, respectively, 
then for every algebra homomorphism $\varphi :  \cA_I \rightarrow \cA_J$ 
and every $\rho \in Z(J)$, the composition $\rho \circ \varphi$ is a 
homomorphism from $\cA_I$ into $\mb{C}$.
Therefore it is either a character or it is the functional $0$. 
Thus every unital homomorphism $\varphi : \cA_I \rightarrow \cA_J$ 
gives rise to a mapping $\varphi^* : Z(J) \rightarrow Z(I)$.

\begin{proposition}\label{prop:algiso_biholo}
Let $I$ and $J$ be radical homogeneous ideals in $\mb{C}[z_1, \ldots, z_d]$ 
and $\mb{C}[z_1, \ldots, z_{d'}]$, respectively. 
Let $\varphi : \cA_I \rightarrow \cA_J$ be a unital algebra homomorphism. 
Then there exists a holomorphic map $F: \mb{B}_{d'} \rightarrow \mb{C}^d$ 
that extends continuously to $\overline{\mb{B}}_{d'}$, such that 
\bes
F|_{Z(J)} = \varphi^*.
\ees
The components of $F$ are in $\cA_{d'}$.
Moreover, $\varphi$ is given by composition with $F$, that is
\bes
\varphi(f) = f \circ F \quad , \quad f \in \cA_I .
\ees
\end{proposition}

\begin{proof}
Let $\lambda \in Z(J)$ give rise to the evaluation functional $\rho_\lambda$ 
on $\cA_J$ given by $\rho_\lambda(f) = f(\lambda)$. 
Then $\varphi^*(\rho_\lambda)$ is also an evaluation functional. 
In fact, for the coordinate functions $z_i \in \cA_I$, we find
\bes
[\varphi^*(\rho_\lambda)](z_i) = z_i (\varphi^*(\rho_\lambda)) =
 \rho_\lambda( \varphi(z_i)) = \varphi(z_i) (\lambda). 
\ees
We find that the mapping $\varphi^*$ is given by 
\bes
\varphi^*(\lambda) = (\varphi(z_1) (\lambda), \ldots, \varphi(z_d) (\lambda)).
\ees
Now $\varphi(z_1), \ldots, \varphi(z_d)$ are restrictions to $Z^o(J)$ of 
functions $f_1, \ldots, f_d \in \cA_{d'}$ (see Remark \ref{rem:slickproof}). 
Defining 
\bes
F(z) = (f_1(z), \ldots, f_d(z)), 
\ees
we obtain the required function $F$.
Finally, for every $\lambda \in Z(J)$, 
\bes
\varphi(f) (\lambda) = \rho_\lambda (\varphi(f)) = 
\varphi^*( \rho_\lambda) (f) = \rho_{F(\lambda)} (f)  = f(F (\lambda)),
\ees
so $\varphi(f) = f \circ F$.
\end{proof}

Using the fact that every unital homomorphism is a composition operator, 
together with a standard application of the closed graph theorem, yields 
the following corollary.

\begin{corollary}\label{cor:bounded}
Every unital algebra homomorphism $\varphi : \cA_I \rightarrow \cA_J$ is bounded.
\end{corollary}

\subsection{Some complex geometric rigidity results}

We now follow the discussion in \cite[Chapter 2]{RudinBall} to obtain some 
rigidity results for isomorphisms between the varieties $Z(I)$. 
These rigidity results will help us determine the possibilities for isomorphisms 
between the various algebras $\cA_I$.

\begin{lemma}
Let $I$ be a homogeneous ideal in $\mb{C}[z]$. 
Let $F : \overline{\mb{B}}_d \rightarrow \mb{C}^d$ be a continuous map, 
holomorphic on $\mb{B}_d$, such that $F|_{Z(I)}$ is a bijection of $Z(I)$. 
If $F(0) = 0$ and $\frac{d}{dt}F(tz)\Big|_{t=0} = z$ for all $z \in Z(I)$, 
then $F|_{Z(I)}$ is the identity.
\end{lemma}

\begin{proof}
It seems that a careful variation of the proof for ``Cartan's Uniqueness 
Theorem" given in \cite{RudinBall} (page 23) will work. 
One only needs to use the facts that $Z(I)$ is circular and bounded. 
The reason one must be careful is that $Z(I)$ typically has empty interior.

Let's make sure that it all works.
We write the homogeneous expansion of $F$:
\be\label{eq:Ffstexp}
F(z) = Az + \sum_{n\geq 2} F_n(z),
\ee
where $A = F'(0)$. First let us show that, without loss of generality, 
we may assume
\be\label{eq:Fexpansion}
F(z) = z + \sum_{n\geq 2} F_n(z).
\ee
Let $W$ be the linear span of $Z(I)$, and let $W^\perp$ be its orthogonal 
complement in $\mb{C}^d$.
By the assumption $\frac{d}{dt}F(tz)\Big|_{t=0} = z$ for $z \in Z(I)$, 
so the matrix $A$ can be written as
\bes
A = \begin{pmatrix} I & B \\
0 & C 
\end{pmatrix}
\ees
with respect to the decomposition $\mb{C}^d = W \oplus W^\perp$. 
Replacing $F$ by $F + I_{\mb{C}^d} - A$   we obtain a function that is 
continuous on $\overline{\mb{B}}_{d}$, analytic on $\mb{B}_d$,  
agrees with $F$ on $Z(I)$, and has homogeneous decomposition as in 
(\ref{eq:Fexpansion}).

Following Rudin \cite[bottom of page 23]{RudinBall}, 
we consider the $k$th iterate $F^k$ of $F$:
\bes
F^k(z) = z + kF_2(z) + \ldots .
\ees
Since $Z(I)$ is circular and since $F^k$ maps $Z(I)$ onto itself, 
we find that for all $z \in Z^o(I)$
\bes
k F_2(z) = \frac{1}{2\pi}\int_0^{2 \pi} F^k(e^{i\theta} z) e^{-2 i \theta} d \theta ,
\ees
from which it follows that $\|k F_2(z)\| \leq 1$ for all $k$ and all $z \in Z^o(I)$. 
This implies that $F_2(z) = 0$ for all $z \in Z^o(I)$. 
Therefore there exists a continuous function 
$G : \overline{\mb{B}}_d \rightarrow \mb{C}^d$ that is holomorphic on $\mb{B}_d$ 
and agrees with $F$ on $Z(I)$, that has homogeneous expansion
\bes
G(z) = z + \sum_{n \geq 3}G_n(z),
\ees
(namely, one takes $G = F - F_2$). Note that $G_n = F_n$ for all $n>2$. 
This last observation allows us to repeat the argument inductively and 
deduce that $F(z) = z$ for all $z \in Z^o(I)$.  
By continuity, $F|_{Z(I)}$ equals the identity.
\end{proof}

We now obtain the desired analogue  of Cartan's uniqueness theorem.

\begin{theorem}\label{thm:linear}
Let $I$ and $J$ be homogeneous ideals in $\mb{C}[z_1, \ldots, z_d]$ and 
$\mb{C}[z_1, \ldots, z_{d'}]$, respectively.
Let $F: \overline{\mb{B}}_{d'} \rightarrow \mb{C}^d$ be a continuous map 
that is holomorphic on $\mb{B}_{d'}$ and maps $0$ to $0$. 
Assume that there exists a continuous map 
$G: \overline{\mb{B}}_{d} \rightarrow \mb{C}^{d'}$ that is holomorphic 
on $\mb{B}_{d}$ such that $F \circ G |_{Z(I)}$ and $G \circ F|_{Z(J)}$ 
are the identity maps. 
Then there exists a linear map $A : \mb{C}^{d'} \rightarrow \mb{C}^d$ 
such that $F|_{Z(J)} = A$.
\end{theorem}

\begin{proof}
Again we adjust the proof of \cite[Theorem 2.1.3]{RudinBall} to the current setting. 
The derivatives $F'(0)$ and $G'(0)$ might not be inverses of each other, 
but from $G \circ F(z) = z$, we find that $G'(0) F'(0) z = z$ for all $z \in Z(J)$.

Fix $\theta \in [0,2\pi]$, and define $H: \overline{\mb{B}}_{d'} \to \mb{C}^{d'}$ by
\bes
H(z) = G(e^{-i\theta}F(e^{i\theta}z)) .
\ees
Then $H(0) = 0$ and 
\bes
\frac{d}{dt}H(tz)\Big|_{t=0} = G'(0) e^{-i\theta} F'(0) e^{i\theta}z = z.
\ees
By the previous lemma 
\[ H(z) = z \] 
for $z \in Z(J)$. After replacing $z$ by $e^{-i\theta}z$ and applying 
$F$ to both sides we find that
\bes
F(e^{-i\theta}z) = e^{-i\theta}F(z) \quad\text{for all  } z \in Z(J).
\ees
Integrating over $\theta$, this implies that if (\ref{eq:Ffstexp}) is the 
homogeneous expansion of $F$, then $F_n(z) = 0$ for all $z \in Z^o(J)$ 
and all $n \geq 2$. 
Thus $F|_{Z(J)} = A$.
\end{proof}

The following easy result is a straightforward consequence of homogeneity.

\begin{lemma}\label{lem:isometric}
Let $I$ and $J$ be homogeneous ideals in $\mb{C}[z_1, \ldots, z_d]$ and 
$\mb{C}[z_1, \ldots, z_{d'}]$, respectively.
If a linear map $A : \mb{C}^{d'} \rightarrow \mb{C}^d$ carries $Z(J)$
bijectively onto $Z(I)$, then $A$ is isometric on $V(J)$.
\end{lemma}

\begin{proof}
Each unit vector $v \in V(J)$ determines a disc 
$\ol{\mb D} v = \mb C v \cap \ol{\mb B_{d'}}$ in $Z(J)$.
Observe that $A$ carries $\mb C v$ onto $\mb C Av$, and must take the intersection
with the ball to the corresponding intersection with the ball $\ol{\mb B_d}$.
Thus it takes $\ol{\mb D} v$ onto $\ol{\mb D} Av$.
Therefore $\|Av\| = \|v\|$.
\end{proof}

This lemma can be significantly strengthened to obtain a
rigidity result which will be useful for the algebraic classification 
of the algebras $\cA_I$.

\begin{proposition}\label{prop:rigid}
Let $V$ be a homogeneous variety in $\mb{C}^d$, 
and let $A$ be a linear map on $\mb{C}^d$ such that 
$\|Az\| = \|z\|$ for all $z \in V$. 
If $V = W_1 \cup \cdots \cup W_k$ is the decomposition of 
$V$ into irreducible components, then $A$ is isometric 
on $\spn(W_i)$ for $1 \le i \le k$.
\end{proposition}

\begin{proof}
It is enough to prove the proposition for an irreducible variety $V$. 
The idea of the proof is to produce a sequence of algebraic varieties 
$V \subseteq V_1 \subseteq V_2 \subseteq ...$ such that 
$\|Az\| = \|z\|$ for all $z \in V_i$ and all $i$, 
where either $\dim V_i < \dim V_{i+1}$, or $V_i$ is a subspace 
(and then it is the subspace spanned by $V$).

First, we prove that $\|Ax\| = \|x\|$ for all $x$ lying in the 
tangent space $T_z(V)$ for every $z \in V\setminus\Sing (V)$. 
Since $z$ is nonsingular, for every such $x$ there is a complex 
analytic curve $\gamma: \mb{D} \rightarrow V$ such that 
$\gamma(0) = z$ and $\gamma'(0) = x$. 
By the polar decomposition, we may assume that $A$ is a diagonal matrix 
with nonnegative entries $a_1, \ldots, a_d$. Since $A$ is isometric on $V$,
\bes
\sum_{i=1}^d a_i^2 |\gamma_i(z)|^2 = \sum_{i=1}^d  |\gamma_i(z)|^2 
\quad\text{for  } z \in \mb{D}.
\ees
Applying the Laplacian to both sides of the above equation, and evaluating at $0$, 
we obtain
\bes
\sum_{i=1}^d a_i^2 |\gamma_i'(0)|^2 = 
\sum_{i=1}^d  |\gamma_i'(0)|^2 .
\ees
Thus,  $\|Ax\| = \|x\|$ for all $x \in T_z(V)$ and all nonsingular $z \in V$.

Consider now the set 
\bes
X_0 = \bigcup_{z \in V\setminus\Sing (V)} 
\{z\} \times T_z(V) \subseteq \mb{C}^d \times \mb{C}^d .
\ees
Let $X$ denote the Zariski closure of $X_0$, that is, $X = V(I(X_0))$. 
As $X$ sits inside the tangent bundle $\bigcup_{z\in V}\{z\}\times T_z(V)$, 
$X_0$ is equal to $X \setminus \Big(\Sing (V)\times \mb{C}^d\Big)$.
Therefore $X_0$ is Zariski open in $X$. 
By Proposition 7 of Section 7, Chapter 9 in \cite{CLO92}, the closure 
(in the usual topology of $\mb{C}^{2d}$) of $X_0$ is $X$. 
Letting $\pi$ denote the projection onto the last $d$ variables, 
we have $\pi(X) \subseteq \overline{\pi(X_0)}$. 
But $\pi(X_0) = \bigcup_{z \in V\setminus\Sing (V)}T_z(V)$, 
therefore $\|Ax\| = \|x\|$ for all $x \in \pi(X)$. 
Now, $\pi(X)$ might not be an algebraic variety, 
but by Theorem 3 of Section 2, Chapter 3 in \cite{CLO92}, 
there is an algebraic variety $W$ in which $\pi(X)$ is dense. 
Observe that $W$ must be a homogeneous variety, 
and $\|Az\| = \|z\|$ for every $z \in W$.

Being irreducible, $V$ must lie completely in one of the 
irreducible components of $W$. 
We denote this irreducible component by $V_1$, and let 
$W_2, \ldots, W_m$ be the other irreducible components of $W$. 
We claim: if $V$ itself is not a linear subspace, then $\dim V_1 > \dim V$. 
We prove this claim by contradiction. If $\dim V_1 = \dim V$ 
then $V = V_1$, because $V \subseteq V_1$ and both are irreducible. 
Let $z \in V = V_1$ be a regular point. 
Since $\dim T_z(V) = \dim V$, and $T_z(V)$ is irreducible, 
$T_z(V)$ is not contained in $V_1$. 
But $T_z(V)$ is contained in $W$, thus $T_z(V) \subseteq W_i$ for some $i$. 
But $z \in T_z(V)$ by homogeneity. 
What we have shown is that, under the assumption $\dim V_1 = \dim V$, 
every regular point $z \in V$ is contained in $\bigcup_{i=2}^m W_i$. 
Thus $V_1 \subseteq \cup_i W_i$. That contradicts the assumed 
irreducible decomposition.

If $V$ is not a linear subspace then we are now in the situation in which 
we started, with $V_1$ instead of $V$, and with $\dim V_1 > \dim V$. 
Continue this procedure finitely many times to obtain a  sequence of 
irreducible varieties $V_1 \subseteq \ldots  \subseteq V_n$ that 
terminates at a subspace on which $A$ is isometric. 
$V_n$ must be $\spn V$.  Indeed, it certainly contains $V$. 
On the other hand, every $V_i$ lies in $\spn V_{i-1}$ and hence in $\spn V$.
\end{proof}

When the variety $V$ is a hypersurface we sketch a more elementary proof, 
which provides somewhat more information. 

\begin{proposition}\label{prop:eigenvalueanalysis}
Let $f \in \mb{C}[z_1, \ldots, z_d]$ be a homogeneous polynomial, 
and let $V = V(f)$. 
Let $A$ be a linear map on $\mb{C}^d$ such that $\|Az\| = \|z\|$ for all $z \in V$. 
Let $A = UP$ be the polar decomposition of $A$ with $U$ unitary and $P$ positive. 
Then one of the following possibilities holds:
\begin{enumerate}
\item $P=I$;
\item $P$ has precisely one eigenvalue different from $1$ and $V(f)$ is a hyperplane;
\item\label{it:3} $P$ has precisely two eigenvalues not equal to $1$ 
$($one larger and one smaller$)$, and in this case $V$ is the union of  
hyperplanes which all intersect in a common $d\!-\!2$-dimensional subspace.
\end{enumerate}
\end{proposition}

\begin{proof}
After a unitary change of variables, we may assume that $A$ is a 
positive diagonal matrix $A = \diag(a_1, \ldots, a_d)$ 
with $a_i \ge a_{i+1}$ for $1 \le i < d$.
Now $A$ takes the role of $P$ in the statement. 

We first show that $a_2 = \dots = a_{d-1} = 1$.
For if $a_1 \ge a_2 > 1$, there is a non-zero solution to
$f=0$ and $z_3 = \dots = z_d = 0$, say $v=(z_1,z_2,0,\dots,0)$.
But $\|Av\|>\|v\|$, contrary to the hypothesis.  Hence $a_2 \le 1$.
Similarly one shows that $a_{d-1}\ge 1$.  Hence all singular
values equal 1 except possibly $a_1>1$ and $a_d<1$.

If $A = I$ then we have (1). 
When there is precisely one eigenvalue different from $1$,
$A$ is only isometric on the hyperplane $\ker (A-I)$; 
thus (2) holds.
So we may assume that there are precisely two singular values different from $1$,
$a_1>1>a_d$. Then $f$ must have the form $f = \alpha z_1^m + \ldots$ 
for some $\alpha \neq 0$. Indeed, otherwise (if $z_1$ appears only in 
mixed terms) there is non-zero solution $v = (1,0,\ldots,0)$ to $f = 0$, 
and $\|Av\|>\|v\|$, contrary to the hypothesis.
Now there are two cases:

\noindent {\bf Case 1:} $f$ does not depend on $z_2, \ldots, z_{d-1}$. 
In this case $f$ is essentially a polynomial in two variables, 
and can therefore be factored as $f = \prod_i (\alpha_i z_1 + \beta_i z_d)$, 
from which case (\ref{it:3}) follows.

\noindent {\bf Case 2:} $f$ depends on $z_2, \ldots, z_{d-1}$. 
Say $f$ depends on $z_2$. 
Fix $z_3, \ldots, z_d$ such that the polynomial $f(\cdot,\cdot,z_3, \ldots, z_d)$ 
still depends on $z_2$. 
For every $z_2$ there is a solution $z_1$ to the equation $f(z_1,z_2,\ldots,z_d) = 0$. 
As $z_2$ tends to $\infty$, the form of $f$ forces $z_1$ to tend to $\infty$ as well. 
But since $(z_1, \ldots, z_d)$ is a solution and $A$ is isometric on $V(f)$, 
one has 
\bes
a_1^2 |z_1|^2 + a_d^2 |z_d|^2 = |z_1|^2 + |z_d|^2.
\ees
This cannot hold when $z_d$ is fixed and $z_1$ tends to $\infty$. 
So this case does not occur.
\end{proof}

\begin{example}
Let us show that arbitrarily many hyperplanes can appear in case (\ref{it:3}) above. 
Let $a,b>0$ be such that $a^2 + b^2 = 2$, 
and let $\lambda_1, \ldots, \lambda_k \in \mb{T}$. 
Let $V = \ell_1 \cup \cdots \cup \ell_k$, 
where $\ell_i = \mb C(\lambda_i/\sqrt{2}, 1/\sqrt{2})$. 
Then $A = \diag (a,b)$ is isometric on $V$. 
\end{example}

\begin{example}
Propositions \ref{prop:rigid} and \ref{prop:eigenvalueanalysis} depend 
on the fact that we are working over $\mb{C}$. 
Indeed, consider the cone $V = V(x^2 + y^2 - z^2)$ over $\mb{R}$. 
With $a$ and $b$ as in the previous example, one sees that 
$A = \diag (a,a,b)$ is isometric on $V$, 
but it is clearly not an isometry on $\mb{R}^3 = \spn(V)$.
\end{example}

\subsection{Algebra isomorphisms induced by linear maps}

Let $I$ and $J$ be radical homogeneous ideals. 
We know that for $\cA_I$ and $\cA_J$ to be isomorphic there must be 
a linear map $A: \overline{\mb{B}}_{d'} \rightarrow \mb{C}^d$ taking $Z(J)$ 
bijectively onto $Z(I)$ (see Remark \ref{rem:trick} below). 
Our goal now is to show the converse, that is, the existence of such a 
linear map gives rise to an isomorphism of the algebras via a similarity, 
which we establish for a certain class of varieties.

Let $V$ be a homogeneous variety in $\mb{C}^d$ and let
$V= V_1\cup\cdots \cup V_k$ be the decomposition of $V$ into irreducible components. 
Then we call 
\[
 S(V) := \spn(V_1)\cup\cdots \cup \spn(V_k)
\]
the \textit{minimal subspace span of $V$}.
By Proposition~\ref{prop:rigid}, the linear map $A$ must be isometric on $S(V)$.
Note that $V=S(V)$ if and only if $V$ is already the union of subspaces. 

Our goal is to establish that $A$ induces a bounded linear isomorphism $\tilde A$ 
between the Fock spaces $\cF_J$ and $\cF_I$ given by $\tilde A f = f \circ A^*$.
This is evidently linear (provided it is defined) and satisfies
\be\label{eq:Atilde}
\tilde{A} \nu_\lambda = \nu_{A \lambda} \quad \text{for } \lambda \in Z^o(J).
\ee
Conversely, $\tilde A$ is determined by (\ref{eq:Atilde}) because the 
kernel functions span $\cF_J$.

Before describing the class of varieties for which we can establish this
very natural sounding fact, we prove it in several special cases which
will form the building blocks for the general result.

Let $L$ be a Hilbert space. 
Let $S_n$ denote the symmetric group on $n$ elements. 
For $\sigma \in S_n$, we let $\pi_\sigma$ be the unitary operator 
on $L^{\otimes n}$ given by 
\[
 \pi_\sigma(x_1 \otimes \cdots x_n) = 
 x_{\sigma(1)} \otimes \cdots \otimes x_{\sigma(n)} .
\] 
Then $E_n = \frac{1}{n!}\sum_{\sigma \in S_n} \pi_\sigma $ is the 
orthogonal projection of $L^{\otimes n}$, the $n$-fold tensor product,  
onto $L^n$, the symmetric $n$-fold tensor product. 
If $W \subseteq L$ is a subspace, then $W^n = E_n W^{\otimes n}$ is the 
symmetric $n$-fold tensor product of $W$. 
If $V$ is another subspace, then we write $V^mW^n$ for the 
subspace $E_{m+n} (V^m \otimes W^n) \subseteq (\mb{C}^{d'})^{m+n}$,
which is the symmetric tensor product of $V^m$ and $W^n$.

If $P_V$ is the orthogonal projection of $L$ onto $V$, then
$P_V^{\otimes n}$ is the projection of  $L^{\otimes n}$ onto $V^{\otimes n}$.
The orthogonal projection onto $V^n$ is given by
$P_{V^n} = E_n P_V^{\otimes n} \iota$ where $\iota$ is the natural injection of
$L^n$ into $L^{\otimes n}$.

We need the following lemma which shows that high tensor powers of
disjoint subspaces are almost orthogonal.

\begin{lemma}\label{lem:asymptotic_orthogonality}
Let $V_i$ for $1 \le i \le k$ be subspaces of a Hilbert space $L$
so that $\max_{i\ne j} \|P_{V_i} P_{V_j} \| = c < 1$.
When $c^n \le 1/2k$, any vectors $x_i \in V_i^n$ satisfy
\[
 \frac12 \sum_{i=1}^k \|x_i\|^2 \le 
 \big\| \sum_{i=1}^k x_i \big\|^2 \le 
 \frac 32 \sum_{i=1}^k \|x_i\|^2 .
\]
\end{lemma}

\begin{proof}
Observe that 
\[
 \|P_{V_i^n} P_{V_j^n}\| \le 
 \| P_{V_i}^{\otimes n} P_{V_j}^{\otimes n}\| = 
 c^n \le \frac1{2k} .
\]
Therefore
\begin{align*}
 \Big| \big\| \sum_{i=1}^k x_i \big\|^2 - \sum_{i=1}^k \|x_i\|^2 \Big| &= 
 \Big| \sum_{i \ne j} \ip{x_i,x_j} \Big| \le
 \sum_{i \ne j} |\ip{x_i,x_j}| \\ &\le 
 \sum_{i \ne j} c^n \|x_i\| \, \|x_j\| \le
 c^n \Big( \sum_{i=1}^k \|x_i\| \Big)^2 \\&\le 
 c^nk \sum_{i=1}^k \|x_i\|^2
 \le \frac12 \sum_{i=1}^k \|x_i\|^2 .
 \qedhere
\end{align*}
\end{proof}

Recall that $\cF(X)$ is a reproducing kernel Hilbert spaces and 
that $\nu_\lambda$ denotes the kernel function at $\lambda$. 
We introduce a convenient basis for the symmetric Fock
space $\cF(X)$ of a subspace $X$. 
Decompose $\nu_\lambda$ into its homogeneous parts
\[
 \nu_\lambda = \sum_{n\ge0} \nu_\lambda^n = \sum_{n\ge0} \lambda^{\otimes n} .
\]
Thus if $f = \sum_n f_n$ is the homogeneous decomposition of $f \in H^2_d$,
\[
 \nu_\lambda^n(f) = \lel f_n, \lambda^{\otimes n} \rir = f_n(\lambda) .
\]
This functional is completely determined by the identity
\[
 \nu_\lambda^n(z^n) = \lel z^n, \lambda^{\otimes n} \rir = 
 \sum_{|\alpha|=n} \frac{n!}{\alpha_1! \dots \alpha_d!} 
 \overline{\lambda^\alpha} z^\alpha .
\]
For any subspace $X$, 
\begin{align*}
 \cF(X) &= \spn\{\nu_\lambda : \lambda \in \mb B_d \cap X \} \\&= 
 \sum_{n\ge0}\strut^\oplus \spn\{ \nu_\lambda^n : \lambda\in \mb B_d \cap X \}
 =  \sum_{n\ge0}\strut^\oplus X^n .
\end{align*}

\begin{lemma}\label{lem:disjoint}
Let $V = V_1 \cup \cdots \cup V_k$ and $W = W_1 \cup \cdots \cup W_k$
be unions of linear subspaces in $\mb C^{d'}$ and $\mb C^d$, respectively,
with zero intersections $V_i\cap V_j = \{0\}$ and $W_i\cap W_j = \{0\}$ for $i \ne j$.  
Suppose that $A$ is a linear map from $\mb C^{d'}$ to $\mb C^d$ 
such that $A(W_i) = V_i$ and  $A$ is isometric on each of the $W_i$'s.
Then $\tilde A$, defined by $\tilde A \nu_\lambda = \nu_{A\lambda}$,
determines a bounded linear map of $\cF(W)$ into $\cF(V)$.
\end{lemma}

\begin{proof}
For any variety $V$ that is a union of subspaces, $V=V_1\cup\cdots\cup V_k$, 
\[
 \cF(V) = \sum_{i=1}^k \cF(V_i) = 
 \sum_{n\ge 0}\strut^\oplus (\sum_{i=1}^k V_i^n) .
\]
For $f \in \cF(W)$, $\tilde{A} f = f \circ A^*$. 
In particular, 
\[
\tilde A \nu_\lambda^n = \nu_{A\lambda}^n = A^{\otimes_s n} \nu_\lambda^n .
\]
That is,  $\tilde A|_{(\mb C^{d'})^n} = A^{\otimes_s n}$ is the
symmetric tensor product of $n$ copies of $A$.

In particular, on any subspace $X$ on which $A$ is isometric,
$\tilde A$ is a unitary map of $\cF(X)$ onto $\cF(AX)$.
In particular, $\tilde A$ carries $\cF(W_i)$ isometrically onto $\cF(V_i)$
for $1 \le i \le k$.
The only issue is whether this defines a bounded linear map on their span.
Since $\tilde A$ respects the homogeneous decomposition, it suffices to
consider the restriction of $A^{\otimes_s n}$ to $\sum_{i=1}^k W_i^n$.
We will write $W^n := \sum_{i=1}^k W_i^n$.

Since $W_i \cap W_j = \{0\}$ for $i < j$, and $d'<\infty$, the projections onto these 
subspaces satisfy $\| P_{W_i} P_{W_j} \| < 1$. Thus we can define
\[
 c = \max\{ \| P_{W_i} P_{W_j} \|, \, \| P_{V_i} P_{V_j} \| : 1 \le i < j \le k \} < 1 .
\]
We consider two cases.
Observe that 
\[ \| \tilde A|_{W^n}\| \le \| A^{\otimes_s n} \| =  \|A\|^n . \] 
When $c^n > 1/2k$, $n \le N:= \log_{c^{-1}}(2k)$, and so we obtain 
\[
 \| \tilde A |_{W^n} \| \le \|A\|^N 
 \quad\textrm{ provided }n \le N.
\]

When $c^n \le 1/2k$, we use Lemma~\ref{lem:asymptotic_orthogonality}.
By hypothesis $W_i \cap W_j = \{0\}$ for $i\neq j$.
A typical vector in $W^n = \sum_{i=1}^k W_i^{n}$ can be written as
$x = \sum_{i=1}^k x_i$ where $x_i \in W_i^{n}$.
It follows from Lemma~\ref{lem:asymptotic_orthogonality} that 
\[
 \frac12 \sum_{i=1}^k \|x_i\|^2 \le 
 \big\| \sum_{i=1}^k x_i \big\|^2 \le 
 \frac 32 \sum_{i=1}^k \|x_i\|^2 .
\]
Lemma~\ref{lem:asymptotic_orthogonality} also applies to
$A^{\otimes_s n} x = \sum_{i=1}^k  A^{\otimes_s n} x_i$ in $\sum_{i=1}^k V_i^n$, namely
\[
 \frac12 \sum_{i=1}^k \| A^{\otimes_s n} x_i\|^2 \le 
 \big\| \sum_{i=1}^k  A^{\otimes_s n} x_i \big\|^2 \le 
 \frac 32 \sum_{i=1}^k \| A^{\otimes_s n} x_i\|^2 .
\]
However $A$ is isometric on each $W_i$, and thus
$\| A^{\otimes_s n} x_i \| = \|x_i\|$.
We deduce that for any vector $x \in W^n$, we have
\[ \frac13 \|x\|^2 \le \| A^{\otimes_s n}  x \|^2 \le 3 \|x\|^2 .\]
In particular, $\| \tilde A|_{W^n} \| \le \sqrt 3$.

Putting the pieces together, we see that
\[ \|\tilde A\| \le \max\{ \|A\|^N, \sqrt{3} \} .\]
Hence $\tilde A$ is a bounded linear map of $\cF(W)$ into $\cF(V)$.
\end{proof}

If $W=W_1\cup\cdots \cup W_k$ is a union of subspaces and $E$ is a subspace 
orthogonal to each of the $W_i$'s, then we let $E\oplus W$ denote
$(E\oplus W_1) \cup \cdots \cup (E\oplus W_k)$.

\begin{lemma}\label{lem:direct_sum}
Suppose that $V = V_1 \cup \cdots \cup V_k$ and $W = W_1 \cup \cdots \cup W_k$
are unions of linear subspaces; and  
$A$ is a linear map from $\mb C^{d'}$ to $\mb C^d$ 
such that $A(W_i) = V_i$ and  $A$ is isometric on each of the $W_i$'s.
Furthermore suppose that $\tilde A$, defined by $\tilde A \nu_\lambda = \nu_{A\lambda}$,
determines a bounded linear map of $\cF(W)$ into $\cF(V)$.
If $E$ is a subspace orthogonal to $\spn(V)$ and $F$ is a subspace orthogonal to 
$\spn(W)$ such that $A$ carries $F$ isometrically onto $E$, 
then $\tilde A$ determines a bounded linear map of 
$\cF(F \oplus W)$ into $\cF(E \oplus V)$.
\end{lemma}

\begin{proof}
This is straightforward.  
If $F$ and $X$ are orthogonal subspaces, 
\[
 \cF(F \oplus X) = \sum_{m,n \ge0}\!\!\!\strut^\oplus F^m X^n = 
 \sum_{n\ge0}\strut^\oplus \cF(F) X^n .
\]
If $A$ is isometric on $F \oplus X$ and $AF=E$ and $AX=Y$,
then it follows that $\tilde A$ is an isometry of $\cF(F \oplus X)$
onto $\cF(E \oplus Y)$ which takes $\cF(F)X^n$ isometrically
onto $\cF(E) Y^n$.  
Moreover if $A|_F = U$ is the isometry onto $E$, the restriction
of $\tilde A$ to $\cF(F)X^n$ is $\tilde U \otimes_s A^{\otimes_s n}|_{X^n}$.

This situation applies to each space $F\oplus W_i$.
Hence $\tilde A$ carries $\cF(F) \sum_{i=1}^k W_i^n$ onto
$\cF(E) \sum_{i=1}^k V_i^n$ via 
\[ \tilde U \otimes_s A^{\otimes_s n}|_{\sum_{i=1}^k W_i^n} .\]
Since $\tilde U$ is isometric, the norm of this map coincides with 
\[ \| A^{\otimes_s n}|_{\sum_{i=1}^k W_i^n} \| \le \|\tilde A|_{\cF(W)} \| .\]
It follows that $\| \tilde A|_{\cF(F\oplus W)} \| = \|\tilde A|_{\cF(W)} \| $.
\end{proof}

\begin{corollary}\label{cor:common_intersection}
Let $V=V_1 \cup \cdots \cup V_k$ and $W = W_1 \cup \cdots \cup W_k$
be homogeneous varieties decomposed into irreducible components. 
Suppose that $A$ is a linear map from $\mb C^{d'}$ to $\mb C^d$ 
such that $A(W_i) = V_i$ and  $A$ is isometric on each of the $W_i$'s.
If there is a common subspace $E$ so that $S(V_i) \cap S(V_j) = E$ for $i \ne j$,
then $\tilde A$ determines a bounded linear map of $\cF(W)$ into $\cF(V)$.
\end{corollary}

\begin{proof}
By Proposition~\ref{prop:rigid}, $A$ maps the minimal subspace span
$S(W_i)$ isometrically onto $S(V_i)$ for $1 \le i \le k$.  
In particular, $F := S(W_i) \cap S(W_j)$ is independent of $i \ne j$,
and is mapped isometrically onto $E$. 
Let $V'_i = S(V_i) \ominus E$ and $W'_i = S(W_i) \ominus F$. 
As these are disjoint subspaces, Lemma~\ref{lem:disjoint} implies that 
$\tilde A$ is a bounded map of $\cF(W'_1 \cup \cdots \cup W'_k)$
into $\cF(V'_1 \cup \cdots \cup V'_k)$.
Then by Lemma~\ref{lem:direct_sum}, this extends to a bounded map of
$\cF(S(W))$ into $\cF(S(V))$.
The restriction of this map to $\cF(W)$ is a bounded map into $\cF(V)$.
\end{proof}

A third construction is obtained by using the ideas in 
Proposition~\ref{prop:eigenvalueanalysis}.

\begin{lemma}\label{lem:codim_1}
Let $V=V_1 \cup \cdots \cup V_k$ and $W = W_1 \cup \cdots \cup W_k$
be homogeneous varieties decomposed into irreducible components. 
Suppose that $A$ is a linear map from $\mb C^{d'}$ to $\mb C^d$ 
such that $A(W_i) = V_i$ and  $A$ is isometric on each of the $W_i$'s.
If $\dim \big( \spn(W)/S(W_1) \big) \le 1$, then $\tilde A$ 
determines a bounded linear map of $\cF(W)$ into $\cF(V)$.
\end{lemma}

\begin{proof}
If $S(W_1) = \spn(W)$, then $A$ is an isometry of $\spn(W)$ onto $\spn(V)$.
In this case, $\tilde A$ is an isometry of $\cF(W)$ onto $\cF(V)$.
So we may suppose that $S(W_1)$ is codimension 1 in $\spn(W)$.

As in the proof of Proposition~\ref{prop:eigenvalueanalysis},
the restriction of $A$ to $\spn(W)$ has singular values 
 $a_1 \ge 1 = a_2 = \dots = a_{p-1} \ge a_p$.
And $A$ will be isometric on $\spn(W)$ as in cases (1) and (2)
of Proposition~\ref{prop:eigenvalueanalysis}, unless $a_1 > 1 > a_p$.  
So we assume that we are in this situation.
Let $f_1,\dots,f_p$ be the orthonormal basis for $\spn(W)$ so that
there is a corresponding orthonormal basis $e_1,\dots,e_p$ for $\spn(V)$ 
with $A f_j = a_j e_j$. 

There is a unique $\alpha \in (0,\pi/2)$ so that
\[ a_1^2 \cos^2\alpha + a_p^2 \sin^2\alpha = 1 .\]
The maximal subspaces on which $A$ is isometric have the form
\[
 W_\theta = 
 \spn\{ \cos\alpha f_1 + e^{i\theta}\sin\alpha f_p, f_2,\dots f_{p-1} \}
 \quad\text{for }\theta \in [0,2\pi) .
\]
Each irreducible component $W_i$ is contained in some $W_{\theta_i}$.
By Corollary~\ref{cor:common_intersection}, $\tilde A$ is bounded 
on $\cF(W_{\theta_1} \cup \dots \cup W_{\theta_k})$.
Hence it restricts to a bounded map of $\cF(W)$ into $\cF(V)$.
\end{proof}

\begin{corollary}\label{cor:dim3}
Let $V$ and $W$ be homogeneous varieties in $\mb C^3$.
If there is a linear map $A$ on $\mb C^3$ such that $A(W)=V$
and $A$ is isometric on the irreducible components of $W$,
then $\tilde A$ is a bounded linear map of $\cF(W)$ into $\cF(V)$.
\end{corollary}

\begin{proof}
Let $W = W_1 \cup \cdots \cup W_k$.  If $\dim S(W_i) > 1$ for any $i$,
then Lemma~\ref{lem:codim_1} applies.
Otherwise each $W_i$ is a subspace of dimension one.
In this case, $W_i\cap W_j = \{0\}$ when $i \ne j$.
Hence Lemma~\ref{lem:disjoint} applies.
\end{proof}

We are now in a position to state the class of varieties to which our techniques apply.
We introduce a definition for the purposes of easier exposition.
Call a variety $V$ \textit{tractable} if $W = S(V)$ is tractable, meaning
that it can be constructed as follows:
\begin{enumerate}
\item A finite union $W$ of subspaces $W_i$ with zero intersection, 
 $W_i\cap W_j = \{0\}$ for $i \ne j$, is tractable.
\item A finite union $W$ of subspaces $W_i$ so that  $\dim \big( \spn W/W_{i_0} \big) = 1$ 
 for some $i_0$ is tractable.
\item If $W$ is tractable, and $E$ is a subspace orthogonal to $\spn W$,
then $E\oplus W$ is tractable.
\item If $W_i$ for $1 \le i \le k$ are tractable unions of subspaces and 
 $\spn(W_i) \cap \spn(W_j) = \{0\}$ for $i \ne j$,
 then $W_1 \cup \dots \cup W_k$ is tractable. 
\end{enumerate}

The crucial technical result we need is the following:

\begin{theorem}\label{thm:Atilde}
Let $I$ and $J$ be radical homogeneous ideals in $\mb{C}[z_1, \ldots, z_d]$ 
and $\mb{C}[z_1, \ldots, z_{d'}]$, respectively. 
Assume that $V(J)$ is tractable.
If there is a linear map $A: \mb{C}^{d'} \rightarrow \mb{C}^d$  
that maps $Z(J)$ bijectively onto $Z(I)$, 
then the map $\tilde{A} : \cF_J \rightarrow \cF_I$ given by
$(\ref{eq:Atilde}):$
\[
 \tilde{A} \nu_\lambda = \nu_{A \lambda} \quad \text{for  } \lambda \in Z^o(J)
\]
is a bounded linear map of $\cF_J$ into $\cF_I$.
\end{theorem}

\begin{proof}
By Lemma~\ref{lem:isometric}, $A$ preserves the norm on $V(J)$.
By Proposition~\ref{prop:rigid}, $A$ is also isometric on the minimal 
subspace span $S(V(J))$.
If we can show that $\tilde A$ is a bounded map of $\cF(S(V(J)))$ 
into $\cF(S(V(I)))$, then by restriction, it maps $\cF(V(J))$ into $\cF(V(I))$.
So the theorem reduces to the case in which the varieties are unions of subspaces.

Lemma~\ref{lem:disjoint} shows that the result holds in case (1) of
a union of subspaces with zero pairwise intersection.
Lemma~\ref{lem:codim_1} shows that the result holds in case (2)
in which one subspace $W_{i_0}$ has codimension one in $\spn W$.
And Lemma~\ref{lem:direct_sum} shows that if the result holds for
$W$, then it holds for $E \oplus W$ when $E$ is orthogonal to $\spn W$.
Thus it remains to show that if $W_i$ for $1 \le i \le k$ are tractable 
unions of subspaces and $\spn(W_i) \cap \spn(W_j) = \{0\}$ for $i \ne j$,
then the result holds for $W_1 \cup \dots \cup W_k$.
The proof is a refinement of the proof of Lemma~\ref{lem:disjoint}.

The hypotheses guarantee that $\tilde A$ is a bounded linear map
of $\cF(W_i)$ into $\cF(V_i)$ for $1 \le i \le k$.
As in the proof of Lemma~\ref{lem:disjoint}, it suffices to estimate
$\| \tilde A|_{W^n} \|$ for each $n\ge 0$.
Again we let 
\[
 c = 
 \max \{ \| P_{\spn(W_i)} P_{\spn(W_j)} \|,\, \| P_{\spn(V_i)} P_{\spn(V_j)} \| 
 : 1 \le i < j \le k \} < 1.
\]
The proof that $\| \tilde A|_{W^n} \| \le \| A^{\otimes_s n}\| \le \|A\|^N$
for $n \le N := \log_{c^{-1}}(2k)$ remains the same.
So we consider $\| \tilde A|_{W^n} \|$ for $n>N$.

Following the proof of Lemma~\ref{lem:disjoint} again, we split a
typical vector $x \in W^n$ as $x = \sum_{i=1}^n x_i$ with 
$x_i \in W_i^n \subset \spn(W_i)^n$.  
As before, Lemma~\ref{lem:asymptotic_orthogonality} yields
\[
 \frac12 \sum_{i=1}^k \|x_i\|^2 \le 
 \big\| \sum_{i=1}^k x_i \big\|^2 \le 
 \frac 32 \sum_{i=1}^k \|x_i\|^2 
\]
and
\[
 \frac12 \sum_{i=1}^k \| A^{\otimes_s n} x_i\|^2 \le 
 \big\| \sum_{i=1}^k  A^{\otimes_s n} x_i \big\|^2 \le 
 \frac 32 \sum_{i=1}^k \| A^{\otimes_s n} x_i\|^2 .
\]
Let  $M = \max \big\{ \|\tilde A|_{\cF(W_i)} \| : 1 \le i \le k \big\}$.
Then
\begin{align*}
 \big\| \sum_{i=1}^k  A^{\otimes_s n} x_i \big\|^2 &\le 
 \frac 32 \sum_{i=1}^k \| A^{\otimes_s n} x_i\|^2  \\ &\le
 \frac 32 M^2  \sum_{i=1}^k \|x_i\|^2 \le
 3   M^2  \big\| \sum_{i=1}^k x_i \big\|^2 .
\end{align*}
Hence $\| \tilde A \| \le \max\{ \|A\|^N,  \sqrt3 M \}$ on $\cF(W)$.
Thus $\tilde A$ is a bounded map of $\cF(W)$ into $\cF(V)$.
\end{proof}  

To recapitulate, we list a number of examples of tractable varieties:
\begin{enumerate}
\item Any irreducible variety $V$ because $S(V)$ is a subspace.
\item $V= V_1 \cup V_2$, the union of two irreducible varieties,
 because there is only one $S(V_i) \cap S(V_j)$ for $i \ne j$.
\item $V=V_1 \cup \cdots \cup V_k$ where $V_i$ are irreducible
 and $S(V_i) \cap S(V_j) = E$, a fixed subspace, for all $i \ne j$.
\item $V=V_1 \cup \cdots \cup V_k$ where $\dim S(V_1) \ge d-1$.
\item Any variety in $\mb C^3$.
\end{enumerate}

As an immediate consequence, we obtain the following statement about
isomorphism of operator algebras of the form $\cA_I$ when $V(I)$ is tractable.
We conjecture that this result is valid for all homogeneous varieties.

\begin{theorem}\label{thm:linear_induce_similarity}
Let $I$ and $J$ be radical homogeneous ideals in $\mb{C}[z_1, \ldots, z_d]$ and\\
$\mb{C}[z_1, \ldots, z_{d'}]$, respectively, such that $V(J)$ is tractable. 
Let $A : \mb{C}^{d'} \rightarrow \mb{C}^d$ and 
$B: \mb{C}^d \rightarrow \mb{C}^{d'}$ be linear maps such that 
$AB|_{Z(I)} =\id_{Z(I)}$ and $BA|_{Z(J)} = \id_{Z(J)}$. 
Let  $\tilde{A}$ be the map given by Theorem~$\ref{thm:Atilde}$.
Then $\tilde A$ is invertible, and the map
\bes
 \varphi: f \rightarrow f \circ A
\ees
is a completely bounded isomorphism from $\cA_I$ onto $\cA_J$,
and it is given by conjugation with $\tilde{A}^*$:
\bes
 \varphi(f) = \tilde{A}^* f (\tilde{A}^{-1})^*.
\ees
\end{theorem}

\begin{proof}
By Theorem~\ref{thm:Atilde}, $\tilde{A}$ and $\tilde B$ are bounded.
By checking the products on the kernel functions,
it follows easily that $\tilde{B} = \tilde{A}^{-1}$.
So these maps are linear isomorphisms.

Let $f \in \cA_I$ and $\lambda \in Z(J)$. 
Denote by $M_f$ the operator of multiplication by $f$ on $\cF_I$. Then 
\bes
 \tilde{A}^{-1} M_{f}^* \tilde{A} \nu_\lambda =
 \tilde{A}^{-1} M_{f}^* \nu_{A \lambda} = 
 \tilde{A}^{-1} \overline{f \circ A (\lambda)}\nu_{A\lambda} = 
 \overline{f \circ A (\lambda)}\nu_{\lambda}.
\ees
Thus $(\tilde{A}^{-1} M_{f}^* \tilde{A})^* = \tilde{A}^* M_f (\tilde{A}^{-1})^*$ 
is the operator on $\cF_J$ given by multiplication by $f\circ A$.
\end{proof}

\begin{remark} \label{rem:A not injective} 
The various lemmas established above only require that $A$ be length  preserving on $V$.  
It need not be invertible on $\spn(V)$ in order to show that the map $\tilde A$ is bounded.  
However, if $A$ is singular on $\spn(V)$, then $\tilde A$ is not injective because the 
homogeneous part of order one,   
$M_1 := \spn\{ \nu_\lambda^1 : \lambda \in Z^o(V) \} \simeq \spn(V)$ 
and $\tilde A|_{M_1} \simeq A$. 

For   example, if $V = \mb C e_1 \cup \mb C e_2 \cup \mb C e_3$ and
$A =  \begin{sbmatrix}1&0&1/\sqrt2\\0&1&1/\sqrt2\\0&0&0\end{sbmatrix}$,  
\vspace{.3ex} 
then one can see that $A$ is isometric on $V$ and maps $\mb C^3$ 
into $\spn\{e_1,e_2\}$, taking $V$ to the union  of three lines in 2-space. 
The map $\tilde A$ is bounded, and satisfies 
$\tilde A \nu_\lambda = \nu_{A\lambda}$ for  $\lambda \in Z^o(V)$. 
But for the reasons mentioned in the previous paragraph, it is not injective. 

On the other hand,  if $A$ is bounded below by $\delta > 0$ on $\spn V$, 
one can argue in each of the various lemmas that  $A^{\otimes_s n}$ is 
bounded below by $\delta^n$ for $n \le N$ and use the original arguments 
for upper and lower  bounds on the higher degree terms. 
In this way, one sees directly that $\tilde A$ is an isomorphism. 
\end{remark}

Although the following example does not disprove Theorem~\ref{thm:Atilde} 
for arbitrary complex algebraic varieties, 
it does illustrate some of the difficulties one must overcome.

\begin{example}
In this example we identify $\mb{C}^2$ with $\mb{R}^4$. 
Let 
\[ V = \{(w,x,y,z) : w^2 + x^2 = y^2 + z^2\} .\]
Then $V$ is a real algebraic variety in $\mb{R}^4$, but is not a complex 
algebraic variety in $\mb{C}^2$ because it has odd real dimension. 
Note that 
\bes
 V =  \bigcup_{\theta \in \mb{T}} 
 \left\{ \lambda \left(\frac{1}{\sqrt{2}},\frac{\theta}{\sqrt{2}} \right) 
 : \lambda \in \mb{C}\right\}.
\ees
Let $A = \bigl( \begin{smallmatrix} 
a&0\\ 0&b 
\end{smallmatrix} \bigr) 
$, where $a> 1 > b >0$ satisfy $a^2 + b^2 = 2$. 
Then $A$ is an invertible linear map that preserves the lengths of vectors in $V$. 
Put $V' = AV$. 
We will now show that the densely defined operator given by 
$\tilde{A} \nu_\lambda = \nu_{A_\lambda}$ does not extend to a bounded map taking
$
\overline{\spn} \{\nu_\lambda : \lambda \in V \cap \mb{B}_2\}
$
into 
$
\overline{\spn} \{\nu_\lambda : \lambda \in V' \cap \mb{B}_2\}.
$
Let $\alpha, \beta >0$, and consider 
\bes
\sum_{j=1}^n (\alpha e_1 + \theta_j  \beta e_2)^n \in (\mb{C}^2)^n,
\ees
where $\theta_j = \exp(\frac{2\pi i}{n}j)$. We find 
\begin{align*}
\sum_{j=1}^n (\alpha e_1 + \theta_j  \beta e_2)^n 
&= \sum_{j=1}^n \sum_{k=0}^n (\alpha e_1)^k (\theta_j  \beta e_2)^{n-k} \\
&=  \sum_{k=0}^n \alpha^k \beta^{n-k}  (e_1)^k 
\Big(\sum_{j=1}^n \theta_j^{n-k} (e_2)^{n-k}\Big) \\
&= \beta^n n  e_2^n + \alpha^n n e_1^n,
\end{align*}
because $\sum_{j=1}^n \exp(\frac{2\pi i}{n}(n-k) j)$ is equal to $0$ for $1 \le k \le n-1$, 
and equal to $n$ for $k=0$ and $n$. Thus,
\bes
\| \sum_{j=1}^n (\alpha e_1 + \theta_j  \beta e_2)^n \|^2 
=  (\alpha^{2n} + \beta^{2n}) n^2.
\ees
Comparing this norm for $(\alpha,\beta) = (a,b)$ and $(\alpha, \beta) = (1,1)$ 
we find that the densely defined $\tilde{A}$ is unbounded.
\end{example}

\section{Classification of the algebras}\label{sec:class2}

We now have enough machinery to give a geometric classification of the 
operator algebras $\cA_I$. 
In the case of algebraic isomorphism, we require the varieties to be tractable.

First, let us say a few words about 
the purely algebraic problem.
When $I$ is a radical ideal in $\mb{C}[z]$, then $\mb{C}[z]/I$ can be 
identified with the ring of polynomial functions on $V(I)$, which is 
nothing but the ring of restrictions of polynomials to $V(I)$. 
This algebra is also the universal unital commutative algebra generated 
by a tuple satisfying the relations in $I$. If $J$ is another radical ideal, 
then every homomorphism from $\mb{C}[z]/I$ to $\mb{C}[z]/J$ gives 
rise to a regular map (i.e., a polynomial map) $V(J) \rightarrow V(I)$, 
and the two algebras are isomorphic if and only if the varieties are 
isomorphic (see \cite[p. 29]{Shafarevich}). 
Consequently, a grading preserving isomorphism is implemented by a 
linear change of variables. 
Therefore, when $I$ and $J$ are homogeneous, $\mb{C}[z]/I$ and 
$\mb{C}[z]/J$ are isomorphic as graded algebras if and only if there 
is a linear map that takes $V(J)$ bijectively onto $V(I)$. 
We will see that the situation for the algebras $\cA_I$ is both similar and different.

\subsection{Classifying the algebras $\cA_I$ up to isometric isomorphism}
We provide a concrete criterion for when two algebras $\cA_I$ 
and $\cA_J$ are (completely) isometrically isomorphic.

\begin{remark}[Adding variables]
Let $I$ be an ideal in $\mb{C}[z_1, \ldots, z_d]$, and let $d'>d$. 
We may want to consider $I$ as an ideal in $\mb{C}[z_1, \ldots, z_{d'}]$. 
Of course, it isn't. But note that if we define 
$I' = \lel I, x_{d+1}, \ldots, x_{d'}\rir$, then $I'$ is an ideal in 
$\mb{C}[z_1, \ldots, z_{d'}]$ and $V(I')$ is isomorphic to $V(I)$. 
Furthermore, $\mb{C}[V(I)] \cong \mb{C}[V(I')]$ and 
$\cA_I$ is completely isometrically isomorphic to $\cA_{I'}$. 
Therefore, when studying the situation where 
$I$ is an ideal in $\mb{C}[z_1, \ldots, z_d]$ and 
$J$ is an ideal in $\mb{C}[z_1, \ldots, z_{d'}]$, 
we may assume that $d = d'$. 
We do not always make this assumption, but the next theorem is much 
more elegant when stated for the case $d = d'$.
\end{remark}

\begin{theorem}\label{thm:iso_iso_alg}
Let $I$ and $J$ be two homogeneous radical ideals in 
$\mb{C}[z_1, \ldots, z_d]$. $\cA_I$ and $\cA_J$ are isometrically isomorphic 
if and only if they are completely isometrically isomorphic.
This happens if and only if there is a unitary $U$ on $\mb{C}^d$ taking 
$V(J)$ onto $V(I)$. 
\end{theorem}

\begin{proof}
By Proposition \ref{prop:idealsps} and Theorem \ref{thm:alg_sps_iso}, 
$\cA_I$ and $\cA_J$ are (completely) isometrically isomorphic 
if and only if there is a unitary $U$ such that
\bes
J = \{f \circ U^{-1} : f \in I\}.
\ees
Since $I$ and $J$ are radical, it follows from Hilbert's Nullstellensatz 
that this holds if and only if $U(V(J)) = V(I)$.
\end{proof}

\subsection{Classifying the algebras $\cA_I$ up to isomorphism}

\begin{proposition}\label{prop:exist_vacuum2}
Let $I$ and $J$ be two homogeneous radical ideals of polynomials 
and assume that there exists an  isomorphism $\varphi: \cA_I \rightarrow \cA_J$. 
Then there exists a vacuum preserving isomorphism from $\cA_I$ to $\cA_J$.
\end{proposition}

\begin{proof}
The proof is identical to the proof of Proposition \ref{prop:exist_vacuum}, 
where one uses Proposition \ref{prop:algiso_biholo} instead of Lemma \ref{lem:analytic}.
\end{proof}

\begin{remark}\label{rem:trick}
The same trick used to prove Propositions \ref{prop:exist_vacuum} and 
\ref{prop:exist_vacuum2} can be used to show that, if there is biholomorphism 
between $Z^o(I)$ and $Z^o(J)$, then there is a biholomorphism 
between them that fixes $0$. 
This may seem like an obvious result, but consider the following problem: 
\emph{given that $Z(I)$ and $Z(J)$ are homeomorphic, prove that there exists a
homeomorphism between them that fixes $0$.}
\end{remark}

\begin{theorem}\label{thm:algiso_lin}
Let $I$ and $J$ be homogeneous ideals in $\mb{C}[z_1, \ldots, z_d]$
and $\mb{C}[z_1, \ldots, z_{d'}]$, respectively, such that $V(J)$ is tractable. 
The algebras $\cA_I$ and $\cA_J$ are isomorphic if and only if there exist two linear maps 
$A: \mb{C}^d \rightarrow \mb{C}^{d'}$ and 
$B : \mb{C}^{d'} \rightarrow \mb{C}^{d}$ such that 
$A \circ B |_{Z(J)}$ and $B \circ A|_{Z(I)}$ are identity maps.
\end{theorem}

\begin{proof}
If $\cA_I$ and $\cA_J$ are isomorphic, then by 
Proposition~\ref{prop:exist_vacuum2} there exists also a 
vacuum preserving isomorphism between them. 
By Proposition \ref{prop:algiso_biholo} and Theorem \ref{thm:linear}, 
there exist linear maps $A,B$ as asserted.

If, conversely, there exist linear maps $A,B$ as in the statement 
of the theorem, then Theorem \ref{thm:linear_induce_similarity} 
applies to show that there is an isomorphism (in fact, a similarity) 
from  $\cA_I$ onto $\cA_J$.
\end{proof}

\begin{example}\label{expl:lines}
Consider the simplest case when $d = d' = 2$. 
Then the maximal ideal spaces $Z(I)$ and $Z(J)$ are either $0$, 
$\overline{\mb{B}}_2$ or finitely many lines. 
If $Z(I)$ and $Z(J)$ are one line, then $\cA_I$ and $\cA_J$ 
are completely isometrically isomorphic. 
If $Z(I)$ and $Z(J)$ consist of two lines, then $\cA_I$ and $\cA_J$ 
are isomorphic  but if the angle between the two lines is not the same 
then they will not be isometrically isomorphic. 
If $Z(I)$ and $Z(J)$ consist of three or more lines, 
then $\mb{C}[z]/I$ and $\mb{C}[z]/J$ might not be isomorphic, 
because the action of a linear map on $\mb{C}^2$ is determined 
already by its action on two lines. 
The coordinate rings $\mb{C}[z]/I$ and $\mb{C}[z]/J$ are isomorphic 
precisely when there exists a linear map $A$ mapping $V(J)$ onto $V(I)$. 
When this happens, there exist cases when $\cA_I$ and $\cA_J$ 
are isomorphic, and there exist cases when they are not--- depending on 
whether or not this $A$ maps $Z(J)$ onto $Z(I)$. 
\end{example}

The geometric rigidity of the varieties implies that the operator algebras also 
have a rigid structure.

\begin{theorem}\label{thm:rigid}
Let $I$ and $J$ be two radical homogeneous ideals in 
$\mb{C}[z_1, \ldots, z_d]$, and assume that $V(I)$ is either irreducible 
or a nonlinear hypersurface. 
If $\cA_I$ and $\cA_J$ are isomorphic, then $\cA_I$ and $\cA_J$ are unitarily equivalent. 
If $\varphi: \cA_I \rightarrow \cA_J$ is a vacuum preserving isomorphism, 
then it is unitarily implemented.
\end{theorem}
\begin{proof}
This follows from Theorems \ref{thm:algiso_lin}, \ref{thm:iso_iso_alg} 
and Proposition \ref{prop:rigid}.
\end{proof}

\section{Automorphisms of $\cA_d$ and induced isomorphisms}\label{sec:aut}

\subsection{Automorphisms of $\cA_d$}\label{subsec:aut}

By Proposition \ref{prop:algiso_biholo}, every (algebraic) automorphism of 
$\cA_d$ arises as a composition operator 
$f \mapsto f \circ \phi$, where $\phi \in \Aut(\mb{B}_d)$. 
Conversely, it is known that every conformal automorphism of the ball 
yields a completely isometric isomorphism of $\cA_d$. 
As we do not have a convenient reference, we briefly sketch the ideas.  
Voiculescu \cite{Voic} constructed unitaries on full Fock space which
implement $*$-automorphisms of the Cuntz-Toeplitz algebra and fix the 
noncommutative disc algebra $\mathfrak{A}_d$.
Davidson and Pitts \cite{DavPitts2} showed that the action on the character space
was the action of the full group $\Aut(\mb{B}_d)$. 
It is clear that these automorphisms preserve the commutator ideal, and thus the
unitaries preserve the range of the commutator ideal, $(H^2_d)^\perp$.
Thus they also fix $H^2_d$.
Now $\cA_d$ is completely isometrically isomorphic to the quotient of $\mathfrak{A}_d$ by the commutator ideal, and this is completely
isometric to the compression to $H^2_d$ by \cite{DavPittsPick}.
So the compressions of the Voiculescu unitaries implement the action
of $\Aut(\mb{B}_d)$ on $\cA_d$.

\begin{theorem}\label{thm:autoAd}
Every $\phi \in \Aut (\mb{B}_d)$ gives rise to a 
completely isometric automorphism of $\cA_d$.
\end{theorem}

In fact we can say more than this, specifically that the Voiculescu unitaries,
when restricted to symmetric Fock space, 
are just composition with the conformal map followed by an appropriate multiplier.

\begin{theorem}
Let $\phi \in \Aut(\mb B_d)$.  
Then there is a completely isometric automorphism $\Theta_\phi$ of $\cA_d$ 
given by $\Theta_\phi(f) = f\circ\phi = UfU^*$,  
where the unitary $U: H^2_d \rightarrow H^2_d$ is 
\[
 Uf = \big( 1 - |\phi^{-1}(0)|^2 \big)^{1/2} \nu_{\phi^{-1}(0)} (f\circ\phi).
\] 
\end{theorem}

\begin{proof}
We begin with Voiculescu's construction of automorphisms of the Cuntz algebra
\cite{Voic}.
Consider the Lie group $U(1,d)$ consisting of $(d+1)\times(d+1)$ matrices $X$ satisfying $X^*JX = J$, where $J = \begin{sbmatrix} -1&0\\0&I_d \end{sbmatrix}$.
When $X$ is of the form 
$X = \begin{sbmatrix} x_0&\eta_1^*\\ \eta_2&X_1 \end{sbmatrix}$ it must have the following relations:
\begin{enumerate}
\item $\|\eta_1\|^2 = \|\eta_2\|^2 = |x_0|^2 - 1$
\item $X_1\eta_1 = \overline x_0\eta_2 \ \ {\rm and} \ \ X_1^*\eta_2 = x_0\eta_1$
\item $X_1^*X_1 = I_d + \eta_1\eta_1^* \ \  {\rm and} \ \ X_1X_1^* = I_d + \eta_2\eta_2^*$.
\end{enumerate}
Furthermore, if $X\in U(1,d)$ then $JX^TJ\in U(1,d)$ since 
\[
(JX^TJ)^*J(JX^TJ) = J(X^*)^TJX^TJ = (XJX^*J)^TJ = I_{d+1}J = J.
\]
It follows from Voiculescu's work that the map $U(1,d) \rightarrow \Aut(\mb B_d)$ given by
\[
X \mapsto \phi_X(z) := \frac{X_1z + \eta_2}{x_0 + \langle z,\eta_1\rangle} 
\]
is a surjective homomorphism. Thus, fix $X\in U(1,d)$ such that $\varphi = \varphi_{JX^TJ}$ which makes 
\[
\varphi_{\overline{X}} = \varphi_{\overline{JX^*J}}^{-1} = \varphi_{JX^TJ}^{-1} = \varphi^{-1}.
\]

There is a unique automorphism of $\mathfrak{A}_d$ defined by
\[
\Theta_\varphi(L_\zeta) = (x_0 I - L_{\eta_2})^{-1}(L_{X_1\zeta} - \langle \zeta,\eta_1\rangle I),
\] 
where we use the convention that $L_\zeta = \sum_{i=1}^n \zeta_i L_i$ for 
$\zeta \in \mb{C}^d$. 
This extends to an automorphism of the Cuntz-Toeplitz algebra. 
As well, Voiculescu  defined a unitary $U \in \cU(\cF(\mb{C}^d))$ by
\[
U(A\Omega) = \Theta_\varphi(A)(x_0 I - L_{\eta_2})^{-1}\Omega, 
\quad \text{ for all }  A\in\mathcal L_d ,
\]
establishing that the automorphism $\Theta_\varphi(A) = UAU^*$ is unitarily implemented.
As was discussed in the beginning of this section, $H^2_d$ is an invariant subspace 
of $U$ and so $\Theta_\varphi$ also yields an automorphism of $\cA_d$ which is
implemented by the restriction of $U$.
We will show that $U$ has the desired form.

For $w\in \mathbb F_d^+, |w|=m$, we have 
\begin{align*}
 U(z_w) = U \big( \frac{1}{m!}\sum_{\sigma\in S_m} \xi_{\sigma(w)} \big) &= 
 P_{H^2_d} U \Big( \big( \frac{1}{m!}
  \sum_{\sigma\in S_m} L_{\sigma(w)} \big) \Omega \Big)\\
 &= P_{H^2_d}\Theta_\varphi(M_{z_w})P_{H^2_d}(x_0 I - L_{\eta_2})^{-1}\Omega.
\end{align*}
As noted above, because $H^2_d$ must reduce $U$,
we obtain $P_{H^2_d}\Theta_\varphi(A) = P_{H^2_d}\Theta_\varphi(A)P_{H^2_d}$. 
Suppose that $\zeta\in\mb{C}^d$.  Then
\[
P_{H^2_d} (L_\zeta)(z) = \sum_{i=1}^d \zeta_i z_i(z) = 
\sum_{i=1}^d \zeta_i \langle z,e_i\rangle = 
\langle z, \overline{\zeta}\rangle.
\]
Now with $\overline{x_0^{-1}\eta_2} = \phi_{\overline X}(0) = \varphi^{-1}(0)$, we have that 
\[
P_{H^2_d}(x_0 I - L_{\eta_2})^{-1} \Omega  
= \frac{1}{x_0 - \langle z, \overline{\eta_2}\rangle} = x_0^{-1}\nu_{\varphi^{-1}(0)} .
\]
Note that if $|\theta|=1$, then $\theta X$ implements $\phi_X$ as well.
So we may assume that $x_0\geq 0$. 
As well, $X\in U(1,d)$ implies that $|x_0|^2 - |\eta_2|^2 = 1$. 
Hence,
\[
|\varphi^{-1}(0)|^2 = |\phi_X(0)|^2 = \frac{|\eta_2|^2}{|x_0|^2} = \frac{|x_0|^2 - 1}{|x_0|^2}.
\]
Thus $x_0 = (1-|\phi^{-1}(0)|^2)^{-1/2}$.

Next we compute
\begin{align*}
P_{H^2_d}\Theta_\varphi(M_{z_w}) &= 
P_{H^2_d}\Theta_\varphi \big( \frac{1}{m!}\sum_{\sigma\in S_m} L_{\sigma(w)} \big) \\&=
\frac{1}{m!}\sum_{\sigma\in S_m} \prod_{j=1}^m P_{H^2_d}\Theta_\varphi(L_{\sigma(w)_j}) \\
&= \prod_{j=1}^m P_{H^2_d}\Theta_\varphi(L_{w_j}) \\&= 
\prod_{j=1}^m P_{H^2_d}\frac{L_{X_1e_{w_j}}  - 
 \langle e_{w_j}, \eta_1\rangle I}{x_0 I - L_{\eta_2}} .
\end{align*}
Observe that 
\[
JX^TJ = 
\left[\begin{array}{cc} 
x_0& -\overline{\eta_2}^* \\ 
-\overline{\eta_1} & X_1^T 
\end{array}\right] .
\]
Consequently,
\begin{align*}
 P_{H^2_d} \Theta_\varphi(M_{z_w})(z) &= 
 \prod_{j=1}^m \frac{P_{H^2_d} L_{X_1e_{w_j}}(z) - 
  \langle e_{w_j}, \eta_1\rangle }{x_0 - P_{H^2_d} L_{\eta_2}(z)} \\&
 = \prod_{j=1}^m \frac{ \langle z, \overline{X_1e_{w_j}}\rangle - 
  \langle \overline{\eta_1}, e_{w_j} \rangle}{x_0 - \langle z, \overline{\eta_2}\rangle} 
 = \prod_{j=1}^m \frac{ \langle X_1^Tz, e_{w_j}\rangle + 
  \langle -\overline{\eta_1}, e_{w_j}\rangle }{ x_0 + 
  \langle z, -\overline{\eta_2}\rangle } \\&
 = \prod_{j=1}^m z_{w_j}\left( \frac{ X_1^Tz + -\overline{\eta_1} }{x_0 + 
  \langle z,-\overline{\eta_2}\rangle} \right) 
 = \prod_{j=1}^m z_{w_j}(\phi_{JX^TJ}(z)) \\&
 = \prod_{j=1}^m z_{w_j}(\phi(z)) = (z_w\circ\varphi)(z).
\end{align*}
Combining these equations, we get that
\begin{align*}
 U(z_w) &= \Big(\prod_{j=1}^m z_{w_j}\circ\phi \Big)
  (1-|\phi^{-1}(0)|^2)^{1/2}\nu_{\phi^{-1}(0)} \\&
 = (z_w\circ\phi) (1-|\phi^{-1}(0)|^2)^{1/2}\nu_{\varphi^{-1}(0)}.
\end{align*}
Extending this to the span, we have that 
\[ Uf = (1-|\phi^{-1}(0)|^2)^{1/2}\nu_{\varphi^{-1}(0)} (f\circ\phi) \]
for all $f\in \cA_d$. 
\end{proof}

\begin{remark}
This also gives a very nice description for $U^*$ on kernel functions. Letting $\lambda_0 = \varphi^{-1}(0)$ then the previous theorem gives us
\[
 Uf = \sqrt{1-\|\lambda_0\|^2}\, \nu_{\lambda_0} (f\circ\varphi)
 \ \ \text{and} \ \ UM_fU^*= M_{f\circ\varphi}.
\]
Then for $\lambda \in \mb B_d$ we have
\[
 M_f^*(U^* \nu_\lambda) = U^*M_{f\circ\varphi}^* \nu_\lambda = 
 U^* \overline{(f\circ\varphi)(\lambda)} \nu_\lambda = 
 \overline{(f\circ\varphi)(\lambda)}(U^*\nu_\lambda).
\]
Hence, $U^*\nu_\lambda = c_\lambda\nu_{\varphi(\lambda)}\in \mb C\nu_{\varphi(\lambda)}$, 
the eigenspace of $M_f^*$ for eigenvalue $\overline{(f\circ\varphi)(\lambda)}$.
Now compute
\begin{align*}
 \overline{c_\lambda}f(\varphi(\lambda)) & =
 \langle f, c_\lambda \nu_{\varphi(\lambda)}\rangle  = 
 \langle f, U^*\nu_\lambda\rangle \\&
= \langle Uf, \nu_\lambda\rangle = 
\sqrt{1-\|\lambda_0\|^2}\langle (f\circ\varphi)\nu_{\lambda_0}, \nu_\lambda\rangle \\&
= \sqrt{1-\|\lambda_0\|^2} \frac{f(\varphi(\lambda))}{1-\langle \lambda,\lambda_0\rangle}.
\end{align*}
Therefore,
\[
U^*\nu_\lambda = 
\frac{\sqrt{1-\|\lambda_0\|^2}}{1-\langle \lambda_0,\lambda\rangle} \nu_{\varphi(\lambda)}.
\]
\end{remark}

We wish to describe how $\phi \in \Aut(\mb{B}_d)$ gives rise to an 
isomorphism $\varphi : \cA_I \rightarrow \cA_J$, when $I$ and $J$ are 
radical ideals in $\mb{C}[z]$.

\begin{proposition}\label{prop:induced}
Let $I$ and $J$ be homogeneous radical ideals in $\mb{C}[z]$. 
Let $\phi \in \Aut(\mb{B}_d)$ map $Z(J)$ onto $Z(I)$. 
Then the automorphism of $\cA_d$ given by $\varphi(f) = f\circ \phi$ 
maps $\overline{I}$ onto $\overline{J}$. 
Consequently, $\varphi$ induces an isometric isomorphism 
$\varphi' : \cA_I \rightarrow \cA_J$ given by $\varphi'(f) = f \circ \phi$.
\end{proposition}

\begin{proof}
It suffices to prove the first assertion. 
In fact, it suffices to prove that $\varphi$ maps $\overline{I}$ \emph{into} $\overline{J}$. 
Let $f \in \overline{I}$. Then $f \circ \phi$ vanishes on $Z(J)$. 
By Corollary \ref{cor:approxpoly}, $f\circ \phi \in \overline{J}$.
\end{proof}

\begin{remark}
As we have seen in the discussion following Theorem \ref{thm:algiso_lin}, 
not every algebraic isomorphism between two algebras 
$\cA_I$ and $\cA_J$ is isometric.
Thus not every such isomorphism is induced from an automorphism of $\cA_d$. 
This leaves us with the question: is every isometric isomorphism between 
two such algebras induced from an automorphism of $\cA_d$? 
We answer this in a very special case.
\end{remark}

\subsection{The automorphism group of a union of subspaces}

Let $I$ be a radical ideal such that $V = V(I)$ is a union of subspaces. 
We will compute the group of automorphisms of $Z:=Z(I)$. 
By an ``automorphism" of $Z$ we mean a map 
$\phi : \overline{\mb{B}}_d \rightarrow \mb{C}^d$, analytic in $\mb{B}_d$, 
such that there exists $\psi: \overline{\mb{B}}_d \rightarrow \mb{C}^d$, 
analytic in $\mb{B}_d$, for which 
$\phi \circ \psi |_Z = \psi \circ \phi |_Z = \id$.
The collection of all such maps is denoted by $\Aut(Z)$.

Write $V = V_1 \cup \ldots \cup V_k$. 
Setting $Z_i = V_i \cap \overline{\mb{B}}_d$, we have also 
$Z = Z_1 \cup \ldots \cup Z_k$. 
Finally, define $Z_0 = \cap_{i=1}^k Z_i$.

For $a \in \mb{B}_d$, we define $\phi_a$ as in (\ref{eq:mobius}).

\begin{lemma}\label{lem:auto_subspaces}
Suppose that $ a \in Z_0$ and $A$ is a linear map which takes $Z$ onto itself.
The map $\phi = \phi_a \circ A$ yields an automorphism of $Z$. 
Conversely, every automorphism of $Z$ arises in this way.
\end{lemma}

\begin{proof}
Let $a \in Z_0$. We must show that $\phi_a$ preserves $Z$. 
Let $z \in Z_i$. Write $z = x + y$, where $x,y \in Z_i$,  $x \in \spn\{a\}$ and 
$y \perp a$. Then 
\bes
\phi_a(z) = (1-\lel x,a\rir)^{-1} (a - x) -  s_a(1-\lel x,a\rir)^{-1} y \in Z_i.
\ees

For the other direction, let $\phi \in \Aut(Z)$, and let $a = \phi(0)$. 
Note that $\phi$ must permute the subspaces $Z_i$, and thus
preserves their intersection $Z_0$.
Hence $\phi(0) = a \in Z_0$.
It was established above that $\phi_a$ preserves $Z$. 
Thus $\phi_a \circ \phi$ is an automorphism of $Z$ which takes $0$ to $0$. 
By Theorem \ref{thm:linear}, $\phi_a \circ \phi = A$, where $A$ is a linear map. 
\end{proof}

\begin{corollary}\label{cor:auto_subspaces}
Suppose that $V$ is a tractable union of subspaces, and $I = I(V)$.
Then $\Aut(\cA_I)$ is isomorphic to $\Aut(Z(I))$, and all of these
maps are implemented by similarities.

The subgroup of (completely) isometric automorphisms is identified with those
$\phi\in\Aut(Z(I))$ of the form $\phi = \phi_a \circ U$ where $U$ is
a unitary map which fixes $Z(I)$.  These are precisely the quotients
of $\theta\in\Aut(\mb B_d)$ which fix $Z(I)$, and they are unitarily
implemented.
\end{corollary}

\begin{proof}
Lemma~\ref{lem:auto_subspaces} identifies the elements of $\Aut(Z(I))$.
The automorphisms $\phi_a$ for $a \in Z_0$ are automorphisms of $\mb B_d$,
and thus are induced by the completely isometric automorphism of $\cA_d$.
In particular, they are unitarily implemented on $H^2_d$ and fix the ideal of
functions which vanish on $Z^o(I)$.
Thus the orthogonal complement, $\cF_I$, is also fixed by this unitary.
So the automorphism $\phi_a$ is unitarily implemented.

The linear map $A$ fixes $Z(I)$ and is necessarily isometric on $V$.
By Theorem~\ref{thm:linear_induce_similarity}, $\tilde A$ implements 
the automorphism via a similarity.
When $U$ is unitary, $\tilde U$ is unitary and the automorphism is
unitarily implemented, and thus is completely isometric.
Conversely, by Theorem~\ref{thm:iso_iso_alg}, isometric automorphisms
are unitarily implemented by $\tilde U$ for some unitary $U$ which fixes $Z(I)$.
These are evidently induced by the corresponding automorphism of $\Aut(\cA_d)$.
\end{proof}

\begin{example}
Consider the variety $V = V_1 \cup V_2 \subset \mb C^3$ given by 
$V_1 = \spn\{ e_1,e_2 \}$ and $V_2 = \spn\{ (e_2+e_3)/\sqrt2 \}$.
If $U = \begin{sbmatrix}u_{11}&u_{12}\\u_{21}&u_{22}\end{sbmatrix}$
is any $2\times2$ unitary matrix, and $\beta \in [0,2\pi)$, the map
\[
 A = \begin{bmatrix}
 u_{11}&u_{12}&-u_{12}\\
 u_{21}&u_{22}&e^{i\beta}-u_{22}\\
 0&0&e^{i\beta}\end{bmatrix}
\]
is an isometric map of $V$ onto itself. 
It is easy to see that these are the only possibilities.
Since  $\spn V = \mb C^3$, this does not coincide with any unitary map except
when it is unitary, which occurs only for the subgroup of the form
\[
 A = \begin{bmatrix}e^{i\alpha}&0&0\\0&e^{i\beta}&0\\0&0&e^{i\beta}\end{bmatrix} ,
 \quad\text{for } \alpha,\, \beta \in [0,2\pi) .
\]

Since $V_1$ has codimension one, $V$ is tractable.
So Corollary~\ref{cor:auto_subspaces} applies.
$Z_0 = \{0\}$.  So $\Aut(Z(I))$ coincides with the linear maps described above,
and the isometric subgroup corresponds to the unitaries, and
so is isomorphic to $\mb T^2$.
\end{example}

\section{Toeplitz algebras and C*-envelopes}\label{sec:Toeplitz}

In this section we consider the Toeplitz algebra of $X$, defined as 
$\cT_X = \ca(\cA_X)$. We begin with some simple consequences of 
Section \ref{sec:class}.

\begin{theorem}
Let $X$ and $Y$ be subproduct systems.
\begin{enumerate}
\item Every vacuum preserving isometric isomorphism 
$\varphi: \cA_X \rightarrow \cA_Y$ extends to a $*$-isomorphism 
$\tilde{\varphi}: \cT_X \rightarrow \cT_Y$.
\item If $\cA_X$ and $\cA_Y$ are isometrically isomorphic, 
then $\cT_X$ and $\cT_Y$ are $*$-iso\-mor\-phic. 
\end{enumerate}
\end{theorem}

\begin{proof}
Assertion (1) follows from Theorem \ref{thm:iso_vacuum}. 
Assertion (2) then follows from Proposition \ref{prop:exist_vacuum}.
\end{proof}

Example \ref{expl:notiso} shows that the converse of assertion (2) above is false.
We do not know whether an isomorphism that does not preserve the vacuum 
can be extended to a $*$-isomorphism of the C*-algebras. 
In \cite{Viselter}, Viselter studied (in greater generality) the problem of 
when a completely contractive representation of $\cA_X$ can be extended 
to a $*$-representation of $\cT_X$, but his results do not apply directly.

\subsection{The C*-envelope of $\cA_X$, $X$ commutative}
In this subsection all our subproduct systems will be commutative. 
Thus, below, $X$ and $Y$ will always denote commutative subproduct 
systems and the algebras $\cA_X, \cA_Y$ will always be commutative algebras. 
Recall that we denote $\cO_X = \cT_X/\cK(\cF_X)$, where $\cK(\cF_X)$ denotes the
compact operators on $\cF_X$.

A variant of the following lemma appears as \cite[Proposition 6.4.6]{ChenGuo}, 
where the result is proven for arbitrary (not necessarily homogeneous) 
submodules of $H^2_d$. The situation in \cite{ChenGuo} is slightly different, 
but after a simple modification the proof carries over to our case.

\begin{lemma}\label{lem:notCI}
If $\dim X(1) > 1$ then the quotient map $q: \cT_X \rightarrow \cO_X$ 
is not a complete isometry. 
\end{lemma}

By \cite[Theorem 2.1.1]{Arv72}, the identity representation is a boundary 
representation if and only if the quotient map $q: \cT_X \rightarrow \cO_X$ 
is not a complete isometry. 
Thus the above lemma gives immediately:

\begin{corollary}
The identity representation of $\cT_X$ is a boundary representation for $\cA_X$.
\end{corollary}

Since the Silov boundary ideal is contained in the kernel of any boundary 
representation, we find that the Silov ideal of $\cA_X$ in $\cT_X$ is $\{0\}$. 
Thus we obtain:

\begin{theorem}
The C*-envelope of $\cA_X$ is $\cT_X$.
\end{theorem}

This allows us to prove that all the completely isometric isomorphisms 
in the commutative setting are unitarily implemented:

\begin{theorem}\label{thm:unitimpl}
Let $\varphi : \cA_X \rightarrow \cA_Y$ be a completely isometric isomorphism. 
Then there exists a unitary $U: \cF_X \rightarrow \cF_Y$ such that
\bes
\varphi(T) = U T U^* \,\, , \,\, T \in \cA_X .
\ees
\end{theorem}

\begin{proof}
By Arveson's ``Implementation Theorem" \cite[Theorem 0.3]{Arv72}, 
$\varphi$ is implemented by a $*$-isomorphism $\pi:\cT_X \rightarrow \cT_Y$. 
Since $\cK(\cF_X) \subseteq \cT_X$ (see \cite[Proposition 8.1]{ShalitSolel}), 
$\pi = \pi_0 \oplus \pi_1$, where $\pi_0$ is a multiple of representations 
unitarily equivalent to the identity representation and $\pi_1$ annihilates the compacts. 
Since $\cK(\cF_Y) \subseteq \cT_Y$ and $\pi$ is an isomorphism, 
$\pi$ is irreducible and therefore has just one summand. 
Thus either $\pi$ is unitarily implemented, or $\pi$ annihilates the compacts. 
But if $\pi$ annihilates the compacts it factors through $\cO_X$, 
that is, $\pi = \tilde{\pi} \circ q$ where $\tilde{\pi}: \cO_X \rightarrow \cT_Y$ 
is a $*$-homomorphism and $q: \cT_X \rightarrow \cO_X$ is the quotient map. 
Thus $\varphi = \tilde{\pi} \circ q |_{\cA_X}$. By Lemma \ref{lem:notCI}, 
this contradicts the assumption that $\varphi$ is completely isometric.
\end{proof}

The above result is interesting for the non vacuum-preserving case, 
as Theorem \ref{thm:iso_vacuum} shows that every vacuum preserving 
isometric isomorphism is unitarily implemented (even for $X$ not commutative).

Having brought C*-algebras into our discussion about universal operator algebras, 
one might wonder whether our methods give any handle on the universal unital 
C*-algebra generated by a row contraction subject to homogeneous polynomial relations. 
Unfortunately, these universal C*-algebras are out of our reach. 
All we can say is that $\cT_X$ is \emph{not}, in general, the universal 
unital C*-algebra generated by a row contraction subject to 
the relations in $I^X$. 
One can see this by considering the case $d=1$ and no relations. 
Then $\cT_X$ is the ordinary Toeplitz algebra, which is not the universal 
unital C*-algebra generated by a contraction.

\subsection{The Toeplitz algebras and topology}

It is a fact that, for any subproduct system $X$, $
\cK(\cF_X) \subseteq \cT_X$ (see \cite[Proposition 8.1]{ShalitSolel}). 
Thus, there is always an exact sequence
\bes
0 \longrightarrow \cK(\cF_X) \longrightarrow \cT_X 
\longrightarrow \cO_X \longrightarrow 0.
\ees
Arveson conjectured that, for any homogeneous ideal such that 
$I \subseteq \mb{C}[z]$, the algebra $\cO_{X_I}$ is commutative \cite{Arv05}. 
This conjecture is still open; the most up-to-date results can be found in \cite{E10}. 
There are several significant consequences of this conjecture treated in the 
literature (see, e.g., \cite{Arv07}). 
We will see below that another consequence is a connection between the 
$*$-algebraic structure of the Toeplitz algebras $\cT_X$ and 
the topology of the variety $V(I^X)$. 
The ``topological classification" results in this subsection should be compared 
with the ``geometrical classification" results of Section \ref{sec:class2}.

Given a homogeneous ideal $I \subseteq \mb{C}[z]$, let us say that 
\emph{Arveson's conjecture holds for $I$}, if $\cO_{X_I}$ is commutative. 
Note that if Arveson's conjecture holds for $I$ and $X = X_I$, then 
the above exact sequence becomes
\be\label{eq:exact}
0 \longrightarrow \cK(\cF_X) \longrightarrow \cT_X 
\longrightarrow C(V(I)\cap \partial\mb{B}_d) \longrightarrow 0.
\ee

\begin{proposition}
Let $I,J \subseteq \mb{C}[z]$ be two homogeneous ideals for 
which Arveson's conjecture holds. 
Let $X = X_I$ and $Y = X_J$. If $\cT_X$ is $*$-isomorphic to $\cT_Y$, 
then $V(I)\cap \partial\mb{B}_d$ is homeomorphic to $V(J)\cap \partial\mb{B}_d$, 
and consequently $V(I)$ is homeomorphic to $V(J)$.
\end{proposition}

\begin{proof}
In the proof of Theorem \ref{thm:unitimpl} it was observed that a $*$-isomorphism 
from $\cT_X$ onto $\cT_Y$ is unitarily implemented, 
and therefore sends the compacts onto the compacts. 
Therefore, given that the exact sequence (\ref{eq:exact}) holds for $X$ and for $Y$, 
every such $*$-isomorphism  induces a $*$-isomorphism between 
$C(V(I)\cap \partial\mb{B}_d)$ and $C(V(J)\cap \partial\mb{B}_d)$. 
The assertion follows.
\end{proof}

Thus, the topology of $V(I)$ is an invariant of the algebras $\cT_X$. 
Examples \ref{expl:iso} and \ref{expl:notiso} show that it is not a complete invariant 
(in both examples $V(I) = \{0\}$, but $\cT_X$ is either $M_4(\mb{C})$ or $M_5(\mb{C})$). 
This is not surprising, as the ideals arising in Examples \ref{expl:iso} 
and \ref{expl:notiso} are not radical. 
Does the topology of $V(I)$ determine the structure of the associated 
algebra $\cT_X$ when $I$ is radical? 
All we can say right now is that the answer is \emph{yes} in dimension $d=2$ 
(when there is, in fact, not too much topology going on). 
It is interesting to compare the following proposition with the discussion in 
Example \ref{expl:lines}.

\begin{proposition}\label{prop:C_star_lines} 
Let $I,J \subseteq \mb{C}[x,y]$ be two radical homogeneous ideals. 
Let $X = X_I$ and $Y = X_J$. Then $V(I)$ is homeomorphic to $V(J)$ 
if and only if $\cT_X$ is $*$-isomorphic to $\cT_Y$.
\end{proposition}

\begin{proof}
In dimension $d=2$, Arveson's conjecture holds for all ideals 
\cite[Theorem 3.1] {GuoWang} (see also \cite{E10,Sh10}). 
For a nontrivial ideal $I \subset \mb{C}[x,y]$, $V(I)$ is equal to a union of 
lines $\cup_{i=1}^k \ell_i$. If $I_i$ is the radical ideal corresponding to 
the line $\ell_i$, then we have $I = \cap_{i=1}^k I_i$. 
It is easy to see that the Toeplitz algebra corresponding to $I_i$ 
is equal to the ordinary Toeplitz algebra $\cT$, that is, 
the C*-algebra generated by the unilateral shift. 
By \cite[Proposition 5.2]{GuoWang}, 
\bes
\cT_X = \Big(\underbrace{\cT \oplus \cdots \oplus \cT}_{k \textrm{ times}}\Big) + \cK \, ,
\ees
and this C*-algebra is completely determined by the number $k$, 
which encodes also the topology of $V(I)$.
\end{proof}

Similar assertions can be made in higher dimensions about unions of subspaces
intersecting at $\{0\}$, assuming that Arveson's conjecture holds.

\section{The classification of the \wot-closures of the algebras $\cA_X$}
\label{sec:wot}

Let $\cL_X$ be the \wot-closure of $\cA_X$ in $B(\cF_X)$. 
In the commutative case we write $\cL_I$ instead of $\cL_X$, where, 
as usual, $I = I^X$. In this section we will classify the algebras $\cL_X$ 
up to isometric isomorphism, and for $I$ radical and $V(I)$ tractable, 
we will classify the algebras $\cL_I$ up to isomorphism. 
We will also show that in the radical commutative case, every isomorphism 
is automatically bounded and continuous in the weak-operator and weak-$*$ topologies.

It turns out that, just like in the norm-closed case, the Banach algebra 
structure of $\cL_X$ is completely determined by the the subproduct system $X$; 
the algebraic structure of $\cL_I$ determines the geometry of $V(I)$, 
and is determined by this geometry when $V(I)$ is tractable. 
The rigidity results obtained above also survive the \wot-closure. 
Before proving these results, let us explain why they are not obvious.

Let $V_1, \ldots, V_d$ be a set of isometries on a Hilbert space 
with pairwise orthogonal ranges. 
The normed closed algebra $\overline{\alg}\{V_1, \ldots, V_d\}$ is always 
isometrically isomorphic to the noncommutative disc algebra 
$\mathfrak{A}_d = \overline{\alg}\{L_1, \ldots, L_d\}$ 
(see the proof of Theorem 2.1, \cite{Popescu91}). 
On the other hand, the \wot-closure of $\overline{\alg}\{V_1, \ldots, V_d\}$ 
may fall into several quite different isomorphism classes: 
it might be $\cL_d$, it might be a type $I_\infty$ factor, and it might be 
something ``in between" (see \cite{DavKatPitts,DavKribsShp,DavPitts1,Read}).
On the other hand, the C*-algebras encountered in 
Proposition \ref{prop:C_star_lines} fall into infinitely many 
$*$-isomorphism classes, while their \wot-closures are all type $I_\infty$ factors. 
These two examples show that taking the \wot-closure of an operator algebra 
is not as innocuous an operation as one might think.

As we have seen in Example \ref{expl:lines}, it can happen that the algebras 
$\alg(S^I_1, \ldots, S^I_d)$ and $\alg(S^J_1, \ldots, S^J_d)$ are isomorphic, 
but their norm closures are non-isomorphic. 
It is plausible that the \wot-closed algebras split further into 
more isomorphism classes, or degenerate to fewer isomorphism classes. 
We will see below that this is not the case.

The proofs of our results follow closely the proofs for the norm-closed case. 
We will give complete details only where the proofs are significantly different.

The main connection to geometry is made via the character space. 
We denote the maximal ideal space of $\cL_X$ by $\cM(\cL_X)$. 
As above, we call elements of $\cM(\cL_X)$ \emph{characters}. 
In general, $\cM(\cL_X)$ can be a very wild topological space, 
and the useful characters are the \wot-continuous ones.

\begin{proposition}\label{prop:charwot}
The \wot-continuous characters of $\cL_X$ can be identified 
with $Z^o(I^X)$.
\end{proposition}

\begin{proof}
For every $\lambda \in Z^o(I^X)$, the vector $\nu_\lambda$ 
is in $\cF_X$. 
Therefore the character $\rho_\lambda$, defined by 
\bes
\rho_\lambda(T) = \lel T \nu_\lambda, \nu_\lambda \rir
\ees
is a \wot-continuous character.

On the other hand, there is a natural quotient from the free semigroup 
algebra $\cL_d$ onto $\cL_X$ that is \wot-continuous. 
Thus, if $\rho$ is a \wot-continuous character of $\cL_X$, 
then it gives rise to \wot-continuous character on $\cL_d$. 
Therefore, using \cite[Theorem 2.3]{DavPitts2}, we find that $\rho$ must 
be equal to the evaluation functional $\rho_\lambda$ at some point 
$\lambda \in \mb{B}_d$. But since $\rho$ restricts to a character of 
$\cA_X$, we must have $\lambda \in Z^o(I^X)$.

The correspondence $\lambda \leftrightarrow \rho_\lambda$ is easily seen 
to be a homeomorphism of $Z^o(I^X)$ onto a subset of $\cM(\cL_X)$.
\end{proof}

Every $\rho \in \cM(\cL_X) \setminus Z^o(I^X)$ restricts to a character of $\cA_X$. 
Thus, the corona $\cM(\cL_X) \setminus Z^o(I^X)$ is the union of fibers 
over $Z(I^X) \setminus Z^o(I^X)$. If $\lambda \in Z(I^X) \setminus Z^o(I^X)$,  
$\rho$ being in the fiber over $\lambda$ means that $\rho(S^X_i) = \lambda_i$, 
or, equivalently, that $\rho|_{\cA_X}$ is equal to evaluation at $\lambda$.

\subsection{The noncommutative case}

\begin{theorem}\label{thm:wotiso}
Let $X$ and $Y$ be subproduct systems. 
Then $\cL_X$ is isometrically isomorphic to $\cL_Y$ if and only if $X \cong Y$.
\end{theorem}

\begin{proof}
One direction follows immediately from the classification 
of the algebras $\cA_X$. 
Indeed, if $X \cong Y$, then there is a unitarily implemented 
isomorphism from $\cA_X$ onto $\cA_Y$, and 
this isomorphism extends to the \wot-closures. 

The proof of the other direction  is similar to the proof
in the normed closed case, with small modifications. 
The proofs of Lemmas \ref{lem:interior} and \ref{lem:analytic} can be 
adjusted to this case to show that for every isometric isomorphism 
$\phi:\cL_X \rightarrow \cL_Y$, the restriction of $\varphi^*$ is a 
biholomorphism of $Z^o(I^Y)$ onto $Z^o(I^X)$. 
Appropriate versions of Theorem \ref{thm:iso_vacuum} and
Proposition \ref{prop:exist_vacuum} are true for the \wot-closed algebras, 
with basically the same proofs.
The result therefore follows just as in the norm-closed case.
\end{proof}

\subsection{The commutative radical case}

 From now on we concentrate on the commutative, radical case.  
In this case, the modifications of the proofs given in the 
norm-closed case are more significant.

\begin{lemma}\label{lem:iso_is_cont}
Let $I$ and $J$ be homogeneous radical ideals in $\mb{C}[z]$. 
Then every homomorphism $\varphi : \cL_I \rightarrow \cL_J$ is bounded. 
\end{lemma}

\begin{proof}
By Proposition \ref{prop:mult}, $\cL_J$ is the multiplier algebra of $\cF_J$. 
Thus, if $f \in \cL_J$ satisfies $f(\lambda) = 0$ for all $\lambda \in Z^o(J)$, 
then $f = 0$. This shows that $\cL_J$ is semi-simple. 
A general result in the theory of commutative Banach algebras says that 
every homomorphism into a semi-simple algebra is automatically continuous 
(see Exercise 3.5.23 in \cite{KRI}). Thus $\varphi$ is bounded. 
\end{proof}

\begin{remark}
The same argument works for the norm closed algebras. 
In Corollary \ref{cor:bounded}, we deduced that every unital homomorphism 
$\varphi : \cA_I \rightarrow \cA_J$ is bounded by using the fact that 
every such homomorphism is given by a composition operator. 
In the case of the \wot-closed algebras, we will use the boundedness 
of homomorphisms to show that they preserve \wot-continuous characters, 
which is crucial to showing that they are implemented by composition. 
\end{remark}

\begin{lemma}\label{lem:interior_isomorphism}
Let $I$ and $J$ be homogeneous radical ideals in $\mb{C}[z]$. 
If $\varphi : \cL_I \rightarrow \cL_J$ is an isomorphism, then $\varphi^*$ 
maps $Z^o(J)$ onto $Z^o(I)$.
\end{lemma}

\begin{proof}
The proof of the lemma uses the notion of \emph{Gleason parts}.  
Let $\cB_I$ be the norm closure of the Gelfand transform 
$\hat{\cL_I} = \{\hat{T} : T \in \cL_I\}$ of $\cL_I$ in $C(\cM(\cL_I))$. 
$\cB_I$ is a function algebra. 
The algebra $\cB_I$ does not really play an important role below. 
It is introduced just for convenience of applying the theory of 
Gleason parts in its usual setting: function algebras. 
For any function algebra, Gleason defined an equivalence relation as follows.

For two characters $\rho_1, \rho_2 \in \cM(\cL_L)$, write $\rho_1 \sim \rho_2$ if 
\bes
\sup\{|f(\rho_1) - f(\rho_2)| : f \in \cB_I , \|f\| \leq 1\} < 2. 
\ees
The relation $\sim$ is an equivalence relation on $\cM(\cL_I)$, 
and the equivalence classes are called \emph{Gleason parts} or 
just \emph{parts} (see \cite{Bear}, Sections 1 and 2). 
Since by the previous lemma $\varphi : \cL_I \rightarrow \cL_J$ 
is a bounded isomorphism, then $\varphi^*$ will map a 
part into a single part.

Since $Z^o(J)$ is a union of disks through the origin, and since $\cM(\cL_J)$ 
is the union of $Z^o(J)$ with the fibers over $Z(J) \setminus Z^o(J)$, 
it follows from classical considerations that $Z^o(J)$ is a part 
(see Example 1, p. 3, \cite{Bear}). 
We need to show that the part $Z^o(J)$ is mapped by $\varphi^*$ 
into the part $Z^o(I)$. 
>From the remarks above, it suffices to show that the vacuum state 
$\rho_0 \in Z^o(J)$ is mapped into $Z^o(I)$.

Assume for the sake of contradiction that $\varphi^* \rho_0 = \rho$, 
where $\rho \in \cM(\cL_I )\setminus Z^o(I)$. By applying a unitary transformation to 
the variables we may assume that $\rho$ is in the fiber over $(1,0,\ldots,0)$.

Put $T = \varphi(S^I_1)$. Let $\lambda$ be any point in $Z^o(J)$, 
and define a function $\hat{T}_\lambda$ on $\mb{D}$ by 
$\hat{T}_\lambda(t) = \rho_{t\lambda}(T)$. 
{}From the discussion preceding Lemma \ref{lem:interior}, $\hat{T}_\lambda$ is analytic.
Now, 
\[
 |\hat{T}_\lambda(t)| = |\rho_{t\lambda}(T)| = 
 |\varphi^* \rho_{t\lambda} (S_1^X)| \leq 1 \quad\text{for } t \in \mb{D} , 
\]
because $\varphi^* \rho_{t\lambda}$ is contractive. 
On the other hand, $\hat{T}_\lambda(0) = \rho(S^X_1) = 1$. 
By the maximum modulus principle, $\hat{T}_\lambda$ is constant $1$ on $\mb{D}$. 
Thus $\hat{T}$, the Gelfand transform of $\varphi(S^X_1)$, is constantly 
equal to $1$ on the disc $\mb{D} \cdot \lambda \subseteq Z^o(J)$. 
Since $\lambda$ was an arbitrary point in $Z^o(J)$, 
it follows that $\hat{T} \equiv 1$ on $Z^o(J)$. 
But the multiplier $T$ and the Gelfand transform $\hat{T}$ are 
the same function on $Z^o(J)$, so $T = 1$. 
This contradicts the fact that $\varphi$ is injective and unit preserving.
This contradiction shows that no $\rho \in \cM(\cL_I) \setminus Z^o(I)$ 
can be equal to $\varphi^*\rho_0$, and this completes the proof.
\end{proof}

\begin{lemma}\label{lem:algiso_biholo}
Let $I$ and $J$ be radical homogeneous ideals in $\mb{C}[z]$. 
Let $\varphi : \cL_I \rightarrow \cL_J$ be an isomorphism. 
Then there exists a holomorphic map $F: \mb{B}_{d} \rightarrow \mb{C}^d$ such that
\bes
F|_{Z^o(J)} = \varphi^* |_{Z^o(J)}.
\ees
The components of $F$ are in $\text{\em Mult}(H^2_d)$.
Moreover, $\varphi$ is given by composition with $F$, that is
\bes
\varphi(f) = f \circ F \quad , \quad f \in \cL_I .
\ees
\end{lemma}

\begin{proof}
The proof is very similar to the proof of Proposition \ref{prop:algiso_biholo}, where 
the change is that we have to restrict attention to $Z^o(J)$ and $Z^o(I)$.
We must use the crucial lemma above, together with Proposition \ref{prop:mult}. 
We omit the details.
\end{proof}

\begin{theorem}\label{thm:wotrad}
Let $I,J \subseteq \mb{C}[z]$ be radical homogeneous ideals. 
\begin{enumerate}
\item Then $\cL_I$ is isometrically isomorphic to $\cL_J$ if and only if 
$\cL_I$ is unitarily equivalent to $\cL_J$, and this happens 
if and only if there is a unitary mapping $Z(I)$ onto $Z(J)$. 

\item If $\cL_I$ is isomorphic to $\cL_J$, then there is an invertible
linear map mapping $Z(I)$ onto $Z(J)$. 
Conversely, if $V(I)$ and $V(J)$ are tractable, and there exists an 
invertible linear map mapping $Z(I)$ onto $Z(J)$, 
then $\cL_I$ is similar to $\cL_J$.

\end{enumerate}
\end{theorem}

\begin{proof}
Part (1) follows from Theorem \ref{thm:wotiso}  and the Nullstellensatz. 

If $V(I)$ and $V(J)$ are tractable, and there is an invertible linear map 
mapping $Z(I)$ onto $Z(J)$, then by Theorem~\ref{thm:linear_induce_similarity} 
$\cA_I$ and $\cA_J$ are similar. 
This extends to a similarity of the \wot-closures $\cL_I$ and $\cL_J$.

Conversely, assume that $\cL_I$ and $\cL_J$ are isomorphic. 
By an analogue of Proposition \ref{prop:exist_vacuum}, 
there exists a vacuum preserving isomorphism between the two algebras. 
By Lemma \ref{lem:algiso_biholo}, there exists a holomorphic map 
$F: \mb{B}_d \rightarrow \mb{C}^d$ sending $Z^o(J)$ onto $Z^o(I)$ 
that fixes the origin. 
By Theorem \ref{thm:linear}, one can assume that $F$ is an invertible linear map.
\end{proof}

A consequence of the geometric classification of the algebras $\cL_I$ 
is that they are as rigid as the varieties that classify them. 
The proof is identical to the proof in the norm-closed case.

\begin{theorem}
Let $I$ and $J$ be two homogeneous radical ideals in 
$\mb{C}[z_1, \ldots, z_d]$, and assume that $V(I)$ is either 
irreducible or a nonlinear hypersurface.  
If $\cL_I$ and $\cL_J$ are isomorphic, then $\cL_I$ and $\cL_J$ are unitarily equivalent. 
If $\varphi: \cL_I \rightarrow \cL_J$ is a vacuum preserving isomorphism, 
then it is unitarily implemented.
\end{theorem}

\subsection{Automatic continuity in the weak-operator and weak-$*$ topologies}
In this section we show that if $I$ and $J$ are radical homogeneous ideals, 
and if $\varphi : \cL_I \rightarrow \cL_J$ is an isomorphism, 
then $\varphi$ is continuous with respect to the weak-operator and 
the weak-$*$ topologies. Note that the above results only imply 
this for vacuum preserving isomorphisms. 

\begin{lemma}\label{lem:topologies_coincide}
Let $I \subseteq \mb{C}[z]$ be a radical homogeneous ideal. 
The weak-$*$ and weak-operator topologies on $\cL_I$ coincide.
\end{lemma}

\begin{proof}
By \cite[Proposition 1.2]{AriasPopescu} (see also \cite[Theorem 5.2]{DavHam}), 
$\cL_I$ has property $\mb{A}_1(1)$. 
This means that for every $\rho$ in the open unit ball of  $(\cL_I)_*$, 
there are $x,y \in \cF_I$ with $\|x\| \|y\| < 1$ such that
\bes
\rho(T) = \lel T x, y \rir \, \, ,  \,\, T \in \cL_I .
\ees
The conclusion immediately follows from this.
\end{proof}

To avoid confusion, in the next two results we will distinguish between a 
function $f$ on $Z^o(I)$ and the multiplication operator $M_f$ on $\cF_I$ 
that it gives rise to.

\begin{lemma}\label{lem:wot_pointwise}
A bounded net $\{M_{f_n}\}$ in $\cL_I$ converges in the weak-operator 
topology to $M_f$ if and only if for all $z \in Z^o(I)$, $f_n(z) \rightarrow f(z)$. 
\end{lemma}

\begin{proof}
If $M_{f_n} \xrightarrow{\wot} M_f$, then for all $z \in Z^o(I)$, 
\bes
 \frac{f_n(z)}{1-\|z\|^2} =   \lel  \nu_z,  \overline{f_n(z)} \nu_z \rir =   
 \lel M_{f_n} \nu_z,  \nu_z \rir  \rightarrow \lel M_f \nu_z,  \nu_z \rir = 
 \frac{f(z)}{1-\|z\|^2}.
\ees
Conversely, suppose $\{M_{f_n}\} \subset \cL_I$ is a bounded net 
such that $\{f_n\}$ converges pointwise to $f$. 
Since $\{M_{f_n}\}$ is bounded, it suffices to show that 
$\lel M_{f_n} \nu_\lambda,  \nu_\mu \rir  \rightarrow \lel M_f \nu_\lambda,  \nu_\mu \rir$ 
for all $\lambda, \mu \in Z^o(I)$, because 
$\spn\{v_\lambda : \lambda \in Z^o(I)\}$ is dense in $\cF_I$. But
\bes
\lel M_{f_n} \nu_\lambda,  \nu_\mu \rir  =  
\frac{f_n(\mu)}{1-\lel \mu, \lambda \rir} \rightarrow \frac{f(\mu)}{1-\lel \mu, \lambda \rir} = 
\lel M_f \nu_\lambda,  \nu_\mu \rir . 
\qedhere
\ees
\end{proof}

\begin{theorem}
Let $I,J \subseteq \mb{C}[z]$ be radical homogeneous ideals. 
If $\varphi: \cL_I \rightarrow \cL_J$ is an isomorphism, then 
$\varphi$ is continuous with respect to the weak-operator and the weak-$*$ topologies.
\end{theorem}

\begin{proof}
By Lemma \ref{lem:topologies_coincide} together with the Krein-\v{S}mulian 
Theorem (Theorem 7, Section V.5, \cite{DunSch}), it is enough to show 
that $\varphi$ is \wot-continuous on bounded sets.

Let $\{M_{f_n}\}$ be a bounded net in $\cL_I$ converging to 
$M_f$ in the weak-operator topology. 
By Lemma \ref{lem:iso_is_cont}, $\{\varphi(M_{f_n})\}$ is a bounded net in $\cL_J$. 
By Lemma \ref{lem:algiso_biholo}, there is some holomorphic $F$ 
such that $\varphi(M_{g}) = M_{g \circ F}$. 
Therefore, by Lemma \ref{lem:wot_pointwise}, it suffices to 
show that $f_n\circ F$ converges pointwise to $f \circ F$. 
But since $f_n$ converges pointwise to $f$ (by the same lemma), this is evident.
\end{proof}

\bibliographystyle{amsplain}

\begin{thebibliography}{9}

\bibitem{AM} J. Agler and J. E. McCarthy, 
\emph{Pick Interpolation and Hilbert Function Spaces}, 
Graduate Studies in Mathematics \textbf{44},  
American Mathematical Society, Providence, 2002.

\bibitem{AriasPopescu} A. Arias and G. Popescu, 
\emph{Noncommutative interpolation and Poisson transforms}, 
Israel J. Math. \textbf{115} (2000), 205--234.

\bibitem{Arv69} W. B. Arveson, 
\emph{Subalgebras of C*-algebras}, 
Acta Math.\ \textbf{123} (1969), 141--224.

\bibitem{Arv72} W. B. Arveson, 
\emph{Subalgebras of C*-algebras II}, 
Acta Math.\ \textbf{128} (1972), 271--308.

\bibitem{Arv98} W. B. Arveson, 
\emph{Subalgebras of C*-algebras III: Multivariable operator theory}, 
Acta Math.\ \textbf{181} (1998), 159--228.

\bibitem{Arv05} W. B. Arveson, 
\emph{$p$-Summable commutators in dimension $d$}, 
J. Operator  Theory \textbf{54} (2005), 101--117.

\bibitem{Arv07} W. B. Arveson, 
\emph{Quotients of standard Hilbert modules}, 
Trans.\ Amer.\ Math.\ Soc.\ \textbf{359}  (2007), 6027--6055.

\bibitem{Bear} H. S. Bear, 
\emph{Lectures on Gleason parts}, 
Lecture Notes in Mathematics, Vol. 121 Springer-Verlag, Berlin-New York 1970.

\bibitem{BhatBhat} B. V. R. Bhat and T. Bhattacharyya, 
\emph{A model theory for $q$-commuting contractive tuples}, 
J. Operator Theory \textbf{47} (2002), 1551--1568.

\bibitem{BhatMukherjee} B. V. R. Bhat and M. Mukherjee, 
\emph{Inclusion systems and amalgamated products products of product systems}, 
preprint (arXiv:0907.0095v2 [math.OA]) to appear in 
Infin.\ Dimens.\ Anal.\ Quantum Probab.\ Relat.\ Top.

\bibitem{BRS} D. Blecher, Z.-J. Ruan and A. Sinclair,
\textit{A characterization of operator algebras},
J. Funct.\ Anal.\ \textbf{89} (1990), 188--201.

\bibitem{Bunce} J. W. Bunce,
\emph{Models for $n$-tuples of noncommuting operators}, 
J. Funct.\ Anal.\  \textbf{57} (1984), 21--30.

\bibitem{ChenGuo} X. M. Chen and K. Y. Guo, 
\emph{Analytic Hilbert Modules}, 
Chapman \& Hall/CRC Research Notes in Math.\ \textbf{433}, 
Chapman \& Hall/CRC, Boca Raton, FL, 2003.

\bibitem{CLO92} D. Cox, J. Little and D. O'shea, 
\emph{Ideals, Varieties, and Algorithms}, 
Springer-Verlag, New York, 1992.

\bibitem{DavHam} K. R. Davidson and R. Hamilton, 
\emph{Nevanlinna-Pick interpolation and factorization of linear functionals}, 
to appear in Integral Equations and Operator Theory, 
arXiv:1008.1090v2 [math.FA].

\bibitem{DavKatPitts} K. R. Davidson, E. Katsoulis and D. Pitts, 
\emph{The structure of free semigroup algebras}, 
J. Reine Angew. Math. \textbf{533} (2001), 99--125.

\bibitem{DavKribsShp} K. R. Davidson, D. W. Kribs and M. E. Shpigel, 
\emph{Isometric dilations of non-commuting finite rank $n$-tuples}, 
Canad. J. Math. \textbf{53} (2001), 506--545. 

\bibitem{DavPittsPick} K.R. Davidson and D. R. Pitts, 
\emph{Nevanlinna-Pick interpolation for non-commutative analytic Toeplitz algebras}, 
Integral Equations Operator Theory \textbf{31} (1998), 321--337. 

\bibitem{DavPitts2} K.R. Davidson and D. R. Pitts, 
\emph{The algebraic structure of non-commutative analytic Toeplitz algebras}, 
Math\ Ann.\ \textbf{311} (1998), 275--303. 

\bibitem{DavPitts1} K.R. Davidson and D. R. Pitts, 
\emph{Invariant subspaces and hyper-reflexivity for free semigroup algebras}, 
Proc.\ London Math.\ Soc.\ (3) \textbf{78}, (1999), 401--430. 

\bibitem{DunSch} N. Dunford and J. T. Schwartz, 
\emph{Linear Operators. I. General Theory}, 
Pure and Applied Mathematics, Vol. 7, Interscience Publishers, 1958.

\bibitem{E10} J. Eschmeier, 
\emph{Essential normality of homogeneous submodules}, 
to appear in Integral Equations and Operator Theory.

\bibitem{Frazho} A. E. Frazho, 
\emph{Complements to models for noncommuting operators}, 
J. Funct.\ Anal.\ \textbf{59} (1984), 445--461.

\bibitem{GR} R. C. Gunning, H. Rossi, 
\emph{Analytic Functions of Several Complex Variables}, 
Prentice-Hall, Inc., Englewood Cliffs, NJ, 1965.

\bibitem{GuoWang} K. Guo and K. Wang, 
\emph{Essentially normal Hilbert modules and $K$-homology}, 
Math.\ Ann.\ \textbf{340} (2008), 907--934.

\bibitem{KRI} R. V. Kadison and J. Ringrose, 
\emph{Fundamentals of the Theory of Operator Algebras, vol. I}, 
Academic Press, New-York, 1982.

\bibitem{Kol} J. Koll\'{a}r, 
\emph{Sharp effective Nullstellensatz},
J. Amer.\ Math.\ Soc.\ \textbf{1} (1988), 963--975.

\bibitem{MS98} P. S. Muhly and B. Solel, 
\emph{Tensor algebras over C*-correspondences: representations, 
dilations, and C*-envelopes}, 
J. Funct.\ Anal.\ \textbf{158} (1998), 389--457.

\bibitem{Mum} D. Mumford, 
\emph{Algebraic Geometry I - Complex Projective Varieties}, 
Springer Verlag, New York, 1976.

\bibitem{Popescu89} G. Popescu, 
\emph{Isometric dilations for infinite sequences of noncommuting operators}, 
Trans.\ Amer.\ Math.\ Soc.\ \textbf{316} (1989), 51--71.

\bibitem{Popescu91} G. Popescu, 
\emph{Von Neumann inequality for $(B(\mathcal{H})^n)_1$}, 
Math.\ Scand.\ \textbf{68} (1991), 292--304.

\bibitem{Popescu06} G. Popescu, 
\emph{Operator theory on noncommutative varieties}, 
Indiana Univ.\ Math.\ J.  \textbf{55}  (2006), 389--442.

\bibitem{Read} C. J. Read, 
\emph{A large weak operator closure for the algebra generated by two isometries}, 
J. Operator Theory \textbf{54} (2005), 305--316.

\bibitem{RudinPoly} W. Rudin, 
\emph{Function Theory in Polydiscs}, 
W.A. Benjamin, New York, 1969.

\bibitem{RudinBall} W. Rudin, 
\emph{Function Theory in the Unit Ball of $\mb{C}^n$}, 
Springer-Verlag, New York, 1980.

\bibitem{Shafarevich} I. R, Shafarevich, 
\emph{Basic Algebraic Geometry 1}, 
Springer-Verlag, New York, 1994.

\bibitem{Sh10} O. M. Shalit, 
\emph{Stable polynomial division and essential normality of graded Hilbert modules}, 
to appear in J. Lond.\ Math.\ Soc., preprint available on arXiv:1003.0502v1 [math.OA]. 

\bibitem{ShalitSolel} O. M. Shalit and B. Solel, 
\emph{Subproduct systems}, 
Doc.\ Math.\ \textbf{14} (2009), 801--868.

\bibitem{SKKT} K.E. Smith, L. Kahanp\"a\"a, P. Kek\"al\"ainen and W. Traves, 
\emph{An Invitation to Algebraic Geometry}, 
Springer, New York, 2000.

\bibitem{Viselter} A. Viselter, 
\emph{Covariant representations of subproduct systems}, 
preprint, 	arXiv:1004.3004v1 [math.OA].

\bibitem{Voic} D. Voiculescu,
\textit{Symmetries of some reduced free product C*-algebras},
Operator algebras and their connections with topology and ergodic
theory (Busteni, 1983),  
Lecture Notes in Math.\ \textbf{1132} (1985), 556--588, 
Springer Verlag, New York, 1985.

\end{thebibliography}

\end{document}